\documentclass[11pt,leqno]{article}

\usepackage[paperwidth=8.5in, textheight=9.5in]{geometry}
\usepackage{amsmath,amsthm,amssymb}
\usepackage{tikz}
\usetikzlibrary{arrows,decorations.pathmorphing,backgrounds,positioning,fit,petri}
\usepackage{enumerate,array,float,verbatim}
\usepackage[labelsep=space]{caption}
\usepackage{chngcntr}
\usepackage{mathrsfs}
\counterwithin{figure}{section}

\theoremstyle{plain}
\newtheorem{theorem}{Theorem}[section]
\newtheorem{corollary}[theorem]{Corollary}
\newtheorem{lemma}[theorem]{Lemma}
\theoremstyle{definition}
\newtheorem{definition}[theorem]{Definition}
\newtheorem{example}[theorem]{Example}
\newtheorem{case}{Case}

\numberwithin{subcase}{case}
\numberwithin{subsubcase}{subcase}
\numberwithin{table}{section}
\numberwithin{equation}{section}

\newcommand{\singleedge}[5]{\begin{tikzpicture}%
[node distance=0.8cm,baseline={(left.base)},%
pre/.style={<-,shorten <=1pt,>=angle 45},%
post/.style={->,shorten >=1pt,>=angle 45}],%
rectangle/.style={inner sep=0pt,minimum size=4mm}];%
\node[rectangle] (left) {$\strut$#1};%
\node[rectangle] (right) [right=of left] {$\strut$#2}%
edge [#4,#5] node[yshift=5pt] {#3} (left);%
\end{tikzpicture}}
\newcommand{\solidedge}[3]{\singleedge{#1}{#2}{#3}{pre}{solid}}
\newcommand{\dashededge}[3]{\singleedge{#1}{#2}{#3}{pre}{dashed}}
\newcommand{\reversedsolidedge}[3]{\singleedge{#1}{#2}{#3}{post}{solid}}
\newcommand{\reverseddashededge}[3]{\singleedge{#1}{#2}{#3}{post}{dashed}}

\newcommand{\doubleedge}[9]{\begin{tikzpicture}%
[node distance=0.8cm,baseline={(middle.base)},%
pre/.style={<-,shorten <=1pt,>=angle 45},%
post/.style={->,shorten >=1pt,>=angle 45}],%
rectangle/.style={inner sep=0pt,minimum size=5mm}];%
\node[rectangle] (left) {$\strut$#1};%
\node[rectangle] (middle) [right=of left] {$\strut$#2}%
edge [#6,#7] node[yshift=5pt] {#3} (left);%
\node[rectangle] (right) [right=of middle] {$\strut$#4}%
edge [#8,#9] node[yshift=5pt] {#5} (middle);%
\end{tikzpicture}}
\newcommand{\solidrightsolidright}[5]{\doubleedge{#1}{#2}{#3}{#4}{#5}{solid}{pre}{solid}{pre}}
\newcommand{\solidrightdashedright}[5]{\doubleedge{#1}{#2}{#3}{#4}{#5}{solid}{pre}{dashed}{pre}}
\newcommand{\dashedrightsolidright}[5]{\doubleedge{#1}{#2}{#3}{#4}{#5}{dashed}{pre}{solid}{pre}}
\newcommand{\dashedrightdashedright}[5]{\doubleedge{#1}{#2}{#3}{#4}{#5}{dashed}{pre}{dashed}{pre}}
\newcommand{\solidrightsolidleft}[5]{\doubleedge{#1}{#2}{#3}{#4}{#5}{solid}{pre}{solid}{post}}
\newcommand{\solidrightdashedleft}[5]{\doubleedge{#1}{#2}{#3}{#4}{#5}{solid}{pre}{dashed}{post}}
\newcommand{\dashedrightsolidleft}[5]{\doubleedge{#1}{#2}{#3}{#4}{#5}{dashed}{pre}{solid}{post}}
\newcommand{\dashedrightdashedleft}[5]{\doubleedge{#1}{#2}{#3}{#4}{#5}{dashed}{pre}{dashed}{post}}
\newcommand{\solidleftsolidright}[5]{\doubleedge{#1}{#2}{#3}{#4}{#5}{solid}{post}{solid}{pre}}
\newcommand{\solidleftdashedright}[5]{\doubleedge{#1}{#2}{#3}{#4}{#5}{solid}{post}{dashed}{pre}}
\newcommand{\dashedleftsolidright}[5]{\doubleedge{#1}{#2}{#3}{#4}{#5}{dashed}{post}{solid}{pre}}
\newcommand{\dashedleftdashedright}[5]{\doubleedge{#1}{#2}{#3}{#4}{#5}{dashed}{post}{dashed}{pre}}
\newcommand{\solidleftsolidleft}[5]{\doubleedge{#1}{#2}{#3}{#4}{#5}{solid}{post}{solid}{post}}
\newcommand{\solidleftdashedleft}[5]{\doubleedge{#1}{#2}{#3}{#4}{#5}{solid}{post}{dashed}{post}}
\newcommand{\dashedleftsolidleft}[5]{\doubleedge{#1}{#2}{#3}{#4}{#5}{dashed}{post}{solid}{post}}
\newcommand{\dashedleftdashedleft}[5]{\doubleedge{#1}{#2}{#3}{#4}{#5}{dashed}{post}{dashed}{post}}

\newcommand{\Gammaarrow}{\Gamma_{\!\to}}
\newcommand{\Gammadir}{{\Gamma_{\text{dir}}}}
\newcommand{\Gammaundir}{{\Gamma_{\text{undir}}}}
\newcommand{\Gammarev}{{\Gamma_{\text{rev}}}}
\newcommand{\myatop}[2]{\genfrac{}{}{0pt}{}{#1}{#2}}
\newcommand{\infinity}{\infty}
\newcommand{\set}[1]{\left\{#1\right\}}
\newcommand{\setof}[2]{\left\{{#1}\mid{#2}\right\}}
\newcommand{\spanof}[1]{\left<#1\right>}
\newcommand{\abs}[1]{\left\vert{#1}\right\vert}
\newcommand{\rationals}{\mathbb Q}%
\newcommand{\complexes}{\mathbb C}

\newcommand{\eval}[3]{\left.{#1}\right\vert_{#2 = #3}}
\newcommand{\closedray}[1]{{\left[{#1},{\infinity}\right)}}
\newcommand{\fieldsymbol}{\mathbb F}

\DeclareMathOperator{\ind}{ind}
\DeclareMathOperator{\sgn}{sgn}
\DeclareMathOperator{\tr}{tr}
\DeclareMathOperator{\linspan}{span}
\DeclareMathOperator{\GL}{GL}
\DeclareMathOperator{\gl}{gl}

\DeclareMathOperator{\vertices}{{\mathscr V}}
\DeclareMathOperator{\edges}{{\mathscr E}}
\DeclareMathOperator{\In}{In}

\begin{document}
\title{A class of representations of Hecke algebras}
\author{Dean Alvis}

\date{}
%\address{Department of Mathematical Sciences \\
%Indiana University South Bend \\
%South Bend, IN, 46634}
%\email{dalvis@iusb.edu}
%\subjclass[2010]{Primary 20C08}

% NOTE: MOVE MAKETITLE BELOW ABSTRACT IN FINAL SOURCE
\maketitle

\begin{abstract}
A type of directed multigraph called a 
{\it $W$-digraph} is introduced to model the
structure of certain representations of Hecke algebras, including those 
constructed by Lusztig and Vogan from involutions
in a Weyl group.
Building on results of Lusztig,
a complete characterization of $W$-digraphs is given 
in terms of subdigraphs for dihedral parabolic subgroups.
In addition, results  are obtained relating 
graph-theoretic properties of $W$-digraphs 
(acyclicity, existence of sources or sinks, connectedness) 
to the structure of the corresponding
$H$-module or its character.
\end{abstract}

%%%%%%%%%%%%%%%%%%%%%%%%%%%%%%%%%%%%%%%%%%%%
%%%%%%%%%%%%%%%%%%%%%%%%%%%%%%%%%%%%%%%%%%%%

\setcounter{section}{-1}
\section{Overview}
Let $W$ be a Weyl group with set of fundamental generators $S$
and length function $\ell$, 
let $u$ be an indeterminate, 
and let $H$ be the Hecke algebra
of $(W,S)$ over $\rationals(u)$.  
(See the next section for a presentation of $H$.)
Put $I = \setof{w \in W}{w^{-1} = w}$.  
In \cite{lusztigvoganinvolutions}, Lusztig and Vogan construct
an $H$-module $M$ with basis $\setof{m_w}{w\in I}$ indexed
by $I$, on which the generator $T_s$ of
 $H$ acts according to the rule
\begin{equation}
T_s m_w
=
\begin{cases}
m_{sws} & \text{ if $sw \ne ws$, $\ell(sw) > \ell(w)$}, \cr
(u^2-1) m_{w} + u^2 m_{sws} & \text{ if $sw \ne ws$, $\ell(sw) < \ell(w)$}, \cr
u m_{w} + (u+1) m_{sw} & \text{ if $sw = ws$, $\ell(sw) > \ell(w)$}, \cr
(u^2-u-1) m_{w} + (u^2-u) m_{sw} & \text{ if $sw = ws$, $\ell(sw) < \ell(w)$}, \cr
\end{cases}
\end{equation}
for $s \in S$, $w \in I$.  
These expressions are given
geometric interpretations in \cite{lusztigvoganinvolutions}:
when $u$ is  replaced by a power $q$ of a prime number,
each coefficient in the expansion of $T_s m_w$ evaluates to
  the number of $\fieldsymbol_q$-rational points in a 
corresponding subset
of a variety constructed from Borel subgroups in an 
algebraic group with Weyl group W.  
(See 1.1-1.6 of \cite{lusztigvoganinvolutions} for the details
of this construction, and  Lusztig's paper \cite{lusztigbarop}
for an extension to arbitrary
Coxeter groups.)

The present work originated in the author's attempt to
visualize the structure of the $H$-module $M$ described 
above.  A directed multigraph $\Gamma$ can be
constructed, with set of vertices $\setof{m_w}{w \in I}$, 
as follows. 
If $w \in I$, $s \in S$, $sw \ne ws$, and
$\ell(w) < \ell(sw)$, 
then a solid
edge \solidedge{$m_w$}{$m_{sws}$}{$s$} is included in $\Gamma$,
while if 
$sw = ws$ and $\ell(w) < \ell(sw)$, then
a dashed edge
\dashededge{$m_{w}$}{$m_{sw}$}{$s$} is included.
The result is an example of what will be called a
$W$-digraph (see Definition~\ref{definition:wdigraph}).
In broad terms, the notion of $W$-digraph is similar
to the notion of $W$-graph introduced by Kazhdan and Lusztig
in \cite{kazhdanlusztig}: both give rise to graph-theoretic 
objects that encode the action of the generators $T_s$, $s \in S$,
on an $H$-module.
There are also combinatorial similarities: 
if a finite dimensional $H$-module $M$ affords both a 
$W$-digraph $\Gamma$ and a $W$-graph $\Psi$, then
the number edges labeled $s \in S$ in $\Gamma$ 
is equal to 
the multiplicity of the eigenvalue $-1$ of $T_s$ on $M$
(see Lemma~\ref{lemma:eigenvalues}(i)), which in turn is equal to the
number of vertices with label including $s$ in $\Psi$.  

On the other hand, there are significant differences between
the notions of $W$-digraph and $W$-graph, including
the obvious structural differences: 
a $W$-digraph is directed rather than undirected, 
can have two different
types of edges (corresponding to commutation relations
in the motivating example above)
rather than one type, and has generators labeling edges
rather than scalars.
The encodings of the actions of generators for $W$-digraphs
and $W$-graphs are necessarily different.  
Further,  the class of modules 
afforded by $W$-digraphs need not coincide with the class
of modules afforded by $W$-graphs.
When $(W,S)$ is finite and $S \ne \emptyset$, 
not all $H$-modules are afforded by $W$-digraphs, 
whereas every $H$-module is afforded 
by a $W$-graph over a suitable field of scalars 
(Gyoja, \cite{gyojawgraph}). 
Specifically, when
$S \ne \emptyset$,   the 
sign representation $T_s \mapsto -1$ is not afforded by 
a $W$-digraph,
(see Theorem~\ref{theorem:linearcharmults}(i)), but is afforded
by the $W$-graph with a single vertex labeled $S$.
In the other direction, an example can be given of
an infinite $(W,S)$ and corresponding $H$-module
that is afforded by a $W$-digraph but  
is not afforded by a $W$-graph
(see Theorem~\ref{theorem:wgraphacyclic} 
and Example~\ref{example:affinea2}).  

%%%%%%%%%%%%%%%%%%%%%%%%%%%%%%%%%%%%%%%%%%%%
%%%%%%%%%%%%%%%%%%%%%%%%%%%%%%%%%%%%%%%%%%%%

\section{Statement of Results}

The Coxeter system $(W,S)$ has a presentation of the form 
\[
W =
\left<
s \in S
\, \mid \,
(rs)^{n(r,s)}=e \text{ for $r,s\in S$, $n(r,s)<\infinity$}
\right> ,
\]
where $n(s,s) = 1$ and $1<n(r,s)=n(s,r)\le\infinity$ for $r,s\in S$, $r \ne s$.
The Hecke algebra $H$ has basis
$\setof{T_w}{w \in W}$ satisfying
\begin{equation}
\label{eq:leftmultbys}
T_s T_w = 
\begin{cases}
T_{sw} & \text{if $\ell(sw) > \ell(w)$,} \\
u^2 T_{sw} + (u^2-1)T_w & \text{if $\ell(sw) < \ell(w)$} \\
\end{cases}
\end{equation}
for $s \in S$.
As a $\rationals(u)$-algebra,
$H$ has generators $\setof{T_s}{s\in S}$ satisfying the relations 
\begin{subequations}
\begin{align}
\label{eq:quadraticrelation}
& (T_s - u^2)(T_s + 1) 
= 0
\qquad &&
\text{if $s\in S$,} \\
\label{eq:dihedralrelation}
& \overbrace{ T_s  T_t \cdots }^{n} 
=   
\overbrace{ T_t T_s \cdots  }^{n}
\qquad && 
\text{if $s,t \in S$, $1 < n = n(s,t)<\infinity$}
\end{align}
\end{subequations}
(where the factors in the products of \eqref{eq:dihedralrelation} 
are alternately
$T_s$ and $T_t$).
Moreover, 
\begin{equation*}
T_x T_y = T_{xy}
\qquad
\text{if $\ell(xy) = \ell(x)+\ell(y)$.}
\end{equation*}

%%%%%%%%%%%%%%%%%%%%%%%%%%%%%%%%%%%%%%%%%%%%

\begin{definition}
Let $S$ be a set.  
Let $\Gamma = (\vertices,\edges)$ 
be a directed multigraph with set of vertices
$\vertices = \vertices(\Gamma)$ and set of edges
$\edges = \edges(\Gamma)$ such that each edge is
either solid or dashed and is labeled by an element of $S$,
that is, has one of the forms
\[
\solidedge{$\alpha$}{$\beta$}{$s$} 
\qquad
\text{or}
\qquad
\dashededge{$\alpha$}{$\beta$}{$s$} 
\]
with $\alpha, \beta \in \vertices$, $s \in S$.
 Then $\Gamma$ is an
{\it $S$-labeled digraph} if
$\Gamma$ has no loops and, for all $s \in S$, 
every vertex of $\Gamma$ occurs in exactly one edge labeled $s$.
\end{definition}

Examples of $S$-labeled digraphs appear in 
Figures~\ref{fig:ex2}--\ref{fig:ex3}.
\begin{figure}[ht]
\begin{minipage}{0.5\textwidth}
\centering
{%
%%%%%%%%%%%%%%%%%%%%%%%%%%%%%%%%%%%%%%%%%%%%
\begin{tikzpicture}
[node distance=0.8cm,%
pre/.style={<-,shorten <=1pt,>=angle 45},%
post/.style={->,shorten >=1pt,>=angle 45}];%
rectangle/.style={inner sep=0pt,minimum size=5mm}];%
\node[rectangle] (left)									{$\gamma_1$};
\node[rectangle] (top1)	[above=of left,xshift=10mm,yshift=-2mm]	{$\gamma_2$}
	edge		[pre,dashed]		node[yshift=2mm,xshift=-2mm]	{$s$}		(left);
\node[rectangle] (top2)	[right=of top1]						{$\gamma_3$}
	edge		[pre]			node[yshift=2mm]				{$t$}		(top1);
\node[rectangle] (bottom1)	[below=of left,xshift=10mm,yshift=2mm]	{$\gamma_6$}
	edge		[pre]		node[yshift=-2mm,xshift=-2mm]			{$t$}		(left);
\node[rectangle] (bottom2)	[right=of bottom1]				{$\gamma_5$}
	edge		[pre]			node[yshift=-2mm]				{$s$}		(bottom1);
\node[rectangle] (right) [below=of top2,xshift=10mm,yshift=2.5mm]	{$\gamma_4$}
	edge		[pre]		node[xshift=2mm,yshift=2mm]			{$s$}		(top2)
	edge		[pre,dashed]	node[xshift=2mm,yshift=-2mm]		{$t$}		(bottom2);
\end{tikzpicture}
%%%%%%%%%%%%%%%%%%%%%%%%%%%%%%%%%%%%%%%%%%%%
%%%%%%%%%%%%%%%%%%%%%%%%%%%%%%%%%%%%%%%%%%%%
}%
\caption{An $\set{s,t}$-labeled digraph.}
\label{fig:ex2}
\end{minipage}%
\begin{minipage}{0.5\textwidth}
\centering
{%
%%%%%%%%%%%%%%%%%%%%%%%%%%%%%%%%%%%%%%%%%%%%
\begin{tikzpicture}
[node distance=0.6cm,%
pre/.style={<-,shorten <=1pt,>=angle 45},%
post/.style={->,shorten >=1pt,>=angle 45}];%
rectangle/.style={inner sep=0pt,minimum size=5mm}];%
\node[rectangle]	(middle)				{$\gamma_4$};
\node[rectangle]	(top)	[above=of middle,yshift=6mm]	{$\gamma_1$}
	edge		[pre]		node[xshift=-2mm,yshift=-2mm]	{$t$}		(middle);
\node[rectangle]	(left)	[below=of middle,xshift=-20mm,yshift=3mm]	{$\gamma_3$}
	edge		[post]	node[xshift=-2.5mm] {$s$}  (top)
	edge		[post,dashed]		node[yshift=2.5mm]	{$r$}		(middle);
\node[rectangle]	(right) [below=of middle,xshift=20mm,yshift=3mm]	{$\gamma_2$}
	edge		[post,dashed]	node[xshift=2.5mm]	{$r$}		(top)
	edge		[pre]		node[yshift=2.5mm]	{$s$} (middle)
	edge		[pre]	node[yshift=-2.5mm]	{$t$}		(left);
\end{tikzpicture}
%%%%%%%%%%%%%%%%%%%%%%%%%%%%%%%%%%%%%%%%%%%%
%%%%%%%%%%%%%%%%%%%%%%%%%%%%%%%%%%%%%%%%%%%%
}%
\caption{An $\set{r,s,t}$-labeled digraph.}
\label{fig:ex3}
\end{minipage}%
\end{figure}

%%%%%%%%%%%%%%%%%%%%%%%%%%%%%%%%%%%%%%%%%%%%
 
 Let $\Gamma$ be an $S$-labeled digraph.  
Let $M(\Gamma)$ be a vector space over 
$\rationals(u)$ with basis $\vertices(\Gamma)$, and
let $\gl(M(\Gamma))$ be the 
$\rationals(u)$-algebra of all linear operators on $M(\Gamma)$.
For each $s\in S$, define 
$\tau_s \in \gl(M(\Gamma)$ as follows: if $\alpha\in \vertices(\Gamma)$, then
\begin{equation}
\label{eq:taudefinition}
\tau_s(\alpha)
=
\begin{cases}
\beta & \quad
\text{if }\solidedge{$\alpha$}{$\beta$}{$s$} \in \edges(\Gamma), \\
(u^2-1)\alpha + u^2\beta & \quad
\text{if }\reversedsolidedge{$\alpha$}{$\beta$}{$s$} \in \edges(\Gamma), \\
u \alpha + (u+1)\beta & \quad
\text{if }\dashededge{$\alpha$}{$\beta$}{$s$} \in \edges(\Gamma),  \\
(u^2-u-1) \alpha + (u^2-u)\beta & \quad
\text{if }\reverseddashededge{$\alpha$}{$\beta$}{$s$} \in \edges(\Gamma).
\end{cases}
\end{equation}

%%%%%%%%%%%%%%%%%%%%%%%%%%%%%%%%%%%%%%%%%%%%

\begin{definition}
\label{definition:wdigraph}
An $S$-labeled digraph 
$\Gamma$ is a {\it $W$-digraph} if
the mapping $T_s \mapsto \tau_s$ extends to a 
representation of $H$, that is, a homomorphism of 
$\rationals(u)$-algebras $\rho : H \rightarrow \gl \left(M(\Gamma) \right)$. 
\end{definition}
\noindent

%%%%%%%%%%%%%%%%%%%%%%%%%%%%%%%%%%%%%%%%%%%%

Let $J \subseteq S$, so $(W_J, J)$ is a Coxeter system
with $W_J = \spanof{J}$ the associated parabolic subgroup of $W$. 
For $\Gamma$ an $S$-labeled digraph,  denote by  $\Gamma_J$
the subdigraph with the same set of vertices
 obtained from $\Gamma$ by removing 
all edges labeled by elements of $S \setminus J$. Thus $\Gamma_J$
is a $J$-labeled digraph.  
If $\Gamma$ is a $W$-digraph, then clearly $\Gamma_J$ 
is a $W_J$-digraph.  
Conversely, because of the presentation
\eqref{eq:quadraticrelation}, 
\eqref{eq:dihedralrelation} it is also
clear that $\Gamma$ is a $W$-digraph
if $\Gamma_J$ is a $W_J$-digraph whenever 
$J \subseteq S$, $\abs{J} \le 2$.  Note also that $\Gamma$ is a 
$W$-digraph if and only if each connected component of $\Gamma$ is
a $W$-digraph.

In Figures~\ref{fig:main1}--\ref{fig:main8} several $J$-labeled 
 digraphs are given with
$J = \set{s,t}$.
These multigraphs have $2m$ vertices, with $m \ge 2$
except for Figures~\ref{fig:main7}--\ref{fig:main8}.  
Also,  $s^\prime = s$ if $m$ is even,
$s^\prime = t$ if $m$ is odd,   $t^\prime$ is defined by
$\set{s^\prime,t^\prime}=\set{s,t}$,
and any edge not shown has one 
of the forms
\solidedge{$\alpha_{2j}$}{$\alpha_{2j+1}$}{$s$}, 
\solidedge{$\alpha_{2j-1}$}{$\alpha_{2j}$}{$t$}, 
\solidedge{$\beta_{2j-1}$}{$\beta_{2j}$}{$s$}, or 
\solidedge{$\beta_{2j}$}{$\beta_{2j+1}$}{$t$}.

\begin{figure}[ht]
\begin{minipage}{0.33\textwidth}
\centering
{%
%%%%%%%%%%%%%%%%%%%%%%%%%%%%%%%%%%%%%%%%%%%%
\begin{tikzpicture}
[node distance=0.4cm,%
pre/.style={<-,shorten <=1pt,>=angle 45},%
post/.style={->,shorten >=1pt,>=angle 45}];%
rectangle/.style={inner sep=0pt,minimum size=5mm}];%
\node[rectangle] (top) {${\strut}\alpha_0$};
\node[rectangle] (left1) [below=of top,xshift=-10mm] {${\strut}\alpha_1$}
   edge [pre] node[xshift=-2mm,yshift=2mm] {$s$} (top);
\node[rectangle] (right1) [below=of top,xshift=12mm] {${\strut}\beta_1$}  
   edge [pre] node[xshift=2mm,yshift=2mm] {$t$} (top);
\node[rectangle] (left2) [below=of left1] {${\strut}\alpha_2$}
   edge [pre] node[xshift=-2.5mm] {$t$} (left1);
\node[rectangle] (right2) [below=of right1] {${\strut}\beta_2$}
   edge [pre] node[xshift=2.5mm] {$s$} (right1);
\node[rectangle] (leftmid) [below=of left2,yshift=8mm] {${\strut}\vdots$};
\node[rectangle] (left3) [below=of leftmid,yshift=8mm] {${\strut}\alpha_{m-2}$};
\node[rectangle] (rightmid) [below=of right2,yshift=8mm] {${\strut}\vdots$};
\node[rectangle] (right3) [below=of rightmid,yshift=8mm] {${\strut}\beta_{m-2}$};
\node[rectangle] (left4) [below=of left3] {${\strut}\alpha_{m-1}$}
   edge [pre] node[xshift=-2.5mm] {$s^\prime$} (left3);
\node[rectangle] (right4) [below=of right3] {${\strut}\beta_{m-1}$}
   edge [pre] node[xshift=2.5mm] {$t^\prime$} (right3);
\node[rectangle] (bottom) [below=of right4,xshift=-10mm] {${\strut}\beta_m$}
   edge [pre] node[xshift=-2mm,yshift=-2mm] {$t^\prime$} (left4)
   edge [pre] node[xshift=1.5mm,yshift=-2mm] {$s^\prime$} (right4);
\end{tikzpicture}
%%%%%%%%%%%%%%%%%%%%%%%%%%%%%%%%%%%%%%%%%%%%%
%%%%%%%%%%%%%%%%%%%%%%%%%%%%%%%%%%%%%%%%%%%%%
}%
\caption{}
\label{fig:main1}
\end{minipage}%
\begin{minipage}{0.33\textwidth}
\centering
{%
%%%%%%%%%%%%%%%%%%%%%%%%%%%%%%%%%%%%%%%%%%%%
\begin{tikzpicture}
[node distance=0.4cm,%
pre/.style={<-,shorten <=1pt,>=angle 45},%
post/.style={->,shorten >=1pt,>=angle 45}];%
rectangle/.style={inner sep=0pt,minimum size=5mm}];%
\node[rectangle] (top) {${\strut}\alpha_0$};
\node[rectangle] (left1) [below=of top,xshift=-10mm] {${\strut}\alpha_1$}
   edge [pre,dashed] node[xshift=-2mm,yshift=2mm] {$s$} (top);
\node[rectangle] (right1) [below=of top,xshift=12mm] {${\strut}\beta_1$}  
   edge [pre] node[xshift=2mm,yshift=2mm] {$t$} (top);
\node[rectangle] (left2) [below=of left1] {${\strut}\alpha_2$}
   edge [pre] node[xshift=-2.5mm] {$t$} (left1);
\node[rectangle] (right2) [below=of right1] {${\strut}\beta_2$}
   edge [pre] node[xshift=2.5mm] {$s$} (right1);
\node[rectangle] (leftmid) [below=of left2,yshift=8mm] {${\strut}\vdots$};
\node[rectangle] (left3) [below=of leftmid,yshift=8mm] {${\strut}\alpha_{m-2}$};
\node[rectangle] (rightmid) [below=of right2,yshift=8mm] {${\strut}\vdots$};
\node[rectangle] (right3) [below=of rightmid,yshift=8mm] {${\strut}\beta_{m-2}$};
\node[rectangle] (left4) [below=of left3] {${\strut}\alpha_{m-1}$}
   edge [pre] node[xshift=-2.5mm] {$s^\prime$} (left3);
\node[rectangle] (right4) [below=of right3] {${\strut}\beta_{m-1}$}
   edge [pre] node[xshift=2.5mm] {$t^\prime$} (right3);
\node[rectangle] (bottom) [below=of right4,xshift=-10mm] {${\strut}\beta_m$}
   edge [pre] node[xshift=-2mm,yshift=-2mm] {$t^\prime$} (left4)
   edge [pre,dashed] node[xshift=1.5mm,yshift=-2mm] {$s^\prime$} (right4);
\end{tikzpicture}
%%%%%%%%%%%%%%%%%%%%%%%%%%%%%%%%%%%%%%%%%%%%%
%%%%%%%%%%%%%%%%%%%%%%%%%%%%%%%%%%%%%%%%%%%%%
}%
\caption{}
\label{fig:main2}
\end{minipage}%
\begin{minipage}{0.33\textwidth}
\centering
{%
%%%%%%%%%%%%%%%%%%%%%%%%%%%%%%%%%%%%%%%%%%%%
\begin{tikzpicture}
[node distance=0.4cm,%
pre/.style={<-,shorten <=1pt,>=angle 45},%
post/.style={->,shorten >=1pt,>=angle 45}];%
rectangle/.style={inner sep=0pt,minimum size=5mm}];%
\node[rectangle] (top) {${\strut}\alpha_0$};
\node[rectangle] (left1) [below=of top,xshift=-10mm] {${\strut}\alpha_1$}
   edge [pre] node[xshift=-2mm,yshift=2mm] {$s$} (top);
\node[rectangle] (right1) [below=of top,xshift=12mm] {${\strut}\beta_1$}  
   edge [pre,dashed] node[xshift=2mm,yshift=2mm] {$t$} (top);
\node[rectangle] (left2) [below=of left1] {${\strut}\alpha_2$}
   edge [pre] node[xshift=-2.5mm] {$t$} (left1);
\node[rectangle] (right2) [below=of right1] {${\strut}\beta_2$}
   edge [pre] node[xshift=2.5mm] {$s$} (right1);
\node[rectangle] (leftmid) [below=of left2,yshift=8mm] {${\strut}\vdots$};
\node[rectangle] (left3) [below=of leftmid,yshift=8mm] {${\strut}\alpha_{m-2}$};
\node[rectangle] (rightmid) [below=of right2,yshift=8mm] {${\strut}\vdots$};
\node[rectangle] (right3) [below=of rightmid,yshift=8mm] {${\strut}\beta_{m-2}$};
\node[rectangle] (left4) [below=of left3] {${\strut}\alpha_{m-1}$}
   edge [pre] node[xshift=-2.5mm] {$s^\prime$} (left3);
\node[rectangle] (right4) [below=of right3] {${\strut}\beta_{m-1}$}
   edge [pre] node[xshift=2.5mm] {$t^\prime$} (right3);
\node[rectangle] (bottom) [below=of right4,xshift=-10mm] {${\strut}\beta_m$}
   edge [pre,dashed] node[xshift=-2mm,yshift=-2mm] {$t^\prime$} (left4)
   edge [pre] node[xshift=1.5mm,yshift=-2mm] {$s^\prime$} (right4);
\end{tikzpicture}
%%%%%%%%%%%%%%%%%%%%%%%%%%%%%%%%%%%%%%%%%%%%%
%%%%%%%%%%%%%%%%%%%%%%%%%%%%%%%%%%%%%%%%%%%%%
}%
\caption{}
\label{fig:main3}
\end{minipage}%
\end{figure}

%%%%%%%%%%%%%%%%%%%%%%%%%%%%%%%%%%%%%%%%%%%%

Two $S$-labeled digraphs $\Gamma = (\vertices, \edges)$ and 
$\Gamma^\prime = (\vertices^\prime,\edges^\prime)$ 
are
{\it isomorphic} if there is some 
bijection $\varphi : \vertices \rightarrow \vertices^\prime$ such that
for all $\alpha,\beta \in \vertices$ and $s\in S$, 
$\!\solidedge{$\alpha$}{$\beta$}{$s$}\in\edges$ 
if and only if
$\!\solidedge{$\varphi(\alpha)$}{$\varphi(\beta)$}{$s$} \in \edges^\prime$
and
$\!\dashededge{$\alpha$}{$\beta$}{$s$}\in\edges$
if and only if
$\!\dashededge{$\varphi(\alpha)$}{$\varphi(\beta)$}{$s$} \in \edges^\prime$.

%%%%%%%%%%%%%%%%%%%%%%%%%%%%%%%%%%%%%%%%%%%%

\begin{figure}[ht]
\begin{minipage}{0.33\textwidth}
\centering
{%
%%%%%%%%%%%%%%%%%%%%%%%%%%%%%%%%%%%%%%%%%%%%
\begin{tikzpicture}
[node distance=0.4cm,%
pre/.style={<-,shorten <=1pt,>=angle 45},%
post/.style={->,shorten >=1pt,>=angle 45}];%
rectangle/.style={inner sep=0pt,minimum size=5mm}];%
\node[rectangle] (top) {${\strut}\alpha_0$};
\node[rectangle] (left1) [below=of top,xshift=-10mm] {${\strut}\alpha_1$}
   edge [pre,dashed] node[xshift=-2mm,yshift=2mm] {$s$} (top);
\node[rectangle] (right1) [below=of top,xshift=12mm] {${\strut}\beta_1$}  
   edge [pre,dashed] node[xshift=2mm,yshift=2mm] {$t$} (top);
\node[rectangle] (left2) [below=of left1] {${\strut}\alpha_2$}
   edge [pre] node[xshift=-2.5mm] {$t$} (left1);
\node[rectangle] (right2) [below=of right1] {${\strut}\beta_2$}
   edge [pre] node[xshift=2.5mm] {$s$} (right1);
\node[rectangle] (leftmid) [below=of left2,yshift=8mm] {${\strut}\vdots$};
\node[rectangle] (left3) [below=of leftmid,yshift=8mm] {${\strut}\alpha_{m-2}$};
\node[rectangle] (rightmid) [below=of right2,yshift=8mm] {${\strut}\vdots$};
\node[rectangle] (right3) [below=of rightmid,yshift=8mm] {${\strut}\beta_{m-2}$};
\node[rectangle] (left4) [below=of left3] {${\strut}\alpha_{m-1}$}
   edge [pre] node[xshift=-2.5mm] {$s^\prime$} (left3);
\node[rectangle] (right4) [below=of right3] {${\strut}\beta_{m-1}$}
   edge [pre] node[xshift=2.5mm] {$t^\prime$} (right3);
\node[rectangle] (bottom) [below=of right4,xshift=-10mm] {${\strut}\beta_m$}
   edge [pre] node[xshift=-2mm,yshift=-2mm] {$t^\prime$} (left4)
   edge [pre] node[xshift=1.5mm,yshift=-2mm] {$s^\prime$} (right4);
\end{tikzpicture}
%%%%%%%%%%%%%%%%%%%%%%%%%%%%%%%%%%%%%%%%%%%%%
%%%%%%%%%%%%%%%%%%%%%%%%%%%%%%%%%%%%%%%%%%%%%
}%
\caption{}
\label{fig:main4}
\end{minipage}%
\begin{minipage}{0.33\textwidth}
\centering
{%
%%%%%%%%%%%%%%%%%%%%%%%%%%%%%%%%%%%%%%%%%%%%
\begin{tikzpicture}
[node distance=0.4cm,%
pre/.style={<-,shorten <=1pt,>=angle 45},%
post/.style={->,shorten >=1pt,>=angle 45}];%
rectangle/.style={inner sep=0pt,minimum size=5mm}];%
\node[rectangle] (top) {${\strut}\alpha_0$};
\node[rectangle] (left1) [below=of top,xshift=-10mm] {${\strut}\alpha_1$}
   edge [pre] node[xshift=-2mm,yshift=2mm] {$s$} (top);
\node[rectangle] (right1) [below=of top,xshift=12mm] {${\strut}\beta_1$}  
   edge [pre] node[xshift=2mm,yshift=2mm] {$t$} (top);
\node[rectangle] (left2) [below=of left1] {${\strut}\alpha_2$}
   edge [pre] node[xshift=-2.5mm] {$t$} (left1);
\node[rectangle] (right2) [below=of right1] {${\strut}\beta_2$}
   edge [pre] node[xshift=2.5mm] {$s$} (right1);
\node[rectangle] (leftmid) [below=of left2,yshift=8mm] {${\strut}\vdots$};
\node[rectangle] (left3) [below=of leftmid,yshift=8mm] {${\strut}\alpha_{m-2}$};
\node[rectangle] (rightmid) [below=of right2,yshift=8mm] {${\strut}\vdots$};
\node[rectangle] (right3) [below=of rightmid,yshift=8mm] {${\strut}\beta_{m-2}$};
\node[rectangle] (left4) [below=of left3] {${\strut}\alpha_{m-1}$}
   edge [pre] node[xshift=-2.5mm] {$s^\prime$} (left3);
\node[rectangle] (right4) [below=of right3] {${\strut}\beta_{m-1}$}
   edge [pre] node[xshift=2.5mm] {$t^\prime$} (right3);
\node[rectangle] (bottom) [below=of right4,xshift=-10mm] {${\strut}\beta_m$}
   edge [pre,dashed] node[xshift=-2mm,yshift=-2mm] {$t^\prime$} (left4)
   edge [pre,dashed] node[xshift=1.5mm,yshift=-2mm] {$s^\prime$} (right4);
\end{tikzpicture}
%%%%%%%%%%%%%%%%%%%%%%%%%%%%%%%%%%%%%%%%%%%%%
%%%%%%%%%%%%%%%%%%%%%%%%%%%%%%%%%%%%%%%%%%%%%
}%
\caption{}
\label{fig:main5}
\end{minipage}%
\begin{minipage}{0.33\textwidth}
\centering
{%
%%%%%%%%%%%%%%%%%%%%%%%%%%%%%%%%%%%%%%%%%%%%
\begin{tikzpicture}
[node distance=0.4cm,%
pre/.style={<-,shorten <=1pt,>=angle 45},%
post/.style={->,shorten >=1pt,>=angle 45}];%
rectangle/.style={inner sep=0pt,minimum size=5mm}];%
\node[rectangle] (top) {${\strut}\alpha_0$};
\node[rectangle] (left1) [below=of top,xshift=-10mm] {${\strut}\alpha_1$}
   edge [pre,dashed] node[xshift=-2mm,yshift=2mm] {$s$} (top);
\node[rectangle] (right1) [below=of top,xshift=12mm] {${\strut}\beta_1$}  
   edge [pre,dashed] node[xshift=2mm,yshift=2mm] {$t$} (top);
\node[rectangle] (left2) [below=of left1] {${\strut}\alpha_2$}
   edge [pre] node[xshift=-2.5mm] {$t$} (left1);
\node[rectangle] (right2) [below=of right1] {${\strut}\beta_2$}
   edge [pre] node[xshift=2.5mm] {$s$} (right1);
\node[rectangle] (leftmid) [below=of left2,yshift=8mm] {${\strut}\vdots$};
\node[rectangle] (left3) [below=of leftmid,yshift=8mm] {${\strut}\alpha_{m-2}$};
\node[rectangle] (rightmid) [below=of right2,yshift=8mm] {${\strut}\vdots$};
\node[rectangle] (right3) [below=of rightmid,yshift=8mm] {${\strut}\beta_{m-2}$};
\node[rectangle] (left4) [below=of left3] {${\strut}\alpha_{m-1}$}
   edge [pre] node[xshift=-2.5mm] {$s^\prime$} (left3);
\node[rectangle] (right4) [below=of right3] {${\strut}\beta_{m-1}$}
   edge [pre] node[xshift=2.5mm] {$t^\prime$} (right3);
\node[rectangle] (bottom) [below=of right4,xshift=-10mm] {${\strut}\beta_m$}
   edge [pre,dashed] node[xshift=-2mm,yshift=-2mm] {$t^\prime$} (left4)
   edge [pre,dashed] node[xshift=1.5mm,yshift=-2mm] {$s^\prime$} (right4);
\end{tikzpicture}
%%%%%%%%%%%%%%%%%%%%%%%%%%%%%%%%%%%%%%%%%%%%%
%%%%%%%%%%%%%%%%%%%%%%%%%%%%%%%%%%%%%%%%%%%%%
}%
\caption{}
\label{fig:main6}
\end{minipage}%
\end{figure}
\begin{figure}[ht]
\begin{minipage}{0.15\textwidth}
\qquad
\end{minipage}%
\begin{minipage}{0.35\textwidth}
\centering
%%%%%%%%%%%%%%%%%%%%%%%%%%%%%%%%%%%%%%%%%%%%
{\begin{tikzpicture}%
[node distance=1.2cm,%
pre/.style={<-,shorten <=1pt,>=angle 45},%
post/.style={->,shorten >=1pt,>=angle 45}];%
rectangle/.style={inner sep=0pt,minimum size=5mm}];%
\node[rectangle] (left) {$\alpha_0$};%
\node[rectangle] (right) [right=of left] {$\beta_1$}%
edge [pre,bend right=30] node[yshift=5pt] {$s$} (left)%
edge [pre,bend left=30] node[yshift=-6pt] {$t$} (left);%
\end{tikzpicture}}
%%%%%%%%%%%%%%%%%%%%%%%%%%%%%%%%%%%%%%%%%%%%
%%%%%%%%%%%%%%%%%%%%%%%%%%%%%%%%%%%%%%%%%%%%
\caption{}
\label{fig:main7}
\end{minipage}%
\begin{minipage}{0.35\textwidth}
\centering
%%%%%%%%%%%%%%%%%%%%%%%%%%%%%%%%%%%%%%%%%%%%
{\begin{tikzpicture}%
[node distance=1.2cm,%
pre/.style={<-,shorten <=1pt,>=angle 45},%
post/.style={->,shorten >=1pt,>=angle 45}];%
rectangle/.style={inner sep=0pt,minimum size=5mm}];%
\node[rectangle] (left) {$\alpha_0$};%
\node[rectangle] (right) [right=of left] {$\beta_1$}%
edge [pre,dashed,bend right=30] node[yshift=5pt] {$s$} (left)%
edge [pre,dashed,bend left=30] node[yshift=-6pt] {$t$} (left);%
\end{tikzpicture}}
%%%%%%%%%%%%%%%%%%%%%%%%%%%%%%%%%%%%%%%%%%%%
%%%%%%%%%%%%%%%%%%%%%%%%%%%%%%%%%%%%%%%%%%%%
\caption{}
\label{fig:main8}
\end{minipage}%
\end{figure}

%%%%%%%%%%%%%%%%%%%%%%%%%%%%%%%%%%%%%%%%%%%%

\begin{theorem}
\label{theorem:main}
Let $(W,S)$ be a Coxeter system.
The following are equivalent.
\begin{enumerate}[{\upshape(a)}]
\item
$\Gamma$ is a $W$-digraph.
\item
$\Gamma$ is an $S$-labeled digraph such that
for all $s, t \in S$
with $1 < n = n(s,t)<\infinity$, 
each connected component of $\Gamma_J$, 
$J = \set{s,t}$, 
is isomorphic to one of
the $J$-labeled digraphs in 
Figures~\ref{fig:main1}--\ref{fig:main8}, with
\begin{enumerate}[{\upshape(i)}]
\item
$m \ge 2$ and $m$ a divisor of $n$ in  
 Figure~\ref{fig:main1}, 
Figure~\ref{fig:main2}, or Figure~\ref{fig:main3}, 
\item
$m\ge 2$ and $2m-1$ a divisor of $n$
in Figure~\ref{fig:main4} or
Figure~\ref{fig:main5},
\item
 $m\ge 2$ and $2m-2$ a divisor of $n$
 in Figure~\ref{fig:main6}, 
\item
$m = 1$ and $n \ge 2$ arbitrary 
in Figure~\ref{fig:main7} or Figure~\ref{fig:main8}.
\end{enumerate}
\end{enumerate}
\end{theorem} 

%%%%%%%%%%%%%%%%%%%%%%%%%%%%%%%%%%%%%%%%%%%%

If $\Gamma$ is an $S$-labeled digraph, 
let $\Gammarev$ be the $S$-labled digraph obtained
by reversing the direction of all edges of $\Gamma$
while keeping their types and labels. 
For example, 
if $\Gamma$ is as in Figure~\ref{fig:main4}, then 
$\Gammarev$ is isomorphic to the digraph in Figure~\ref{fig:main5}. 
We can assume $M(\Gamma)$ and $M(\Gammarev)$ are the same
as  vector spaces over $\rationals(u)$ (since
$\vertices(\Gamma) = \vertices(\Gammarev)$), but the
endomorphisms of $M(\Gamma)$ and $M(\Gammarev)$ 
corresponding to an element of $S$ are different. 

\begin{corollary}
\label{corollary:reversed}
If $\Gamma$ is an $S$-labeled digraph, then 
$\Gamma$ is a $W$-digraph if and only if
$\Gammarev$ is a $W$-digraph.
\end{corollary}

%%%%%%%%%%%%%%%%%%%%%%%%%%%%%%%%%%%%%%%%%%%%

For $\Gamma$ an $S$-labeled digraph, let 
$\Gammaarrow$ be the $S$-labeled digraph 
obtained from $\Gamma$ by replacing any
dashed edge \!\dashededge{$\alpha$}{$\beta$}{$s$}\! 
by the corresponding solid edge 
 \!\solidedge{$\alpha$}{$\beta$}{$s$}\!.
 Let $\Gammadir$ be the directed multigraph obtained
 by removing all labels from $\Gammaarrow$.
Let $\Gammaundir$ be the (undirected) multigraph
obtained from $\Gammadir$ by replacing each
directed edge \!\solidedge{$\alpha$}{$\beta$}{}\! by an
undirected edge $\alpha$ ----- $\beta$.
For example, with $\Gamma$ 
as in
Figure~\ref{fig:ex2},  the
\begin{figure}[ht]
\centering
\begin{minipage}{0.32\textwidth}
\centering
%%%%%%%%%%%%%%%%%%%%%%%%%%%%%%%%%%%%%%%%%%%%
\begin{tikzpicture}
[node distance=0.65cm,%
pre/.style={<-,shorten <=1pt,>=angle 45},%
post/.style={->,shorten >=1pt,>=angle 45}];%
rectangle/.style={inner sep=0pt,minimum size=5mm}];%
\node[rectangle] (left)									{$\gamma_1$};
\node[rectangle] (top1)	[above=of left,xshift=10mm,yshift=-2mm]	{$\gamma_2$}
	edge		[pre]		node[yshift=2mm,xshift=-2mm]	{$s$}		(left);
\node[rectangle] (top2)	[right=of top1]						{$\gamma_3$}
	edge		[pre]			node[yshift=2mm]				{$t$}		(top1);
\node[rectangle] (bottom1)	[below=of left,xshift=10mm,yshift=2mm]	{$\gamma_6$}
	edge		[pre]		node[yshift=-2mm,xshift=-2mm]			{$t$}		(left);
\node[rectangle] (bottom2)	[right=of bottom1]				{$\gamma_5$}
	edge		[pre]			node[yshift=-2mm]				{$s$}		(bottom1);
\node[rectangle] (right) [below=of top2,xshift=10mm,yshift=2.5mm]	{$\gamma_4$}
	edge		[pre]		node[xshift=2mm,yshift=2mm]			{$s$}		(top2)
	edge		[pre]		node[xshift=2mm,yshift=-2mm]		{$t$}		(bottom2);
\end{tikzpicture}
%%%%%%%%%%%%%%%%%%%%%%%%%%%%%%%%%%%%%%%%%%%%
%%%%%%%%%%%%%%%%%%%%%%%%%%%%%%%%%%%%%%%%%%%%
\end{minipage}%
\begin{minipage}{0.32\textwidth}%
\centering
%%%%%%%%%%%%%%%%%%%%%%%%%%%%%%%%%%%%%%%%%%%%
\begin{tikzpicture}
[node distance=0.65cm,%
pre/.style={<-,shorten <=1pt,>=angle 45},%
post/.style={->,shorten >=1pt,>=angle 45}];%
rectangle/.style={inner sep=0pt,minimum size=5mm}];%
\node[rectangle]	(left)							{$\gamma_1$};
\node[rectangle]	(top1)	[above=of left,xshift=10mm,yshift=-2mm]	{$\gamma_2$}
	edge		[pre]			(left);
\node[rectangle]	(top2)	[right=of top1]			{$\gamma_3$}
	edge		[pre]			(top1);
\node[rectangle]	(bottom1)	[below=of left,xshift=10mm,yshift=2mm]	{$\gamma_6$}
	edge		[pre]			(left);
\node[rectangle]	(bottom2)	[right=of bottom1]			{$\gamma_5$}
	edge		[pre]		(bottom1);
\node[rectangle]	(right)	[below=of top2,xshift=10mm,yshift=2.5mm]	{$\gamma_4$}
	edge		[pre]	(top2)
	edge		[pre]		(bottom2);
\end{tikzpicture}
%%%%%%%%%%%%%%%%%%%%%%%%%%%%%%%%%%%%%%%%%%%%
%%%%%%%%%%%%%%%%%%%%%%%%%%%%%%%%%%%%%%%%%%%%
\end{minipage}%
\begin{minipage}{0.32\textwidth}%
%%%%%%%%%%%%%%%%%%%%%%%%%%%%%%%%%%%%%%%%%%%%
\begin{tikzpicture}
[node distance=0.65cm,%
rectangle/.style={thick,inner sep=0pt,minimum size=5mm}];
\node[rectangle]	(left)							{$\gamma_1$};
\node[rectangle]	(top1)	[above=of left,xshift=10mm,yshift=-2mm]	{$\gamma_2$}
	edge [thick]				(left);
\node[rectangle]	(top2)	[right=of top1]			{$\gamma_3$}
	edge	 [thick]		(top1);
\node[rectangle]	(bottom1)	[below=of left,xshift=10mm,yshift=2mm]	{$\gamma_6$}
	edge	 [thick]		(left);
\node[rectangle]	(bottom2)	[right=of bottom1]			{$\gamma_5$}
	edge	 [thick]		(bottom1);
\node[rectangle]	(right)	[below=of top2,xshift=10mm,yshift=2.5mm]	{$\gamma_4$}
	edge	 [thick]		(top2)
	edge	 [thick]		(bottom2);
\end{tikzpicture}
%%%%%%%%%%%%%%%%%%%%%%%%%%%%%%%%%%%%%%%%%%%%
%%%%%%%%%%%%%%%%%%%%%%%%%%%%%%%%%%%%%%%%%%%%
\end{minipage}%
\caption{$\Gammaarrow$, $\Gammadir$, 
and $\Gammaundir$ for $\Gamma$ as in Figure~\ref{fig:ex2}}
\label{fig:dirundir}
\end{figure}
associated graphs $\Gammaarrow$, 
$\Gammadir$, and $\Gammaundir$ 
are given in Figure~\ref{fig:dirundir}.
We say a vertex $\alpha$ of $\Gamma$ is
a {\it source} ({\it sink}) of $\Gamma$ if 
$\alpha$ is a source (sink, respectively)
in $\Gammadir$.  
We consider an empty path to be a directed circuit
 in any directed multigraph, and define
  $\Gamma$ to be {\it acyclic} if
$\Gammadir$ is acyclic, that is, if there is no nonempty
directed circuit in $\Gammadir$.  
Also, $\Gamma$ is {\it connected} if
$\Gammaundir$ is connected.  

%%%%%%%%%%%%%%%%%%%%%%%%%%%%%%%%%%%%%%%%%%%%

\begin{theorem}
\label{theorem:acyclic}
If $n(s,t) < \infinity$ for all $s, t \in S$ and
$\Gamma$ is a connected $W$-digraph, then the
following hold.
\begin{enumerate}[{\upshape(i)}]
\item
$\Gamma$ can have at most one source and
at most one sink.
\item
If $\Gamma$ has either a source or a sink, then
$\Gamma$ is acyclic.
\item
If $(W,S)$ is finite, then $\Gamma$ has both
a source and a sink, and so is acyclic. 
\end{enumerate}
\end{theorem}

%%%%%%%%%%%%%%%%%%%%%%%%%%%%%%%%%%%%%%%%%%%%

\begin{corollary}
\label{corollary:numsourcessinks}
If $(W,S)$ is finite, then any $W$-digraph is acyclic.
Further, 
the number of sources (or sinks) in a finite $W$-digraph 
is equal to
the number of its connected components. 
\end{corollary}

%%%%%%%%%%%%%%%%%%%%%%%%%%%%%%%%%%%%%%%%%%%%

Let $\ind$ and $\sgn$ be the linear characters of $H$ determined
by $\ind(T_w) = u_w = u^{2 \ell(w)}$ and
$\sgn(T_w) = \varepsilon_w = (-1)^{\ell(w)}$ for $w \in W$, respectively.
For $\lambda$ a linear character of $H$ and $M$ an $H$-module, 
put $M_\lambda = \setof{v \in M}{h v = \lambda(h) v \text{ for } h \in H}$.

\begin{theorem}
\label{theorem:linearcharmults}
If $\Gamma$ is a $W$-digraph and $\vertices(\Gamma)$ is finite,
then the following hold.
\begin{enumerate}[{\upshape(i)}]
\item
The number of connected components of $\Gamma$ is equal to 
$\dim M(\Gamma)_{\ind}$.
\item
If $n(s,t) < \infinity$ for all $s, t \in S$, then
the number of acyclic connected components of $\Gamma$ is equal
to $\dim M(\Gamma)_{\sgn}$.
\end{enumerate}
\end{theorem}

%%%%%%%%%%%%%%%%%%%%%%%%%%%%%%%%%%%%%%%%%%%%

\begin{theorem}
\label{theorem:index}
If $(W,S)$ is finite, $J \subseteq S$, and 
$\Gamma$ is a connected $W$-digraph, then 
$\Gamma_J$ has at most $\abs{W : W_J}$ connected components.
\end{theorem}

Taking $J = \emptyset$ gives the following.

\begin{corollary}
\label{corollary:bound}
If $(W,S)$ is finite and  $\Gamma$ is
a connected $W$-digraph, then $\abs{\vertices(\Gamma)} \le \abs{W}$. 
\end{corollary}

The bound in Corollary~\ref{corollary:bound} is always attained: 
see Example~\ref{example:regular}.

%%%%%%%%%%%%%%%%%%%%%%%%%%%%%%%%%%%%%%%%%%%%

\begin{theorem}
\label{theorem:equallengths}
Assume $n(s,t)<\infinity$ for $s,t \in S$ and
$\Gamma$ is a connected $W$-digraph with a source or sink.  
Then for $\alpha,\beta \in \vertices(\Gamma)$, any two directed
paths from $\alpha$ to $\beta$ in $\Gamma$ have the same 
number of edges.
\end{theorem}

%%%%%%%%%%%%%%%%%%%%%%%%%%%%%%%%%%%%%%%%%%%%

If $\Gamma$ is a $W$-digraph and $\vertices(\Gamma)$
is finite, so $M(\Gamma)$ is finite dimensional, let
$\chi_{\Gamma}$ be the character of $H$ afforded by 
$M(\Gamma)$.

\begin{theorem}
\label{theorem:charactervalues}
If $\Gamma$ is a $W$-digraph and
$\vertices(\Gamma)$ is finite, then the following hold.
\begin{enumerate}[{\upshape(i)}]
\item
If $\sigma$ is the automorphism of $\rationals(u)$ 
determined by ${}^\sigma u = -1/u$, then
$\chi_{\Gammarev}(T_w)
=
{}^\sigma \chi_{\Gamma}(T_{w^{-1}}^{-1})
$
for $w \in W$.
\item
If $n(s,t) < \infinity$ for $s, t \in S$ and 
$\Gamma$ is a acyclic, then
$\chi_{\Gammarev}(T_w)
=
\varepsilon_w u_w \chi_{\Gamma} ( T_w^{-1})
$
for $w \in W$.
\end{enumerate}
\end{theorem}

%%%%%%%%%%%%%%%%%%%%%%%%%%%%%%%%%%%%%%%%%%%%

In the case of an affine Weyl group,
the following holds.    

\begin{theorem}
If $(W_J,J)$ is finite for proper subsets $J$ of $S$, 
$\Gamma$ is a finite, connected $W$-digraph,
and $M(\Gamma)$ affords a $W$-graph 
(as defined in \cite{kazhdanlusztig}) over $\rationals$, 
then $\Gamma$ is acyclic.
\label{theorem:wgraphacyclic}
\end{theorem}

%%%%%%%%%%%%%%%%%%%%%%%%%%%%%%%%%%%%%%%%%%%%

The organization of this paper is as follows.
Section~\ref{section:preliminary} contains preliminary results,
and Section~\ref{section:mainproof} contains a proof
of Theorem~\ref{theorem:main} and related results.
Section~\ref{section:acyclic} contains proofs 
of Theorems~\ref{theorem:acyclic},
\ref{theorem:linearcharmults}, 
\ref{theorem:index}, and \ref{theorem:equallengths}.
Section~\ref{section:charactervalues}
contains a proof of Theorem~\ref{theorem:charactervalues}
and related results. 
Section~\ref{section:wgraphacyclic} contains a proof
of Theorem~\ref{theorem:wgraphacyclic}, and the last section
has  additional examples.

%%%%%%%%%%%%%%%%%%%%%%%%%%%%%%%%%%%%%%%%%%%%
%%%%%%%%%%%%%%%%%%%%%%%%%%%%%%%%%%%%%%%%%%%%

\section{Preliminary results}
\label{section:preliminary}

Assume that $(W,S)$ is a Coxeter system and let
 $\Gamma$ be an $S$-labeled digraph.
 Throughout this and later sections, the notation $x \le y$ is used to indicate 
the usual Bruhat order on $W$ relative to $S$ when $x, y \in W$.  
For any
$s\in S$, 
we have
\begin{equation}
\label{eq:tauquadraticrelation}
(\tau_s - u^2)(\tau_s + 1) = 0
\qquad
\text{in $\gl(M)$,}
\end{equation}
where $\tau_s$ is as in \eqref{eq:taudefinition} and
$M = M(\Gamma)$
(see \cite{lusztigbarop}, 2.3).  
Indeed, suppose  
$\alpha$ is connected to $\beta$ by an edge of
$\Gamma$ labeled $s$.
Exchanging $\alpha$, $\beta$ if necessary, we can assume
this edge is directed from $\alpha$ to $\beta$.
By \eqref{eq:taudefinition},  $\tau_s$ leaves invariant the subspace
with basis $\set{\alpha,\beta}$,  
and the matrix of $\tau_s$ acting on this subspace 
with respect to this basis is
\[
\begin{pmatrix}
0 & u^2 \\
1 & u^2-1 \\
\end{pmatrix}
\qquad
\text{or}
\qquad
\begin{pmatrix}
u  & u^2 - u \\
u + 1 & u^2 - u - 1 \\
\end{pmatrix}
\]
according to whether
$\solidedge{$\alpha$}{$\beta$}{$s$} \in \edges(\Gamma)$ or
$\dashededge{$\alpha$}{$\beta$}{$s$} \in \edges(\Gamma)$. 
In either case, the eigenvalues 
are $u^2$ and $-1$, and thus
\eqref{eq:tauquadraticrelation} holds. 
Hence $\Gamma$ is a $W$-digraph if and
only if
\begin{equation}
\label{eq:taudihedralrelation}
\overbrace{ \tau_s \tau_t \cdots}^{n}
=
\overbrace{\tau_t \tau_s \cdots }^{n}
\qquad
\text{whenever $s,t\in S$, $1 < n(s,t) < \infinity$.}
\end{equation}

Define $T^{\circ}_s \in H$ by
\begin{equation*}
T_s^{\circ} = (u+1)^{-1} (T_s - u).
\end{equation*}  
(This element is denoted ${\overset{\circ}T}_s$ in
\cite{lusztigbarop}, 2.2.)
By \eqref{eq:quadraticrelation}, both
$T_s$ and $T^{\circ}_s$ are units in $H$, with
inverses given by
\[
T_s^{-1} = u^{-2}(T_s - (u^2-1)),
\qquad
(T^{\circ}_s)^{-1} = (u^2-u)^{-1}(T_s -(u^2 - u - 1)).
\]
The terminology used in the next definition
will be justified by the remarks after 
Lemma~\ref{lemma:submodule}.

%%%%%%%%%%%%%%%%%%%%%%%%%%%%%%%%%%%%%%%%%%%%

\begin{definition}
Let $M$ be an $H$-module.  Then a subset $X$ of $M$
{\it supports a $W$-digraph} if $X$ is linearly independent
over $\rationals(u)$ and, 
for each  $\alpha\in X$ and $s \in S$,
\[
X \cap \set{T_s\alpha, T_s^{-1}\alpha, T^{\circ}_s\alpha,
  (T^{\circ}_s)^{-1}\alpha} \ne \emptyset .
\] 
\end{definition}

%%%%%%%%%%%%%%%%%%%%%%%%%%%%%%%%%%%%%%%%%%%%

\begin{lemma}
\label{lemma:submodule}
If $M$ is an $H$-module and $X\subseteq M$ 
supports a $W$-digraph, then the following hold.
\begin{enumerate}[{\upshape(i)}]
\item 
If $s \in S$ and $\alpha \in X$, then $\alpha$, 
$T_s\alpha$, $T_s^{-1}\alpha$, $T^{\circ}_s\alpha$,
  $(T^{\circ}_s)^{-1}\alpha$ are distinct and
$X$ contains a unique element of 
$\set{T_s\alpha, T_s^{-1}\alpha, T^{\circ}_s\alpha,
  (T^{\circ}_s)^{-1}\alpha}$.
\item
The subspace of $M$ spanned by $X$ is an $H$-submodule of $M$.
\end{enumerate}
\end{lemma}

\begin{proof}
Suppose $s\in S$ and $\alpha\in X$.  
Put $Y = \set{T_s\alpha, T_s^{-1}\alpha, T^{\circ}_s\alpha,
(T^{\circ}_s)^{-1}\alpha}$.  
By \eqref{eq:quadraticrelation}, there are unique $\gamma,\delta \in M$
such that
\[
\alpha = \gamma + \delta,
\qquad
T_s \gamma = - \gamma,
\qquad
T_s \delta = u^2 \delta.
\]
Thus
\begin{align*}
T_s \alpha 
& =  -\gamma + u^2 \delta, &
T_s^{-1} \alpha 
& =  -\gamma + \frac{1}{u^2}\delta, \\
T^{\circ}_s \alpha
& =  -\gamma + \frac{u^2-u}{u+1}\delta, &
(T^{\circ}_s)^{-1} \alpha
& =  -\gamma + \frac{u+1}{u^2-u}\delta .
\end{align*}
Since $X$ is linearly independent 
and $X$ contains $\alpha$ and at least one
element of $Y$, it follows that
$\gamma$, $\delta$ are linearly independent over $\rationals(u)$.
Therefore $\alpha$, $T_s\alpha$, $T_s^{-1}\alpha$, $T^{\circ}_s\alpha$,
$(T^{\circ}_s)^{-1}\alpha$ are distinct.  Also, since $\alpha$, 
$T_s\alpha$, $T_s^{-1}\alpha$, $T^{\circ}_s\alpha$,
$(T^{\circ}_s)^{-1}\alpha$ are all in
$\linspan \set{\gamma,\delta}$, $X$ can contain at most one
element of
$Y$.  Thus (i) holds.
Further, since $X$ contains two elements of
$\linspan \set{\gamma,\delta}$, 
$\linspan X$ contains $\linspan \set{\gamma,\delta}$ by
dimension, and
thus $T_s \alpha \in \linspan X$.  Since $\alpha\in X$
was arbitrary, we have
$T_s \linspan  X \subseteq \linspan X$.  Thus
$\linspan  X$ is an $H$-submodule of $M$ since 
$s \in S$ was arbitrary, so (ii) holds.
\end{proof}

If $M$ is an $H$-module and $X\subseteq M$ supports
a $W$-digraph, then we construct a directed multigraph 
 $\Gamma$, as follows.  If $\alpha, \beta\in X$ and
$s\in S$, then 
\begin{equation*}
\begin{split}
\solidedge{$\alpha$}{$\beta$}{$s$} \text{ is an edge of $\Gamma$ if } &
\beta = T_s \alpha , \\
\dashededge{$\alpha$}{$\beta$}{$s$} \text{ is an edge of $\Gamma$ if }  &
\beta = T^{\circ}_s \alpha .
\end{split}
\end{equation*}
Then $\Gamma$ is a well-defined $S$-labeled digraph by
Lemma~\ref{lemma:submodule}.  
Moreover, from the definition of $T^{\circ}_s$, it is easily
checked that $H$ acts on $M_0 = \linspan  X$ according to
\[
T_s \alpha = \tau_s(\alpha),
\]
where $\tau_s$ is as in \eqref{eq:taudefinition}. 
Therefore $\Gamma$ is indeed a $W$-digraph
with associated $H$-module $M_0$.

%%%%%%%%%%%%%%%%%%%%%%%%%%%%%%%%%%%%%%%%%%%%

\begin{lemma}
\label{lemma:partitions}
Suppose $X$ is a linearly independent subset of an $H$-module $M$.
Then $X$ supports a $W$-digraph if and only if for each $s\in S$,
there exists a partition $P_s$ of $X$ such that, for all $U \in P_s$, 
there are $\alpha, \beta \in U$ such that $\alpha \ne \beta$, 
$U = \set{\alpha,\beta}$, and either $T_s \alpha = \beta$ or
$T^{\circ}_s \alpha = \beta$.
\end{lemma}

\begin{proof}
First suppose $X$ supports a $W$-digraph.
Let $s \in S$.  For $\lambda \in X$, define $U_\lambda = \set{\lambda,\mu}$
where
\[
X \cap \set{T_s\lambda, T_s^{-1}\lambda, T^{\circ}_s\lambda,
(T^{\circ}_s)^{-1}\lambda}
= \set{\mu}.
\]
Then $\lambda \in X \cap \set{T_s\mu, T_s^{-1}\mu, T^{\circ}_s\mu,
(T^{\circ}_s)^{-1}\mu}$, and so $U_\lambda = U_\mu$.  
By Lemma~\ref{lemma:submodule},
$P_s = \setof{U_\lambda}{\lambda \in X}$ is a partition of 
$X$ satisfying
the conditions above: if $U = U_\lambda$ and
$\mu = T_s \lambda$ or $\mu = T^{\circ}_s \lambda$, then take
$\alpha = \lambda$, $\beta = \mu$, and otherwise take
$\alpha = \mu$, $\beta = \lambda$.

Conversely, suppose for each $s\in S$, a partition $P_s$ satisfying
the conditions above exists.  Let $\gamma \in X$.  There is some
$\delta \in X$ such that $U = \set{\gamma,\delta} \in P_s$.  For this
$\delta$ we either have $T_s^{\pm 1} \gamma = \delta$ or
$(T^{\circ}_s)^{\pm1} \gamma = \delta$.  Thus $X$ supports a
$W$-digraph.
\end{proof}

%%%%%%%%%%%%%%%%%%%%%%%%%%%%%%%%%%%%%%%%%%%%

\begin{lemma}
\label{lemma:eigenvalues}
Suppose $M$ is an $H$-module with
basis $X$ supporting a
 $W$-digraph $\Gamma$, 
 $v = \sum_{\gamma\in X} \lambda_\gamma \gamma \in M$, 
 and  $s \in S$.
Then the following hold.
\begin{enumerate}[{\upshape(i)}]
\item
$T_s v = u^2 v$ if and only if
$\lambda_\beta = \lambda_\alpha$ whenever 
\solidedge{$\alpha$}{$\beta$}{$s$}
or
\dashededge{$\alpha$}{$\beta$}{$s$}
is an edge of $\Gamma$.
\item
$T_s v = - v$ if and only if
\[
\lambda_\beta
=
\begin{cases}
- u^{-2} \lambda_\alpha & 
\text{whenever $\solidedge{$\alpha$}{$\beta$}{$s$} \in \edges(\Gamma)$,} \\
- (u+1)(u^2-u)^{-1} \lambda_\alpha &
\text{whenever $\dashededge{$\alpha$}{$\beta$}{$s$} \in \edges(\Gamma)$.}
\end{cases}
\]
\end{enumerate}
\end{lemma}

\begin{proof}
With $P_s$ as in Lemma~\ref{lemma:partitions}, 
$M$ is the direct sum of 
the subspaces $\linspan \set{\alpha,\beta}$,
 $\set{\alpha,\beta}\in P_s$.  Also, $T_s$ leaves each
 such subspace invariant, with eigenvalues $u^2$
 and $-1$.
  Hence it suffices to show that for
 $\set{\alpha,\beta}\in P_s$,  
 $\linspan \set{\alpha,\beta}$
 has a basis
 consisting of eigenvectors for  $T_s$ of the form 
 $\lambda_\alpha \alpha + \lambda_\beta \beta$
 with coefficients satisfying the relations of (i) and (ii).  
If \solidedge{$\alpha$}{$\beta$}{$s$} is an edge of
$\Gamma$, then
\[
T_s (\alpha + \beta)
= \beta + \left(u^2 \alpha + (u^2-1)\beta\right) 
= u^2 (\alpha+\beta)
\]
and
\[
T_s \left(\alpha - u^{-2} \beta\right)
= \beta - u^{-2} \left(u^2 \alpha + (u^2-1)\beta \right)
= -\left(\alpha - u^{-2} \beta\right) ,
\]
so the  basis
$\set{\alpha + \beta, \alpha - u^{-2} \beta}$
has the desired property.
On the other hand, if 
\dashededge{$\alpha$}{$\beta$}{$s$} is an 
edge of $\Gamma$, then
\begin{equation*}
\begin{split}
T_s (\alpha + \beta)
& = 
\left((u \alpha + (u+1)\beta\right) +
   \left((u^2-u)\alpha + (u^2-u-1)\beta \right) \\
& = 
u^2 ( \alpha + \beta )
\end{split}
\end{equation*}
and
\begin{equation*}
\begin{split}
& T_s \left(\alpha - (u+1)(u^2-u)^{-1} \beta\right) \\
& \qquad
=
\left((u \alpha + (u+1)\beta\right) \\
& \qquad  \qquad
  {}  - (u+1)(u^2-u)^{-1} \left((u^2-u)\alpha + (u^2-u-1)\beta\right) \\
& \qquad
=
- \left(\alpha - (u+1)(u^2-u)^{-1} \beta\right) , 
\end{split}
\end{equation*}
so $\set{\alpha + \beta, \alpha - (u+1)(u^2-u)^{-1} \beta}$
is an appropriate basis.
\end{proof}

%%%%%%%%%%%%%%%%%%%%%%%%%%%%%%%%%%%%%%%%%%%%
 
For the remainder of this section we assume
$J = \set{s,t} \subseteq S$, 
$1 < n = n(s,t) < \infinity$.
For $0 \le k \le n$, define elements $s_k$, $t_k$ of
$W_J$ by
\[
s_{k} =
\overbrace{ \cdots s t s }^{k},
\qquad
t_{k} =
\overbrace{ \cdots t s t }^{k},
\]
with $k$ factors in each product, alternately 
$s$ and $t$. 
For example, $s_{0} = e = t_{0}$, and
$s_{n} = w_0 = t_{n}$ is the longest element of $W_J$.
Define elements $\sigma_k$ of $H_J$ as follows:
\[
\sigma_k = \sum_{\myatop{w \in W_J}{\ell(w)=k}} T_w.
\]
Thus $\sigma_0 = T_e$, $\sigma_n = T_{w_0}$, and
$\sigma_k = T_{s_{k}}+ T_{t_{k}}$ for $0 < k < n$.

%%%%%%%%%%%%%%%%%%%%%%%%%%%%%%%%%%%%%%%%%%%%

\begin{lemma}
\label{lemma:dihedral}
Suppose 
$a_0 = \sigma_k + \sum_{\myatop{w \in W_J}{\ell(w) < k}} \gamma_w T_w \in H_J$,
where $0 \le k < n$ and $\gamma_w \in \rationals(u)$ for $w\in W_J$.  
Suppose further that 
for $0 \le j \le n-k$, 
$\overline{S}_j \in \set{T_s, T^{\circ}_s}$ and
$\overline{T}_j \in \set{T_t,T_t^\circ}$.
Put $b_0=a_0$, and 
define $a_1,\dots,a_{n-k},b_1,\dots,b_{n-k}$ by
\[
a_{j+1} 
=
\begin{cases}
\overline{S}_j a_j & \text{ if $j$ is even,} \cr
\overline{T}_j a_j & \text{ if $j$ is odd,} \cr
\end{cases}
\quad\text{and}\quad
b_{j+1} 
=
\begin{cases}
\overline{T}_j b_j & \text{ if $j$ is even,} \cr
\overline{S}_j b_j & \text{ if $j$ is odd} \cr
\end{cases}
\]
for $0 \le j < n-k$.
Then
$X=\set{a_0,a_1,\dots,a_{n-k-1}, b_1,b_2,\dots,b_{n-k}}$ is linearly 
independent.
Moreover, if $a_{n-k}=b_{n-k}$,  then 
$X$ supports a $W_J$-digraph and
$L=\linspan X$ is a left ideal of $H_J$.
\end{lemma}

\begin{proof}
If $1 \le j \le n-k$ 
and $a_j$ is expressed as a linear combination
of $\setof{T_w}{w\in W_J}$, then
the unique $w \in W$ of maximal length such that $T_w$ appears 
with nonzero coefficient is given by
\[
w = 
\begin{cases}
s_j s_k = s_{j+k} & \text{if $k$ is even,} \\
s_j t_k = t_{j+k} & \text{if $k$ is odd.} \\
\end{cases}
\]
Similarly, if $1\le j \le n-k$, then the unique $w \in W_J$ of maximal 
length such that $T_w$ appears with nonzero coefficient in $b_j$
is given by 
\[
w =
\begin{cases}
t_j t_k = t_{j+k} & \text{if $k$ is even,} \\
t_j s_k = s_{j+k} & \text{if $k$ is odd.} \\
\end{cases}
\]
Thus $X$ is linearly independent.

Suppose $a_{n-k} = b_{n-k}$.
If $n-k$ is even, then the partitions
\begin{equation*}
\begin{split}
P_s 
& = 
 \set{\set{a_0,a_1},\set{b_1,b_2},\dots,\set{b_{n-k-1},b_{n-k}}}, \\
 P_t 
 & =  \set{\set{b_0,b_1},\set{a_1,a_2},\dots,\set{a_{n-k-1},a_{n-k}}}
\end{split}
\end{equation*}
satisfy the conditions of Lemma~\ref{lemma:partitions}.  
On the other hand, if $n-k$ is odd, then the partitions
\begin{equation*}
\begin{split}
 P_s 
 & =
 \set{\set{a_0,a_1},\set{b_1,b_2},\dots,\set{a_{n-k-1},a_{n-k}}}, \\
 P_t 
 & = \set{\set{b_0,b_1},\set{a_1,a_2},\dots,\set{b_{n-k-1},b_{n-k}}}
\end{split}
\end{equation*}
satisfy the conditions of Lemma~\ref{lemma:partitions}.
Thus $X$ supports a $W_J$-digraph,
and $L = \linspan  X$ is a left ideal of $H_J$ by 
Lemma~\ref{lemma:submodule}(ii). 
\end{proof}

%%%%%%%%%%%%%%%%%%%%%%%%%%%%%%%%%%%%%%%%%%%%

For $d \ge 0$, define a polynomial $p_d(u) \in \rationals[u]$ as follows: 
$p_0(u) = 1$, and for $d > 0$, 
\[
p_d(u)
=
1 + 2 \sum_{i=1}^{d-1} (-u^2)^{i} + (-u^2)^{d}.
\]
Thus $p_1(u) = 1-u^2$, $p_2(u) = 1-2u^2 + u^4$, $p_3(u) = 1-2u^2+2u^4-u^6$.
Let $y \in W_J$.  
A straightforward induction argument based on {2.0.b} and {2.0.c} of
\cite{kazhdanlusztig} shows 
\begin{equation}
\label{eq:inverse}
u^{2\ell(y)} T_{y^{-1}}^{-1}
=
T_y
+ \sum_{x < y} p_{\ell(y)-\ell(x)}(u) T_x.
\end{equation}

For $0 \le j \le n$, we define elements $\widetilde{\varphi}_j$, 
$\widetilde{\eta}_j$, $\widetilde{\gamma}_j$, $\widetilde{\delta}_j$ of
$H_J$, 
as follows:
\begin{equation*}
\left\{
\begin{aligned}
{\widetilde{\varphi}}_j
& =  \sum_{i=0}^{j} p_{j-i}(u) \sigma_i \\
& = 
\sigma_j + (1-u^2) \sigma_{j-1} + 
(1-2u^2+u^4) \sigma_{j-2} + \cdots \\
& \qquad\qquad\qquad
{} + (1 - 2u^2 +2 u^4 \mp \cdots + 2(-u^2)^{j-1} + (-u^2)^j) \sigma_0, \\
{\widetilde{\eta}}_j 
& = 
{\widetilde{\varphi}}_j + u {\widetilde{\varphi}}_{j-1} + u^2{\widetilde{\varphi}}_{j-2} + \cdots + u^{j} {\widetilde{\varphi}}_0 , \\
{\widetilde{\gamma}}_j 
& = 
{\widetilde{\varphi}}_j - u {\widetilde{\varphi}}_{j-1} + u^2{\widetilde{\varphi}}_{j-2} \mp \cdots + (-u)^{j} {\widetilde{\varphi}}_0 , \\
{\widetilde{\delta}}_j 
& = 
\frac{1}{2}({\widetilde{\eta}}_j + {\widetilde{\gamma}}_j) .
\end{aligned}
\right.
\end{equation*}
These elements are used in describing the constructions
of Lusztig in the next section.

If $0 < j < n$, then by \eqref{eq:inverse} we have
\begin{equation}
\label{eq:varphi}
\begin{split}
{\widetilde{\varphi}}_j
& = 
T_{s_j} + T_{t_j} + (1-u^2) \sigma_{j-1} +
   (1-2u^2+u^4) \sigma_{j-2} + \cdots \\
 & \qquad
 {} + (1 - 2u^2 \pm \cdots + 2 (-u^2)^{j-1}+ (-u^2)^{j}) \sigma_{0} \\
& = 
u^{2j} T_{s_{j}^{-1}}^{-1} + T_{t_{j}} 
= 
u^{2j} T_{t_{j}^{-1}}^{-1} + T_{s_{j}} .
\end{split}
\end{equation}

%%%%%%%%%%%%%%%%%%%%%%%%%%%%%%%%%%%%%%%%%%%%

\begin{lemma}
\label{lemma:varphi}
If $0 < j \le k$ and $j+k \le n$, then
\[
T_{s_k^{-1}} {\widetilde{\varphi}}_j 
= 
u^{2j} T_{s_{k-j}^{-1}} + T_{s_{k+j}^{-1}}
\qquad
\text{and}
\qquad
T_{t_k^{-1}} {\widetilde{\varphi}}_j 
=
u^{2j} T_{t_{k-j}^{-1}} + T_{t_{k+j}^{-1}} .  
\]
\end{lemma}

\begin{proof}

Note $j < n$, so
\eqref{eq:varphi} applies to ${\widetilde{\varphi}}_j$.  
Define $s^{*}, t^{*} \in \set{s,t}$
by $s_k = s^{*}_{j} s_{k-j}$ and 
$\set{s^{*},t^{*}} = \set{s,t}$.  Then
\begin{equation*}
\begin{split}
T_{s_k^{-1}} {\widetilde{\varphi}}_j
& = 
T_{s_{k-j}^{-1}} T_{{s^{*}}_{j}^{-1}} 
 \left( u^{2j} T_{{s^{*}}_{j}^{-1}}^{-1} + T_{t^{*}_{j}}\right) \\
& = 
u^{2j} T_{s_{k-j}^{-1}} +
  T_{s_{k-j}^{-1}} T_{{s^{*}}_{j}^{-1}} T_{t^{*}_{j}} 
=  
u^{2j} T_{s_{k-j}^{-1}} + T_{s_{k+j}^{-1}} 
\end{split}
\end{equation*}
since $s_{k-j}^{-1} {s^{*}}_{j}^{-1} t^{*}_j
  = s_{k+j}^{-1}$ and
 $\ell(s_{k-j}^{-1})+\ell( {s^{*}}_{j}^{-1}) + \ell( t^{*}_j)
  = \ell(s_{k+j}^{-1})$.
Thus the first equation holds.
The second
 follows by applying the automorphism $T_s \leftrightarrow T_t$
 of $H_J$ 
 to the first. 
\end{proof}

%%%%%%%%%%%%%%%%%%%%%%%%%%%%%%%%%%%%%%%%%%%%

\begin{lemma}
\label{lemma:etagammadelta}
If $0 \le j \le k$ and $j+k \le n$, then
\begin{align*}
\tag{i}
\quad
T_{s_k^{-1}} {\widetilde{\eta}}_j  
& = 
\sum_{i=0}^{2j} u^{i} T_{s_{k+j-i}^{-1}}  &
\text{and} &&
T_{t_k^{-1}} {\widetilde{\eta}}_j  
 & = 
\sum_{i=0}^{2j} u^{i} T_{t_{k+j-i}^{-1}} ,   \\
\tag{ii}
T_{s_k^{-1}} {\widetilde{\gamma}}_j  
& = 
\sum_{i=0}^{2j} (-u)^{i} T_{s_{k+j-i}^{-1}} &
\text{and} &&
T_{t_k^{-1}} {\widetilde{\gamma}}_j  
 & = 
\sum_{i=0}^{2j} (-u)^{i} T_{t_{k+j-i}^{-1}} , \qquad \\
\tag{iii}
T_{s_k^{-1}} {\widetilde{\delta}}_j
& = 
\sum_{i=0}^{j} u^{2i} T_{s_{k+j-2i}^{-1}}  &
\text{and} &&
T_{t_k^{-1}} {\widetilde{\delta}}_j
 & = 
\sum_{i=0}^{j} u^{2i} T_{t_{k+j-2i}^{-1}} . 
\end{align*}
\end{lemma}

\begin{proof}
Observe that (i) holds when $k=0$
because $\widetilde{\eta}_0 = T_e$.  
Assume $\ell > 0$ and (i) holds when $0 \le k < \ell$.
Suppose $0 \le j \le \ell$ and $j+\ell \le n$.  
If $j \le \ell-1$, then 
\begin{equation*}
T_{s_\ell^{-1}} {\widetilde{\eta}}_j
= 
T_s T_{t_{\ell-1}^{-1}} {\widetilde{\eta}}_j 
= 
T_s \sum_{i=0}^{2j} u^{i} T_{t_{\ell+j-i-1}^{-1}} 
= 
\sum_{i=0}^{2j} u^{i} T_{s_{\ell+j-i}^{-1}}
\end{equation*}
by the induction hypothesis. 
On the other hand, if $j = \ell$, then 
\begin{equation*}
\begin{split}
T_{s_\ell^{-1}} {\widetilde{\eta}}_\ell
& = 
T_{s_\ell^{-1}} ({\widetilde{\varphi}}_\ell + u {\widetilde{\eta}}_{\ell-1}) 
= 
T_{s_\ell^{-1}} {\widetilde{\varphi}}_\ell + u T_s T_{t_{\ell-1}^{-1}} 
     {\widetilde{\eta}}_{\ell-1} \\
& = 
u^{2\ell} T_e + T_{s_{2\ell}^{-1}}
 + u T_s \sum_{i=0}^{2\ell-2} u^{i} T_{t_{2\ell-i-2}^{-1}} \\
& = 
 u^{2\ell} T_e + T_{s_{2\ell}^{-1}} 
   + \sum_{i=0}^{2\ell-2} u^{i+1} T_{s_{2\ell-(i+1)}^{-1}} 
= 
\sum_{\ell=0}^{2\ell} u^{\ell} T_{s_{2\ell-\ell}^{-1}}
\end{split}
\end{equation*}
by Lemma~\ref{lemma:varphi}.
Thus the first equation of (i) holds in the case $k = \ell$.
The second equation of (i) 
 follows in the case $k = \ell$ 
 by applying the automorphism $T_s \leftrightarrow T_t$ 
to  the first equation.  
Thus (i) holds by induction.

Let $\zeta$ be the automorphism of $\rationals(u)$ 
determined by $\zeta(u)=-u$. Extend $\zeta$ to 
a semilinear automorphism of $H_J$  by defining
$\zeta(\sum  \alpha_w T_w)
=
\sum \zeta(\alpha_w) T_w$.  
Then $\zeta({\widetilde{\eta}}_m) = {\widetilde{\gamma}}_m$, and so 
the formulas of (ii) are obtained by applying $\zeta$ to the
formulas of (i).  Finally, (iii) follows by averaging the
formulas of (i) and (ii).
\end{proof}

%%%%%%%%%%%%%%%%%%%%%%%%%%%%%%%%%%%%%%%%%%%%
%%%%%%%%%%%%%%%%%%%%%%%%%%%%%%%%%%%%%%%%%%%%

\section{Proof of Theorem~\ref{theorem:main}}
\label{section:mainproof}

We begin this section by outlining constructions 
due to Lusztig (\cite{lusztigbarop}, 2.4--2.10) 
of $H_J$-modules with bases supporting  $W_J$-digraphs
when $W_J$ is a finite dihedral group.  
The arguments 
given here, which differ somewhat from those in 
\cite{lusztigbarop}, are included for the sake of
completeness.  

Assume $J = \set{s,t}$, $1 < n = n(s,t) < \infinity$.  
When arguing that the $\set{s,t}$-labeled digraphs in
Figures~\ref{fig:main1}--\ref{fig:main8} are
$W_J$-digraphs,  
we may as well assume $n = m$ for
Figures~\ref{fig:main1}--\ref{fig:main3}, 
$n=2m-1$ for Figures~\ref{fig:main4}--\ref{fig:main5}, and
$n=2m-2$ in Figure~\ref{fig:main6}.   
Indeed, if $n$, $n^\prime$ are positive integers and $n$ divides
$n^\prime$, then
\[
\overbrace{ \tau_s \tau_t \cdots}^{n} = 
\overbrace{ \tau_t \tau_s \cdots}^{n}
\quad\text{implies}\quad
\overbrace{ \tau_s \tau_t \cdots}^{n^\prime} = 
\overbrace{ \tau_t \tau_s \cdots}^{n^\prime}.
\]
Put $s^\prime = s$ if $m$ is even, $s^\prime = t$ if $m$ is odd, and 
define $t^\prime$ by 
$\set{s^\prime,t^\prime} = \set{s,t}$.  
We consider cases. 

%%%%%%%%%%%%%%%%%%%%%%%%%%%%%%%%%%%%%%%%%%%%

\begin{case} Figure~\ref{fig:main1}, $n = m \ge 2$. 

Define $\mu_0 = T_e$ and
\begin{equation*}
\begin{cases}
\mu_1  = T_s \mu_0, \mu_2 = T_t \mu_1, \cdots, 
  \mu_{m-1} = T_{s^\prime} \mu_{m-2}, \mu_m = T_{t^\prime} \mu_{m-1}, & \\
\mu_1^\prime = T_t \mu_0, \mu_2^\prime = T_s \mu_1^\prime, \cdots, 
  \mu_{m-1}^\prime = T_{t^\prime} \mu_{m-2}^\prime, 
  \mu_{m} = T_{s^\prime} \mu_{m-1}^\prime. &
\end{cases}
\end{equation*}
Then
\[
\mu_m = T_{s_m} = T_{t_m} = \mu_m^\prime,
\]
and so by Lemma~\ref{lemma:dihedral} 
$X = \set{\mu_0,\mu_1,\dots,\mu_{m-1},\mu_1^\prime,
  \mu_2^\prime,\dots,\mu_m^\prime} = \setof{T_w}{w\in W_J}$  
  supports a $W_J$-digraph.  
This $W_J$-digraph is isomorphic to the $J$-labeled digraph of 
Figure~\ref{fig:main1} via $\mu_j \leftrightarrow \alpha_j$ for
$0 \le j \le m-1$, $\mu_j^\prime \leftrightarrow \beta_j$ for $1 \le j \le m$.
\end{case}

%%%%%%%%%%%%%%%%%%%%%%%%%%%%%%%%%%%%%%%%%%%%

\begin{case} Figure~\ref{fig:main2}, $n = m \ge 2$. 

Let $\nu_0 = T_e$, and define
\begin{equation*}
\begin{cases}
\nu_1 = T^{\circ}_s \nu_0, \nu_2 = T_t \nu_1, \dots, 
   \nu_{m-1} = T_{s^\prime} \nu_{m-2}, 
   \nu_{m} = T_{t^\prime} \nu_{m-1}, & \\
\nu_1^\prime = T_t \nu_0, \nu_2^\prime = T_s \nu_1^\prime, \dots, 
   \nu_{m-1}^\prime = T_{t^\prime} \nu_{m-2}^\prime,
   \nu_{m}^\prime = T_{s^\prime}^\circ \nu_{m-1}^\prime. & 
\end{cases}
\end{equation*}
Then
\begin{equation*}
\begin{split}
\nu_m 
& = 
 T_{t_{m-1}} T^{\circ}_s 
 = (u+1)^{-1} T_{t_{m-1}} (T_s - u)
 = (u+1)^{-1} \left( T_{s_{m}} - u T_{t_{m-1}} \right) \\
&  = 
 (u+1)^{-1} \left( T_{t_{m}} - u T_{t_{m-1}} \right) 
 =
 (u+1)^{-1} \left( T_{s^\prime} T_{t_{m-1}}  -  u T_{t_{m-1}} \right) \\
&  = 
 (u+1)^{-1} (T_{s^\prime} - u) T_{t_{m-1}} 
 = T_{s^\prime}^\circ T_{t_{m-1}} = \nu_m^\prime,
\end{split}
\end{equation*}
so $X = \set{\nu_0,\nu_1,\dots,\nu_{m-1},\nu_1^\prime,
 \nu_2^\prime,\dots,\nu_m^\prime}$ is linearly independent, 
 so is a basis for $H_J$, and
 supports a $W_J$-digraph by Lemma~\ref{lemma:dihedral}.  This
 $W_J$-digraph is isomophic to the $J$-labeled digraph of 
 Figure~\ref{fig:main2} via $\nu_j \leftrightarrow \alpha_j$ for
 $0 \le j \le m-1$, $\nu_j^\prime \leftrightarrow \beta_j$ for
 $1 \le j \le m$.
\end{case}

%%%%%%%%%%%%%%%%%%%%%%%%%%%%%%%%%%%%%%%%%%%%

\begin{case} Figure~\ref{fig:main3}, $n = m \ge 2$. 

Interchanging $s$ and $t$ in the argument
given for the previous case shows
that the $J$-labeled multigraph in Figure~\ref{fig:main3} is
a $W_J$-digraph.  
\end{case}

%%%%%%%%%%%%%%%%%%%%%%%%%%%%%%%%%%%%%%%%%%%%

\begin{case} Figure~\ref{fig:main4}, $n = 2m-1$, $m \ge 2$.

Define an element $\eta_0$ of $H_J$ by
\[
\eta_0
=
\widetilde{\eta}_{m-1}
=
\widetilde{\varphi}_{m-1} + u \widetilde{\varphi}_{m-2}
 + u^2 \widetilde{\varphi}_{m-3} + \cdots + u^{m-1}\widetilde{\varphi}_0 .
\]
Suppose $m$ is even.  Then by part (i)
of Lemma~\ref{lemma:etagammadelta}, 
\begin{equation*}
\begin{split}
T_{t_{m-1}} (T_s-u) \eta_0
& = 
T_{s_{m}} {\widetilde{\eta}}_{m} - u T_{t_{m-1}} {\widetilde{\eta}}_{m-1} 
= 
T_{t_{m}^{-1}} {\widetilde{\eta}}_{m-1} - u T_{t_{m-1}^{-1}} {\widetilde{\eta}}_{m-1} \\
& = 
\sum_{i=0}^{2m-2} u^{i} T_{t_{2m-i-1}^{-1}}
   - u \sum_{i=0}^{2m-2} u^{i} T_{t_{2m-i-2}^{-1}} \\
& = 
T_{t_{2m-1}^{-1}} - u^{2m-1} 
=
T_{w_0} - u^{n}.  
\end{split}
\end{equation*}
On the other hand, if $m$ is odd, 
then
\begin{equation*}
\begin{split}
T_{t_{m-1}} (T_s-u) \eta_0
& = 
T_{s_{m}} {\widetilde{\eta}}_{m-1} - u T_{t_{m-1}} {\widetilde{\eta}}_{m-1} 
= 
T_{s_{m}^{-1}} {\widetilde{\eta}}_{m-1} - u T_{s_{m-1}^{-1}} {\widetilde{\eta}}_{m-1} \\
& = 
\sum_{i=0}^{2m-2} u^{i} T_{s_{2m-i-1}^{-1}}
   - u \sum_{i=0}^{2m-2} u^{i} T_{s_{2m-i-2}^{-1}} \\
& = 
T_{s_{2m-1}^{-1}} - u^{2m-1} 
=
T_{w_0} - u^{n}.
\end{split}
\end{equation*}
Hence 
\[
T_{t_{m-1}} (T_s - u)  \eta_0
=
T_{w_0} - u^{n}
=
T_{s_{m-1}} (T_t - u)  \eta_0,
\]
where the second equation follows by applying the automorphism
$T_s \leftrightarrow T_t$ to the first.
Hence if we define
\begin{equation*}
\begin{cases}
 \eta_1 
=  T^{\circ}_s \eta_0, \eta_2 = T_t \eta_1,
\dots, \eta_{m-1} = T_{s^\prime} \eta_{m-2},
\eta_{m} = T_{t^\prime} \eta_{m-1}, & \\
\eta_1^\prime 
= T_t^\circ \eta_0, \eta_2^\prime = T_s \eta_1^\prime,
\dots, \eta_{m-1}^\prime = T_{t^\prime}\eta_{m-2}^\prime,
\eta_{m}^\prime = T_{s^\prime} \eta_{m-1}^\prime, &
\end{cases}
\end{equation*}
then
\[
\eta_{m} = (u+1)^{-1} (T_{w_0} - u^{n}) = \eta_{m}^\prime .
\]
Therefore 
$X=\set{\eta_0,\eta_1,\eta_2,\dots,\eta_{m-1},
\eta_1^\prime,\eta_2^\prime,\dots,\eta_{m-1}^\prime,\eta_{m}^\prime}$ is
a basis for 
a left ideal in $H_J$ and $X$ supports a $W_J$-digraph
by Lemma~\ref{lemma:dihedral}.
This $W_J$-digraph is isomorphic to the $J$-labeled digraph in
Figure~\ref{fig:main4} via 
$\eta_{j} \leftrightarrow \alpha_{j}$ for $0 \le j \le m-1$,
$\eta_{j}^\prime \leftrightarrow  \beta_j$ for $1 \le j \le m$.
\end{case}

%%%%%%%%%%%%%%%%%%%%%%%%%%%%%%%%%%%%%%%%%%%%

\begin{case} Figure~\ref{fig:main5}, $n = 2m-1$, $m \ge 2$.

Put
\[
\gamma_0
= \widetilde{\gamma}_{m-1}
= \widetilde{\varphi}_{m-1} - u \widetilde{\varphi}_{m-2}
 + u^2 \widetilde{\varphi}_{m-3} \pm \cdots + (-u)^{m-1} \widetilde{\varphi}_0 .
\]
If $m$ is even, then part (ii) of Lemma~\ref{lemma:etagammadelta}
gives
\begin{equation*}
\begin{split}
& (T_{t^\prime} - u) T_{s_{m-1}} \gamma_0 \\
& \qquad
=
(T_{s} - u) T_{s_{m-1}} {\widetilde{\gamma}}_{m-1} 
= 
T_{s_{m}} {\widetilde{\gamma}}_{m-1} - T_{s_{m-1}}{\widetilde{\gamma}}_{m-1} \\
& \qquad
=
T_{t_{m}^{-1}} {\widetilde{\gamma}}_{m-1} - T_{s_{m-1}^{-1}} {\widetilde{\gamma}}_{m-1} 
= 
\sum_{i=0}^{2m-2} (-u)^{i} T_{t_{2m-i-1}^{-1}}
   - u \sum_{i=0}^{2m} (-u)^{i} T_{s_{2m-i-2}^{-1}} \\
& \qquad
=
\sum_{w\in W} (-u)^{n-\ell(w)} T_w.
\end{split}
\end{equation*}
On the other hand, if $m$ is odd, then
\begin{equation*}
\begin{split}
& (T_{t^\prime} - u) T_{s_{m-1}} \gamma_0 \\
& \qquad 
= 
(T_{t} - u) T_{s_{m-1}} {\widetilde{\gamma}}_{m-1} 
= 
T_{s_{m}} {\widetilde{\gamma}}_{m-1} - u T_{s_{m-1}} {\widetilde{\gamma}}_{m-1} \\
& \qquad
=
T_{s_{m}^{-1}} {\widetilde{\gamma}}_{m-1} - u T_{t_{m-1}^{-1}} {\widetilde{\gamma}}_{m-1} 
= 
\sum_{i=0}^{2m-2} (-u)^{i} T_{s_{2m-i-1}^{-1}}
   - u \sum_{i=0}^{2m-2} (-u)^{i} T_{t_{2m-i-2}^{-1}} \\
& \qquad
=
\sum_{w\in W} (-u)^{n-\ell(w)} T_w.
\end{split}
\end{equation*}
Therefore
\[
(T_{t^\prime} - u) T_{s_{m-1}} \gamma_0
= \sum_{w\in W} (-u)^{n-\ell(w)} T_w
= (T_{s^\prime} - u) T_{t_{m-1}} \gamma_0,
\]
with the second equation following from the first by applying the
automorphism $T_s \leftrightarrow T_t$.
Hence if we put 
\begin{equation*}
\begin{cases}
\gamma_1 
= T_s \gamma_0, \gamma_2 = T_t \gamma_1, 
\dots, \gamma_{m-1} = T_{s^\prime} \gamma_{m-2}, 
\gamma_{m} = T_{t^\prime}^\circ \gamma_{m-1}, & \\
\gamma_1^\prime 
=  T_t \gamma_0, \gamma_2^\prime = T_s \gamma_1^\prime, 
\dots, \gamma_{m-1}^\prime = T_{t^\prime}\gamma_{m-2}^\prime, 
\gamma_{m}^\prime = T_{s^\prime}^\circ \gamma_{m-1}^\prime, &
\end{cases}
\end{equation*}
then
\[
\gamma_m
= (u+1)^{-1} \sum_{w\in W} (-u)^{n-\ell(w)} T_w
= \gamma_m^\prime.
\]
Thus  
$X = \set{\gamma_0,\gamma_1,\dots,\gamma_{m-1},\gamma_1^\prime,
\gamma_2^\prime,\dots,\gamma_{m-1}^\prime,\gamma_{m}^\prime}$
is a basis for a left ideal of $H_J$ supporting a $W_J$-digraph.
Moreover, this $W_J$-digraph is isomorphic to the digraph
of Figure~\ref{fig:main5} via
$\gamma_{j}\leftrightarrow \alpha_j$ for $0 \le j \le m-1$,
$\gamma_{j}^\prime \leftrightarrow \beta_j$ for $1 \le j \le m$.
\end{case}

%%%%%%%%%%%%%%%%%%%%%%%%%%%%%%%%%%%%%%%%%%%%

\begin{case} Figure~\ref{fig:main6}, $n = 2m-2$, $m \ge 2$.

Define
\[
\delta_0
=
\widetilde{\delta}_{m-2}
=
\frac{1}{2} (\widetilde{\eta}_{m-2} + \widetilde{\gamma}_{m-2}).
\]
If $m$ is even, then by part (iii) of
Lemma~\ref{lemma:etagammadelta} we have
\begin{equation*}
\begin{split}
(T_{t^\prime}-u) T_{t_{m-2}} (T_s - u) \delta_0 
& = 
(T_{s_m} - u T_{s_{m-1}} - u T_{t_{m-1}} + u^2 T_{t_{m-2}}) {\widetilde{\delta}}_{m-2} \\
& = 
(T_{t_m^{-1}} - u T_{s_{m-1}^{-1}} - u T_{t_{m-1}^{-1}} + u^2 T_{s_{m-2}^{-1}}) {\widetilde{\delta}}_{m-2} \\
& = 
\sum_{i=0}^{m-2} u^{2i} T_{t_{2m-2-2i}^{-1}}
- u \sum_{i=0}^{m-2} u^{2i} T_{s_{2m-3-2i}^{-1}} \\
& \qquad\quad
{} - u \sum_{i=0}^{m-2} u^{2i} T_{t_{2m-3-2i}^{-1}}
+ u^2 \sum_{i=0}^{m-2} u^{2i} T_{s_{2m-4-2i}^{-1}} \\
& = 
\sum_{w \in W} (-u)^{n-\ell(w)} T_w .
\end{split}
\end{equation*}
On the other hand, if $m$ is odd then
\begin{equation*}
\begin{split}
(T_{t^\prime}-u) T_{t_{m-2}} (T_s - u) \delta_0
& = 
(T_{s_m} - u T_{s_{m-1}} - u T_{t_{m-1}} + u^2 T_{t_{m-2}}) {\widetilde{\delta}}_{m-2} \\
& = 
(T_{s_m^{-1}} - u T_{t_{m-1}^{-1}} - u T_{s_{m-1}^{-1}} + u^2 T_{t_{m-2}^{-1}}) {\widetilde{\delta}}_{m-2} \\
& = 
\sum_{i=0}^{m-2} u^{2i} T_{s_{2m-2-2i}^{-1}}
- u \sum_{i=0}^{m-2} u^{2i} T_{t_{2m-3-2i}^{-1}} \\
& \qquad\quad
{} - u \sum_{i=0}^{m-2} u^{2i} T_{s_{2m-3-2i}^{-1}}
+ u^2 \sum_{i=0}^{m-2} u^{2i} T_{t_{2m-4-2i}^{-1}} \\
& = 
\sum_{w \in W} (-u)^{n-\ell(w)} T_w.
\end{split}
\end{equation*}
Therefore
\[
(T_{t^\prime}-u) T_{t_{m-2}} (T_s - u) \delta_0
=
\sum_{w \in W} (-u)^{n-\ell(w)} T_w
=
(T_{s^\prime}-u) T_{s_{m-2}} (T_t - u) \delta_0,
\]
with the second equation following from the first by applying the
automorphism $T_s \leftrightarrow T_t$.
Thus if we define
\begin{equation*}
\begin{cases}
\delta_1 
=  T^{\circ}_s \delta_0,
\delta_2 = T_t \delta_1, \dots, \delta_{m-2} = T_{s^\prime} \delta_{m-1}, 
\delta_{m} = T _{t^\prime}^\circ \delta_{m-1}, & \\
 \delta_1^\prime 
 = T_t^\circ \delta_0,
\delta_2^\prime = T_s \delta_1^\prime, \dots, 
\delta_{m-2}^\prime = T_{t^\prime} \delta_{m-1}^\prime, 
\delta_{m}^\prime = T_{s^\prime}^\circ \delta_{m-1}^\prime, & 
\end{cases}
\end{equation*}
then
\[
\delta_{m}
=
(u+1)^{-2} \sum_{k=0}^{n} (-u)^{n-k} \sigma_k
=
\delta_{m}^\prime .
\]
Hence 
$X = \set{\delta_0,\delta_1,\dots,\delta_{m-1},
\delta_1^\prime,\delta_2^\prime,\dots,\delta_{m}^\prime}$ is a basis
for a left ideal of $H_J$ that supports a $W_J$-digraph.
This $W_J$-digraph is isomorphic to the $J$-labeled
multigraph of Figure~\ref{fig:main6} via
$\delta_j \leftrightarrow \alpha_j$ for $0 \le j \le m-1$, 
$\delta_j^\prime \leftrightarrow \beta_j$ for $1 \le j \le m$.
\end{case}

%%%%%%%%%%%%%%%%%%%%%%%%%%%%%%%%%%%%%%%%%%%%

\begin{case} Figure~\ref{fig:main7} or Figure~\ref{fig:main8}, $m=1$, $n \ge 2$ arbitrary.

Suppose $\Gamma$ is one of the $J$-labeled digraphs of
Figures~\ref{fig:main7}--\ref{fig:main8}.  Then with $M = \linspan \set{\alpha_0,\beta_1}$, 
$T_s$ and $T_t$ induce the same operator $\tau_s = \tau_t$ on $M$.
Thus the relation
\eqref{eq:taudihedralrelation} 
holds automatically, and so
$\Gamma$ is a $W_J$-digraph.
\end{case}

%%%%%%%%%%%%%%%%%%%%%%%%%%%%%%%%%%%%%%%%%%%%
%%%%%%%%%%%%%%%%%%%%%%%%%%%%%%%%%%%%%%%%%%%%

From the constructions above, it follows that
(b) implies (a) in Theorem~\ref{theorem:main}.
To establish the converse, 
we can reduce to the case $S = J = \set{s,t}$, 
$1 < n = n(s,t) < \infinity$, 
$W = W_J$, $H = H_J$, and need only show
that any connected $W$-digraph
 $\Gamma$ is isomorphic to one of the
$J$-labeled digraphs of Figures~\ref{fig:main1}--\ref{fig:main8}, 
with $m$ and $n$ satisfying the appropriate divisibility
conditions. 

Let $X$ be the set of vertices of $\Gamma$, and let
$M = \linspan  X$ be the associated $H$-module.  
If $\alpha \in X$, then $X \subseteq H\alpha$ because
$\Gamma$ is connected, so 
$\abs{X} = \dim M = \dim H\alpha \le \dim H = 2n$.  
Moreover, $\abs{X}$ is even by
Lemma~\ref{lemma:partitions}.  Since every 
vertex of $\Gamma$ is contained in exactly
$\abs{S} = 2$ edges, it follows that
$\Gammaundir$ is a simple cycle of size 
$2m$, where $1 \le m \le n$.

Let $\gamma_0$ be any vertex of $\Gamma$.    
Number the remaining vertices 
$\gamma_1, \gamma_2, \dots, \gamma_{2m-1}$ 
in such a way that
$\Gamma$ has an edge 
from $\gamma_{i-1}$ to $\gamma_i$ or
from $\gamma_{i}$ to $\gamma_{i-1}$ for $1\le i \le 2m-1$.
Put $\gamma_{2m} = \gamma_0$, so 
$\Gamma$ also has an edge from $\gamma_{2m-1}$ to
$\gamma_{2m}$ or from $\gamma_{2m}$ to $\gamma_{2m-1}$. 
We consider the subscript $j$ in $\gamma_j$ as an
integer modulo $2m$.  

Recall the linear characters $\lambda_1 = \ind, \lambda_2 = \sgn: H \rightarrow \rationals(u)$
of $H$ are determined by 
\[
\lambda_1 (T_s) = \lambda_1 (T_t) = u^{2}
\quad\text{and}\quad
\lambda_2 (T_s) = \lambda_2(T_t) = -1.
\]
If $n$ is even, there are two additional linear characters
$\lambda_3, \lambda_4 : H \rightarrow \rationals(u)$ given 
by
\[
\lambda_3(T_s) = u^2, \lambda_3(T_t) = -1
\quad\text{and}\quad
\lambda_4(T_s) = -1, \lambda_4(T_t) = u^2.
\]
It is known that
$H_\complexes = \complexes(u)\! \otimes_{\rationals(u)}\! H$ 
is split semisimple over $\complexes(u)$, any
 irreducible representation of $H_\complexes$ 
 of dimension greater than $1$ is
two-dimensional, and the eigenvalues of
$T_s$ and $T_t$ in any two-dimensional irreducible 
 representation are $-1$ and $u^2$ (see \cite{kilmoyersolomon}, 
 or \cite{geckpfeiffer}, 8.3).  
 Let $m_1$, $m_2$, $m_3$, $m_4$ be the 
 number of summands in a
 direct sum decomposition of 
 $M_\complexes = \complexes(u)\otimes_{\rationals(u)}M$
 into irreducible modules
 that afford
 $\lambda_{1}$, $\lambda_{2}$, $\lambda_3$, $\lambda_4$, respectively
 (with $m_3=m_4=0$ if $n$ is odd), and 
 let $N$ be the number of two-dimensional 
 irreducible summands.
With $P_s$ as in Lemma~\ref{lemma:partitions}, $T_s$ has 
eigenvalues $-1$ and $u^2$ on each subspace 
$\linspan \set{\alpha,\beta}$, $\set{\alpha,\beta}\in P_s$, 
and thus $T_s$ has a total of $m$ eigenvalues
$-1$ and $m$ eigenvalues $u^2$ on $M$.  
Since the same is true
of $T_t$, we must have
\begin{equation*}
m_1 + m_3 + N  = 
m_1 + m_4 + N = 
m_2 + m_3 + N = 
m_2 + m_4 + N = m,
\end{equation*}
and so $m_1 = m_2$ and $m_3 = m_4$.
By Lemma~\ref{lemma:eigenvalues}, the unique one-dimensional
subspace $M_{1}$ of $M$ that affords the character $\lambda_1$ is 
spanned by $v_{1}=\sum_{i=1}^{2m}\gamma_i$, and
thus $m_1 = 1$.  Hence $m_2 = 1$, so there is a unique
one-dimensional subspace $M_{2}$ of $M$ affording $\lambda_{2}$.
Let $v_{2} = \sum_{i=1}^{2m} \zeta_i \gamma_i$ be a nonzero
element of $M_{2}$. 
By Lemma~\ref{lemma:eigenvalues}, we have
\[
\zeta_i
=
\begin{cases}
-\dfrac{1}{u^2}\zeta_{i-1} & 
\text{if \solidedge{$\gamma_{i-1}$}{$\gamma_{i}$}{$s$} 
or \solidedge{$\gamma_{i-1}$}{$\gamma_{i}$}{$t$} is an edge of $\Gamma$,} \\
-u^{2}\zeta_{i-1} & 
\text{if \reversedsolidedge{$\gamma_{i-1}$}{$\gamma_{i}$}{$s$}
or \reversedsolidedge{$\gamma_{i-1}$}{$\gamma_{i}$}{$t$} is an edge of $\Gamma$,} \\
- \dfrac{u+1}{u^2-u}  \zeta_{i-1} &
\text{if \dashededge{$\gamma_{i-1}$}{$\gamma_{i}$}{$s$}
or \dashededge{$\gamma_{i-1}$}{$\gamma_{i}$}{$t$} is an edge of $\Gamma$,} \\
- \dfrac{u^2-u}{u+1}  \zeta_{i-1} &
\text{if \reverseddashededge{$\gamma_{i-1}$}{$\gamma_{i}$}{$s$}
or \reverseddashededge{$\gamma_{i-1}$}{$\gamma_{i}$}{$t$} is an edge of $\Gamma$.} 
\end{cases}
\]
for $1 \le i \le 2m$.
If $m = 1$, then it follows that the edge joining $\gamma_0$
and $\gamma_1$ labeled $s$ must have the same type and
direction as  the edge joining $\gamma_0$
and $\gamma_1$ labeled $t$, and so $\Gamma$ is 
isomorphic to one of the
$J$-labeled digraphs of Figure~\ref{fig:main7}--\ref{fig:main8}.  We assume
$m \ge 2$ for the remainder of the proof, so there is a unique
edge joining $\gamma_i$ to $\gamma_{i-1}$ for $1\le i \le 2m$.
Further, 
\[
\zeta_{0} = \zeta_{2m}
= \zeta_0 \prod_{i=1}^{2m} \frac{\zeta_i}{\zeta_{i-1}},
\]
and so 
$\prod_{i=1}^{2m} (\zeta_i / \zeta_{i-1}) = 1$.
It follows that the number of edges of type
\solidedge{$\gamma_{i-1}$}{$\gamma_{i}$}{} 
(labeled either $s$ or $t$)
 is equal to the number of edges of type
\reversedsolidedge{$\gamma_{i-1}$}{$\gamma_{i}$}{}, $1 \le i \le 2m$, 
and the number of edges of type
\dashededge{$\gamma_{i-1}$}{$\gamma_{i}$}{}
is equal to the number of edges of type
\reverseddashededge{$\gamma_{i-1}$}{$\gamma_{i}$}{}, $1 \le i \le 2m$.  

Next, we compute the coefficient $\kappa_j$ 
of $\gamma_j$ in the expression for $T_s T_t \gamma_j$ 
as a linear combination of $\set{\gamma_1,\dots,\gamma_{2m}}$.
These coefficients are given in Table~\ref{tab:traces},
which is organized according to the types of edges
joining $\gamma_j$ to the adjacent vertices
$\delta$, $\varepsilon$ in $\Gamma$.  (Either
$\delta = \gamma_{j-1}$ and $\varepsilon=\gamma_{j+1}$ or
$\delta = \gamma_{j+1}$ and $\varepsilon=\gamma_{j-1}$:
the coefficient $\kappa_j$ is the same in either case.) 
\begin{table}[ht]
\setlength{\tabcolsep}{1pt}
\caption{}%
\label{tab:traces}%
\begin{minipage}{0.43\textwidth}%
\centering%
\begin{tabular}{@{}cc@{}}%
\text{edges in $\Gamma$} & \text{coefficient $\kappa_j$} \\
\hline%
\solidrightsolidright{$\delta$}{$\gamma_j$}{$s$}{$\varepsilon$}{$t$}& $0$ \\
\solidrightdashedright{$\delta$}{$\gamma_j$}{$s$}{$\varepsilon$}{$t$}& $u(u^2-1)$ \\
\dashedrightsolidright{$\delta$}{$\gamma_j$}{$s$}{$\varepsilon$}{$t$}&  $0$ \\
\dashedrightdashedright{$\delta$}{$\gamma_j$}{$s$}{$\varepsilon$}{$t$}&  $u(u^2-u-1)$ \\
\solidleftsolidleft{$\delta$}{$\gamma_j$}{$s$}{$\varepsilon$}{$t$}&  $0$ \\
\solidleftdashedleft{$\delta$}{$\gamma_j$}{$s$}{$\varepsilon$}{$t$}& $0$ \\
\dashedleftsolidleft{$\delta$}{$\gamma_j$}{$s$}{$\varepsilon$}{$t$}& $u(u^2-1)$ \\
\dashedleftdashedleft{$\delta$}{$\gamma_j$}{$s$}{$\varepsilon$}{$t$}&  $u(u^2-u-1)$
\end{tabular}%
\end{minipage}%
\begin{minipage}{0.49\textwidth}%
\centering%
\begin{tabular}{@{}cc@{}}%
\text{edges in $\Gamma$} & \text{coefficient $\kappa_j$} \\
\hline%
\solidleftsolidright{$\delta$}{$\gamma_j$}{$s$}{$\varepsilon$}{$t$} & $0$ \\
\solidleftdashedright{$\delta$}{$\gamma_j$}{$s$}{$\varepsilon$}{$t$} & $0$ \\
\dashedleftsolidright{$\delta$}{$\gamma_j$}{$s$}{$\varepsilon$}{$t$} &  $0$ \\
\dashedleftdashedright{$\delta$}{$\gamma_j$}{$s$}{$\varepsilon$}{$t$} & $u^2$ \\
\solidrightsolidleft{$\delta$}{$\gamma_j$}{$s$}{$\varepsilon$}{$t$} & $(u^2-1)^2$ \\
\solidrightdashedleft{$\delta$}{$\gamma_j$}{$s$}{$\varepsilon$}{$t$} & $(u^2-1)(u^2-u-1)$ \\
\dashedrightsolidleft{$\delta$}{$\gamma_j$}{$s$}{$\varepsilon$}{$t$} & $(u^2-1)(u^2-u-1)$ \\
\dashedrightdashedleft{$\delta$}{$\gamma_j$}{$s$}{$\varepsilon$}{$t$} &  $(u^2-u-1)^2$ 
\end{tabular}
\end{minipage}%
\end{table}\noindent
The entries of this table are easily verified. 
For example, 
if \dashedleftdashedright{$\gamma_{j-1}$}{$\gamma_j$}{$s$}{$\gamma_{j+1}$}{$t$}
are edges in $\Gamma$, then
\[
T_s T_t \gamma_j
=T_s (u \gamma_j + (u+1)\gamma_{j+1})
= u (u \gamma_j + (u+1)\gamma_{j-1}) + (u+1) T_s \gamma_{j+1},
\]
and so $\kappa_j = u^2$ since
$T_s \gamma_{j+1} \in \linspan \set{\gamma_{j+1},\gamma_{j+2}}$. 
On the other hand, if $\Gamma$ has edges
\solidrightdashedright{$\gamma_{j+1}$}{$\gamma_j$}{$s$}{$\gamma_{j-1}$}{$t$}, 
then
\[
T_s T_t \gamma_j
=
T_s (u \gamma_{j+1} + (u+1)\gamma_{j-1})
=
u ( u^2 \gamma_{j+1} + (u^2-1)\gamma_j) + (u+1)T_s \gamma_{j-1},
\]
and so $\kappa_j = u(u^2-1)$ because
$T_s \gamma_{j-1} \in \linspan \set{\gamma_{j-1},\gamma_{j-2}}$.

From Table~\ref{tab:traces} we see that the
constant term in the trace 
$\tr (T_s T_t) = \sum_{j=1}^{2m} \kappa_j$ is equal to 
the number of sinks in $\Gamma$.
However, $T_s T_t$ has values $u^4$ and $1$
under $\lambda_{1}$ and $\lambda_{2}$, respectively, and value
$-u^2$ under both $\lambda_3$ and $\lambda_4$ if $n$ is even.
Also, $T_s T_t$ has eigenvalues of the form 
$e^{\imath \theta} u^2$, $e^{-\imath \theta} u^2$ in any two-dimensional
irreducible representation of $H_\complexes$, where 
$e^{\imath \theta}$ is a complex $n$th root of unity (\cite{kilmoyersolomon}, Theorem~2, 
or  \cite{geckpfeiffer}, Theorem~8.3.1).  
Therefore the constant term of $\tr (T_s T_t)$ 
is $m_2 = 1$.  Hence $\Gamma$ has a unique sink $\beta$, and so also
a unique source $\alpha$.  

Renumbering the vertices if necessary, we can assume that 
$\gamma_0 = \alpha$. 
Since $\Gamma$ has a unique sink $\beta$ 
and the number of edges of type 
\solidedge{$\gamma_{i-1}$}{$\gamma_{i}$}{} 
 is equal to the number of edges of type
\reversedsolidedge{$\gamma_{i-1}$}{$\gamma_{i}$}{}, $1 \le i \le 2m$, 
and the number of edges of type
\dashededge{$\gamma_{i-1}$}{$\gamma_{i}$}{}
is equal to the number of edges of type
\reverseddashededge{$\gamma_{i-1}$}{$\gamma_{i}$}{}, $1 \le i \le 2m$, 
it follows that
$\beta = \gamma_m$ is opposite to $\alpha$. 

Renumbering the vertices if needed, we can assume
that $\gamma_0$ and $\gamma_1$ 
are connected by an edge labeled $s$.  Define
$\gamma^\prime_j = \gamma_{2m-j}$ for $0 \le j \le m$, so
$\beta=\gamma^\prime_m$.
Then $\Gammaarrow$ has the
form shown in Figure~\ref{fig:reghor}.  
(Since each edge of $\Gammaarrow$ arises from either a
solid or a dashed edge in $\Gamma$, there are $2^{2m}$ 
possible $J$-labeled digraphs $\Gamma$ with this configuration.)
\begin{figure}[ht]
\centering
{%
%%%%%%%%%%%%%%%%%%%%%%%%%%%%%%%%%%%%%%%%%%%%
\begin{tikzpicture}
[node distance=0.6cm,%
pre/.style={<-,shorten <=1pt,>=angle 45},%
post/.style={->,shorten >=1pt,>=angle 45}];%
rectangle/.style={inner sep=0pt,minimum size=5mm}];%
\node[rectangle]	(left)		{$\gamma_0$};
\node[rectangle]	(top1)	[above=of left,xshift=10mm,yshift=-2mm]	{$\gamma_1$}
	edge		[pre]		node[yshift=2.5mm,xshift=-2.5mm]	{$s$}		(left);
\node[rectangle]	(top2)	[right=of top1]			{$\gamma_2$}
	edge		[pre]			node[yshift=2mm]	{$t$}		(top1);
\node		(topdots)		[right=of top2]		{$\cdots$}
	edge		[pre]		node[yshift=2mm]	{$s$}			(top2);
\node[rectangle]	(top3)	[right=of topdots]			{$\gamma_{m-2}$}
	edge		[pre]		node[yshift=2mm]	{$t^\prime$}	(topdots);
\node[rectangle]	(top4)	[right=of top3]		{$\gamma_{m-1}$}
	edge		[pre]		node[yshift=2mm]	{$s^\prime$}	(top3);		
\node[rectangle]	(bottom1)	[below=of left,xshift=10mm,yshift=2mm]	{$\gamma^\prime_{1}$}
	edge		[pre]		node[yshift=-2.5mm,xshift=-2.5mm]	{$t$}		(left);
\node[rectangle]	(bottom2)	[right=of bottom1]			{$\gamma^\prime_{2}$}
	edge		[pre]			node[yshift=-2mm]	{$s$}		(bottom1);
\node		(bottomdots)		[right=of bottom2]		{$\cdots$}
	edge		[pre]		node[yshift=-2mm]	{$t$}			(bottom2);
\node[rectangle]	(bottom3)	[right=of bottomdots]			{$\gamma^\prime_{m-2}$}
	edge		[pre]		node[yshift=-2mm]	{$s^\prime$}	(bottomdots);
\node[rectangle]	(bottom4)	[right=of bottom3]		{$\gamma^\prime_{m-1}$}
	edge		[pre]		node[yshift=-2mm]	{$t^\prime$}	(bottom3);
\node[rectangle]	(right)	[below=of top4,xshift=10mm,yshift=3mm]  {$\gamma^\prime_{m}$}
	edge		[pre]		node[xshift=2.5mm,yshift=2.5mm]	{$t^\prime$}	(top4)
	edge		[pre]		node[xshift=2.5mm,yshift=-2.5mm]	{$s^\prime$}	(bottom4);
\end{tikzpicture}
%%%%%%%%%%%%%%%%%%%%%%%%%%%%%%%%%%%%%%%%%%%%
%%%%%%%%%%%%%%%%%%%%%%%%%%%%%%%%%%%%%%%%%%%%
}%
\caption{$\Gammaarrow$}
\label{fig:reghor}
\end{figure}

From the discussion above, 
we know that the number of edges in $\Gamma$ of type
\dashededge{$\gamma_{i-1}$}{$\gamma_{i}$}{} (labeled either $s$ or $t$), $1 \le i \le m$, 
 is equal to the 
number of edges of type 
\dashededge{$\gamma^\prime_{i-1}$}{$\gamma^\prime_{i}$}{}, $1 \le i \le m$.
Also, from the description of the eigenvalues of
$T_s T_t$ above, $\tr (T_s T_t)$ must be an
even function of $u$.  
Let $N_1$ be the number of edges of the form 
\dashededge{$\xi$}{$\omega$}{} 
with $\xi$ not a source, that is, $\xi \ne \alpha = \gamma_0$, and 
let $N_2$ be the number of edges 
\dashededge{$\xi$}{$\omega$}{} with
$\omega$ a sink, that is, $\omega = \beta = \gamma_m^\prime$.
Then from Table~\ref{tab:traces}, the coefficient of $u^3$ 
in $\tr(T_s T_t)$ is $N_1 - N_2$, 
and hence $N_1 = N_2$. 
Therefore any edge of type
\dashededge{$\xi$}{$\omega$}{} 
that does not begin at $\gamma_0$ must end at $\gamma_m^\prime$.
Hence $\Gamma$ is isomorphic to one of the $J$-labeled
digraphs in Figures~\ref{fig:main1}--\ref{fig:main6}
via $\gamma_j \leftrightarrow \alpha_j$, $0 \le j \le m-1$, 
$\gamma^\prime_j \leftrightarrow \beta_j$, $1 \le j \le m$. 

Finally, let $\tau_s$ and $\tau_t$ be as in \eqref{eq:taudefinition},  
and let 
$\widetilde{A}_{s}(u), \widetilde{A}_{t}(u)$ be the
$(2m)\times(2m)$ matrices over $\rationals[u]$ 
representing $\tau_s$ and $\tau_t$ with
respect to the basis $X$ for $M$.  
Put
$A_s = \widetilde{A}_{s}(1)$, $A_t=\widetilde{A}_{t}(1)$.  Then 
\[
A_s^2 = I = A_t^2,
\qquad
\overbrace{\cdots A_t A_s}^{n}
=
\overbrace{\cdots A_s A_t}^{n},
\]
by \eqref{eq:tauquadraticrelation}, \eqref{eq:taudihedralrelation}, 
and so
$s \mapsto A_{s}$, $t\mapsto A_{t}$ extends to a 
 representation of groups $W_J \rightarrow \GL(2m,\rationals)$. 
One checks that the characteristic polynomial
of the matrix $A_{st} = A_{s} A_{t}$ representing $st$ is as
given in Table~\ref{tab:polys}.
\begin{table}[ht]
\centering
\caption{}%
\label{tab:polys}
\begin{tabular}{ccc}
$W$-digraph & \quad & characteristic polynomial of $A_{st}$ \\
\hline 
Figure~\ref{fig:main1}, Figure~\ref{fig:main2}, Figure~\ref{fig:main3} 
 && $(x^m-1)^2$  \\
Figure~\ref{fig:main4}, Figure~\ref{fig:main5} 
 && $(x-1)(x^{2m-1}-1)$ \\
Figure~\ref{fig:main6}
 && $(x-1)^2 (x^{m-1}+1)^2$  
\end{tabular}
\end{table}
Hence the order of $A_{st}$ as an element of $\GL(2m,\rationals)$ 
is $m$ in the case of
Figures~\ref{fig:main1}--\ref{fig:main3}, $2m-1$ in the case
of Figures~\ref{fig:main4}--\ref{fig:main5}, and
$2m-2$ in the case of Figure~\ref{fig:main6}.  Since this
order must divide $n$, 
the proof is complete. 

%%%%%%%%%%%%%%%%%%%%%%%%%%%%%%%%%%%%%%%%%%%%

\begin{proof}[Proof of Corollary~\ref{corollary:reversed}]
Let $(W,S)$ be a Coxeter system, and let
 $\Gamma$ be an $S$-labeled digraph.  
 For $J = \set{s,t} $
and $1 < n < \infinity$, denote by 
 ${\mathcal F}_{J,n}$ the
collection of all $J$-labeled digraphs $C$ of 
Figures~\ref{fig:main1}--\ref{fig:main8} \ 
for which $m = \abs{\vertices(C)} / 2$ satisfies the 
divisibility conditions of Theorem~\ref{theorem:main}.  
Then 
 $\Gamma$ is a $W$-digraph if and only if
 whenever  $J = \set{s,t} \subseteq S$
 with $1 < n = n(s,t) < \infinity$, each 
 connected component of $\Gamma_J$ is isomorphic to an 
 element of ${\mathcal F}_{J,n}$.  
It is easily seen that ${\mathcal F}_{J,n}$ is 
invariant under $C \mapsto C_{\text{rev}}$.  
Also, 
$C$ is a connected component of $\Gamma_J$
if and only if 
 $C_{\text{rev}}$ is a connected component of $(\Gamma_J)_{\text{rev}}$. 
The assertion of the corollary follows.
\end{proof}

%%%%%%%%%%%%%%%%%%%%%%%%%%%%%%%%%%%%%%%%%%%%

Let $w \mapsto w^{*}$ be an involutory automorphism of $(W,S)$, and
let $I_{*}= \setof{x \in W}{x^* = x^{-1}}$ be the set of twisted involutions
of $W$.  By Lusztig \cite{lusztigbarop}, Theorem~0.1,
there is an $H$-module
$M_{*}$ with basis $X = \setof{m_w}{w\in I_{*}}$ and
$H$-action determined by
\begin{equation*}
T_s m_w
=
\begin{cases}
m_{sws^*} & \text{if $sw \ne ws^* > w$}, \cr
(u^2-1) m_{w} + u^2 m_{sws^*} & \text{if $sw \ne ws^* < w$}, \cr
u m_w + (u+1) m_{sw} & \text{if $sw = ws^* > w$}, \cr
(u^2-u-1) m_w + (u^2-u) m_{sw} & \text{if $sw=w^* < w$}. \cr
\end{cases}
\end{equation*}
(Here $<$ and $>$ refer to the usual Bruhat order on $(W,S)$.) 
The basis $X$ for $M_{*}$ then affords the $W$-digraph
$\Gamma_{*}$ defined by 
\begin{equation*}
\solidedge{$m_{w}$}{$m_{sws^*}$}{$s$} \in \edges(\Gamma_*)
 \iff sw \ne ws^* > w
\end{equation*}
and
\begin{equation*}
\dashededge{$m_w$}{$m_{sw}$}{$s$} \in \edges(\Gamma_*)
 \iff sw = ws^* > w.
 \end{equation*}

%%%%%%%%%%%%%%%%%%%%%%%%%%%%%%%%%%%%%%%%%%%%

\begin{theorem}
\label{theorem:involutions}
Let $(W,S)$ be finite with longest element $w_0$, and let
 $\Gamma_*$ be the $W$-digraph corresponding
 to the involutory automorphism $w \mapsto w^*$ of $(W,S)$.  Then
$(\Gamma_*)_{\text{rev}}$ is isomorphic
to the $W$-digraph $\Gamma_{\#}$ corresponding to the
automorphism 
$w \mapsto w^{\#} = w_0 w^* w_0$ of $(W,S)$
via the bijection sending $m_x \in \vertices(\Gamma_*)$
to $m_{xw_0} \in \vertices(\Gamma_{\#})$. 
\end{theorem}

\begin{proof}  
Observe $w_0^* = w_0$ since
$w \mapsto w^*$ preserves lengths.
Suppose $x \in I_*$, so $x^{*} = x^{-1}$.
Then $x w_0 \in I_{\#}$ because
\[
(x w_0)^{\#} =
w_0 (x w_0)^{*} w_0
=
w_0 x^{*} w_0 w_0 
=
w_0 x^{-1} = (x w_0)^{-1} .
\]  
Likewise, if $x w_0 \in I_{\#}$, then $x \in I_*$. 
If $x, y \in I_*$ and $s \in S$, then
\begin{equation*}
\begin{split}
\solidedge{$m_{x}$}{$m_{y}$}{$s$} \in \edges(\Gamma_*)
& \iff
x < y = s x s^* \\
& \iff
y w_0 < x w_0 = (s y s^*) w_0 
= s (y w_0) (w_0 s^* w_0) = s (y w_0) s^{\#} \\
& \iff
\solidedge{$m_{yw_0}$}{$m_{xw_0}$}{$s$} \in \edges(\Gamma_{\#}),
\end{split}
\end{equation*}
and 
\begin{equation*}
\begin{split}
\dashededge{$m_{x}$}{$m_{y}$}{$s$} \in \edges(\Gamma_*)
& \iff
x < y = s x = x s^* \\
& \iff
y w_0 < x w_0 = (s y) w_0 = s (y w_0) = (y s^*) w_0 = (y w_0) s^{\#}
 \\
& \iff
\dashededge{$m_{yw_0}$}{$m_{xw_0}$}{$s$} \in \edges(\Gamma_{\#}).
\end{split}
\end{equation*}
Therefore $(\Gamma_*)_{\text{rev}}$ is isomorphic to $\Gamma_{\#}$ 
via the bijection $m_{x} \mapsto m_{x w_0}$ on vertices.  
\end{proof}

%%%%%%%%%%%%%%%%%%%%%%%%%%%%%%%%%%%%%%%%%%%%

\begin{corollary}
If $w_0$ is central in $W$, then $(\Gamma_*)_{\text{rev}}$ is isomorphic
to $\Gamma_*$.
\end{corollary}

%%%%%%%%%%%%%%%%%%%%%%%%%%%%%%%%%%%%%%%%%%%%

\begin{example}
Suppose
$W = \spanof{r,s,t}$ with
$n(r,s)=n(s,t)=3$, $n(r,t)=2$
and $w^* = w$ for $w \in W$, so
 $I_*$ is the
set of involutions in $W$ (including $e$).  The corresponding
$W$-digraph $\Gamma_*$ is shown in Figure~\ref{fig:a3trivial}.
(The vertices are labeled $x$ rather than
$m_x$ for $x \in I_*$.)
If
$w \mapsto w^{\#}$ is 
the nonidentity graph automorphism of $W$,
the corresponding
$W$-digraph $\Gamma_{\#}$ is as shown 
in Figure~\ref{fig:a3twisted}.
Note $(\Gamma_*)_{\text{rev}} \cong \Gamma_{\#}$.  
\begin{figure}[ht]
\begin{minipage}{1.0\textwidth}
\begin{minipage}{0.5\textwidth}
\centering
%%%%%%%%%%%%%%%%%%%%%%%%%%%%%%%%%%%%%%%%%%%%
\begin{tikzpicture}
[node distance=0.5cm,%
pre/.style={<-,shorten <=1pt,>=angle 45},%
post/.style={->,shorten >=1pt,>=angle 45}];%
rectangle/.style={inner sep=0pt,minimum size=5mm}];%
\node[rectangle]	(e)								{$e$};
\node[rectangle]	(s)	[above=of e]					{$s$}
	edge		[pre,dashed]	node[xshift=2mm]			{$s$}		(e);
\node[rectangle]	(r)	[below=of e,xshift=-10mm]		{$r$}
	edge		[pre,dashed]	node[yshift=2mm,xshift=-2mm]	{$r$}		(e);
\node[rectangle]	(t)	[below=of e,xshift=10mm]			{$t$}
	edge		[pre,dashed]	node[yshift=2mm,xshift=2mm]	{$t$}		(e);
\node[rectangle]	(rt)	[below=of r,xshift=10mm]			{$rt$}
	edge		[pre,dashed]	node[yshift=-2mm,xshift=-2mm]	{$t$}		(r)
	edge		[pre,dashed]	node[yshift=-2mm,xshift=2mm]	{$r$}		(t);
\node[rectangle]	(srs)	[left=of r,xshift=-2mm]			{$srs$}
	edge		[pre,bend left=10] node[yshift=2mm,xshift=-2mm] {$r$}	(s)
	edge		[pre]			node[yshift=2.5mm]			{$s$}		(r);
\node[rectangle]	(sts)	[right=of t,xshift=2mm]			{$sts$}
	edge		[pre,bend right=10] node[yshift=2mm,xshift=2mm] {$t$}	(s)
	edge		[pre]			node[yshift=2.5mm]			{$s$}		(t);
\node[rectangle]	(srts)	[below=of rt]					{$srts$}
	edge		[pre]			node[xshift=2.5mm]			{$s$}		(rt);
\node[rectangle]	(w0)		[below=of srts,yshift=-3mm]	{$w_0$}
	edge		[pre,bend left]	node[xshift=-2.5mm]			{$t$}		(srts)
	edge		[pre,bend right]	node[xshift=2.5mm]			{$r$}		(srts);
\node[rectangle]	(rtstr)	[below=of w0,yshift=-3mm]	{$rtstr$}
	edge		[pre, bend left=15]	node[xshift=-2mm,yshift=-2mm]	{$t$}	(srs)
	edge		[pre, bend right=15]	node[xshift=2mm,yshift=-2mm]	{$r$}	(sts)
	edge		[post,dashed]	node[xshift=2mm]			{$s$}		(w0);
\end{tikzpicture}
%%%%%%%%%%%%%%%%%%%%%%%%%%%%%%%%%%%%%%%%%%%%
%%%%%%%%%%%%%%%%%%%%%%%%%%%%%%%%%%%%%%%%%%%%
\caption{$\Gamma_*$ for $W(A_3)$}
\label{fig:a3trivial}
\end{minipage}%
\begin{minipage}{0.5\textwidth}
\centering
%%%%%%%%%%%%%%%%%%%%%%%%%%%%%%%%%%%%%%%%%%%%
\begin{tikzpicture}
[node distance=0.5cm,%
pre/.style={<-,shorten <=1pt,>=angle 45},%
post/.style={->,shorten >=1pt,>=angle 45}];%
rectangle/.style={inner sep=0pt,minimum size=5mm}];%
\node[rectangle]	(e)								{$e$};
\node[rectangle]	(s)	[above=of e,yshift=2mm]			{$s$}
	edge		[pre,dashed]	node[xshift=2mm]			{$s$}		(e);
\node[rectangle]	(rt)	[below=of e,yshift=-2mm]			{$rt$}
	edge		[pre,bend left]	node[xshift=-2.5mm]			{$r$}		(e)
	edge		[pre,bend right]	node[xshift=2.5mm]			{$t$}		(e);
\node[rectangle]	(srts)	[below=of rt]					{$srts$}
	edge		[pre]			node[xshift=2.5mm]			{$s$}		(rt);
\node[rectangle]	(rsrts) [below=of srts,xshift=-10mm]		{$rsrts$}
	edge		[pre,dashed]	node[xshift=-2mm,yshift=2mm]	{$r$}		(srts);
\node[rectangle]	(rst)	[left=of rsrts]			{$rst$}
	edge		[post]		node[yshift=2.5mm]			{$s$}		(rsrts)
	edge		[pre,bend left=15] node[xshift=-2mm,yshift=2mm] {$r$}	(s);
\node[rectangle]	(tsrts) [below=of srts,xshift=10mm]		{$tsrts$}
	edge		[pre,dashed]	node[xshift=2mm,yshift=2mm]	{$t$}		(srts);
\node[rectangle]	(tsr)	[right=of tsrts]			{$tsr$}
	edge		[post]		node[yshift=2.5mm]			{$s$}		(tsrts)
	edge		[pre,bend right=15] node[xshift=2mm,yshift=2mm] {$t$}	(s);
\node[rectangle]	(w0)	[below=of rsrts,xshift=10mm]		{$w_0$}
	edge		[pre,dashed]	node[xshift=-2mm,yshift=-2mm]	{$t$}		(rsrts)
	edge		[pre,dashed]	node[xshift=2mm,yshift=-2mm]	{$r$}		(tsrts);
\node[rectangle]	(rtstr) [below=of w0,yshift=-3mm]		{$rtstr$}
	edge		[pre,bend left=10] node[xshift=-2mm,yshift=-2mm] {$t$}	(rst)
	edge		[pre,bend right=10] node[xshift=2mm,yshift=-2mm] {$r$}	(tsr)
	edge		[post,dashed]	node[xshift=2mm]			{$s$}		(w0);
\end{tikzpicture}
%%%%%%%%%%%%%%%%%%%%%%%%%%%%%%%%%%%%%%%%%%%%
%%%%%%%%%%%%%%%%%%%%%%%%%%%%%%%%%%%%%%%%%%%%
\caption{$\Gamma_{\#}$ for $W(A_3)$}
\label{fig:a3twisted}
\end{minipage}%
\end{minipage}%
\end{figure}
\noindent
\end{example}

%%%%%%%%%%%%%%%%%%%%%%%%%%%%%%%%%%%%%%%%%%%%

\begin{example}
Suppose 
$W = \spanof{r,s,t}$  with 
$n(r,s)=3$, $n(s,t)=4$, $n(r,t)=2$.  With $w \mapsto w^* = w$, 
$I_*$ is the set of involutions of $W$.  The
corresponding $W$-digraph 
$\Gamma_*$ takes the form shown in 
Figure~\ref{fig:b3}. 
Note $(\Gamma_*)_{\text{rev}} \cong \Gamma_*$.  
\begin{figure}[ht]
\centering
%%%%%%%%%%%%%%%%%%%%%%%%%%%%%%%%%%%%%%%%%%%%
\begin{tikzpicture}
[node distance=0.55cm,%
pre/.style={<-,shorten <=1pt,>=angle 45},%
post/.style={->,shorten >=1pt,>=angle 45}];%
rectangle/.style={inner sep=0pt,minimum size=5mm}];%
\newcommand{\vshift}{25mm}%
\node[rectangle]	(ststw0) 								{${\strut}ststw_0$};
\node[rectangle]	(stsw0) [right=of ststw0]					{${\strut}stsw_0$}
	edge			[pre,dashed]		node[yshift=2.5mm]		{$t$}		(ststw0);
\node[rectangle]	(sts) [above=of ststw0]					{${\strut}sts$}
	edge			[post]			node[xshift=2.5mm]		{$r$}		(ststw0);
\node[rectangle]	(stst) [above=of stsw0]					{${\strut}stst$}
	edge			[pre,dashed]		node[yshift=2.5mm]		{$t$}		(sts)
	edge			[post]			node[xshift=2.5mm]		{$r$}		(stsw0);
\node[rectangle]	(sw0) [below=of ststw0]					{${\strut}sw_0$};
\node[rectangle]	(tstw0) [left=of sw0,xshift=-3mm]			{${\strut}tstw_0$}
	edge			[pre,dashed]	node[xshift=-2mm,yshift=2mm]	{$s$}		(ststw0)
	edge			[post]			node[yshift=2.5mm]		{$t$}		(sw0);
\node[rectangle]	(w0) [below=of stsw0]					{${\strut}w_0$}
	edge			[pre,dashed]		node[yshift=2.5mm]		{$s$}		(sw0);
\node[rectangle]	(tw0) [right=of w0,xshift=3mm]				{${\strut}tw_0$}
	edge			[post,dashed]		node[yshift=2.5mm]		{$t$}		(w0)
	edge			[pre]			node[xshift=2mm,yshift=2mm]	{$s$}		(stsw0);
\node[rectangle]	(t) [above=of tstw0,yshift=\vshift]			{${\strut}t$}
	edge			[post]		node[xshift=-2mm,yshift=-2mm] {$s$}		(sts);
\node[rectangle]	(e) [above=of sw0,yshift=\vshift]			{${\strut}e$}
	edge			[post,dashed]		node[yshift=2.5mm]		{$t$}		(t);
\node[rectangle]	(s) [above=of w0,yshift=\vshift]				{${\strut}s$}
	edge			[pre,dashed]		node[yshift=2.5mm]		{$s$}		(e);
\node[rectangle]	(tst) [above=of tw0,yshift=\vshift]			{${\strut}tst$}
	edge			[post,dashed]	node[xshift=2mm,yshift=-2mm]	{$s$}		(stst)
	edge			[pre]				node[yshift=2.5mm]		{$t$}		(s);
\node[rectangle]	(rt) [above=of t]							{${\strut}rt$}
	edge			[pre,dashed]		node[xshift=2.5mm]		{$r$}		(t);
\node[rectangle]	(rtstrw0) [below=of tstw0]					{${\strut}rtstrw_0$}
	edge			[post]			node[xshift=2.5mm]		{$r$}		(tstw0)
	edge			[pre,bend left=35]	node[xshift=-2.5mm]		{$s$}		(rt);
\node[rectangle]	(r) [above=of e]							{${\strut}r$}
	edge			[post,dashed]		node[yshift=2.5mm]		{$t$}		(rt)
	edge			[pre,dashed]		node[xshift=2.5mm]		{$r$}		(e);
\node[rectangle]	(srs) [above=of s]						{${\strut}srs$}
	edge			[pre]				node[yshift=2.5mm]		{$s$}		(r)
	edge			[pre]				node[xshift=2.5mm]		{$r$}		(s);
\node[rectangle]	(rtstr) [above=of tst]						{${\strut}rtstr$}
	edge			[pre]				node[yshift=2.5mm]		{$t$}		(srs)
	edge			[pre]				node[xshift=-2.5mm]		{$r$}		(tst);
\node[rectangle]	(srsw0) [below=of sw0]					{${\strut}srsw_0$}
	edge			[pre]				node[yshift=2.5mm]		{$t$}		(rtstrw0)
	edge			[post]			node[xshift=2.5mm]		{$r$}		(sw0);
\node[rectangle]	(rw0) [below=of w0]						{${\strut}rw_0$}
	edge			[pre]				node[yshift=2.5mm]		{$s$}		(srsw0)
	edge			[post,dashed]		node[xshift=2.5mm]		{$r$}		(w0);
\node[rectangle]	(rtw0) [below=of tw0]						{${\strut}rtw_0$}
	edge			[post,dashed]		node[yshift=2.5mm]		{$t$}		(rw0)
	edge			[post,dashed]		node[xshift=-2.5mm]		{$r$}		(tw0)
	edge			[pre,bend right=35]	node[xshift=2.5mm]		{$s$}		(rtstr);
\end{tikzpicture}
%%%%%%%%%%%%%%%%%%%%%%%%%%%%%%%%%%%%%%%%%%%%
%%%%%%%%%%%%%%%%%%%%%%%%%%%%%%%%%%%%%%%%%%%%
\caption{$\Gamma_*$ for $W(B_3)$}
\label{fig:b3}
\end{figure}
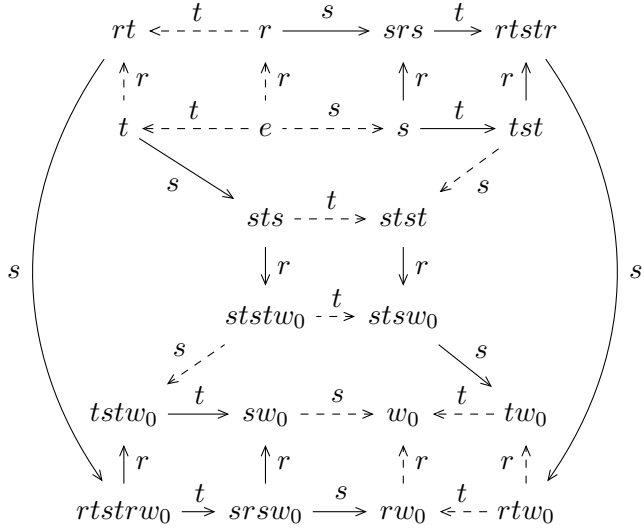 
\end{example}

%%%%%%%%%%%%%%%%%%%%%%%%%%%%%%%%%%%%%%%%%%%%
%%%%%%%%%%%%%%%%%%%%%%%%%%%%%%%%%%%%%%%%%%%%

\section{The proofs of Theorems~\ref{theorem:acyclic}, 
\ref{theorem:linearcharmults}, \ref{theorem:index}, and 
\ref{theorem:equallengths}}
\label{section:acyclic}

Throughout this section $(W,S)$ is a Coxeter system
and $\Gamma$ is a $W$-digraph.  
If 
$\varepsilon = \solidedge{$\alpha$}{$\beta$}{$s$} \in \edges(\Gamma)$ or 
$\varepsilon = \dashededge{$\alpha$}{$\beta$}{$s$} \in \edges(\Gamma)$, 
then we call $ \solidedge{$\alpha$}{$\beta$}{$s$} \in \edges(\Gammaarrow)$ the
{\it image} of $\varepsilon$ in $\Gammaarrow$.  
Clearly there is a directed path from $\alpha$ to $\beta$ in $\Gamma$ if
and only if there is a directed path from $\alpha$ to $\beta$
in $\Gammaarrow$.

Let $H_0$ be the $0$-Hecke algebra 
of $(W,S)$ (see \cite{norton}, or
\cite{bourbaki}, Chapter IV, \S2, Exercise 23, 
with $\lambda_s = -1$, $\mu_s = 0$ for $s \in S$).  
Thus $H_0$ is an associative algebra over
$\rationals$ with 
generating set
$\setof{a_s}{s\in S}$ satisfying
the presentation  
\begin{equation*}
a_s^2 = - a_s
\end{equation*}
for $s \in S$ and
\begin{equation*} 
\overbrace{a_s a_t a_s \cdots}^{n(s,t)}
=
\overbrace{a_t a_s a_t \cdots}^{n(s,t)}
\end{equation*} 
if $s, t \in S$, $n(s,t)<\infinity$.
Also, $H_0$ has basis 
$\setof{a_w}{w\in W}$ with $a_e$ the
identity element of $H_0$ and
\begin{equation}
\label{eq:eqnsw}
a_s a_w
=
\begin{cases}
a_{sw} & \text{if $sw > w$,} \\
- a_{w} & \text{if $sw < w$.}
\end{cases}
\end{equation}
It follows  that
for $x, y \in W$, there is $z \in W$
such that
\begin{equation*}
a_x a_y = \pm a_z
\quad
\text{and}
\quad
\max\set{\ell(x),\ell(y)} \le \ell(z),
\end{equation*}
and 
\begin{equation*}
a_x a_y = a_{xy}
\iff
\ell(x) + \ell(y) = \ell(xy).
\end{equation*}
If $(W,S)$ is finite
and 
$w_0$ is the longest element of $W$,
then 
\begin{equation*}
a_w a_{w_0} = (-1)^{\ell(w)} a_{w_0}   
= a_{w_0} a_w
\end{equation*}
for $w \in W$.

Let $M=M(\Gamma)$ 
be the module afforded by $\Gamma$, so
$M$ has basis $X = \vertices(\Gamma)$
over $\rationals(u)$.  
Let $M_0$ be the $\rationals$-subspace of $M$
 with basis $X$.  
For $s \in S$, define a $\rationals$-linear operator
$(\tau_s)_0$ on $M_0$ by
\begin{equation*}
(\tau_s)_0(\alpha)
=
\begin{cases}
\beta & \text{if  }\solidedge{$\alpha$}{$\beta$}{$s$} \in \edges(\Gammaarrow), \cr
-\alpha & \text{if $\alpha$ is a sink in $\Gamma_s$}. \cr
\end{cases}
\end{equation*}
Notice that by \eqref{eq:taudefinition}, 
$(\tau_s)_0(\alpha)$ can be obtained by
replacing the coefficients of the 
image $\tau_s(\alpha)$ expressed as
a linear combination of the elements of $X$ with their
values at $u = 0$.
Since in $\gl(M)$ we have
\begin{equation*}
(\tau_s - u^2)(\tau_s + 1) = 0
\qquad\text{and}\qquad
\overbrace{\tau_s \tau_t \tau_s \cdots}^{n(s,t)}
=
\overbrace{\tau_t \tau_s \tau_t \cdots}^{n(s,t)},
\end{equation*}
it follows that in $\gl(M_0)$ we have
\begin{equation*}
\left((\tau_s)_0)\right)^2 = -(\tau_s)_0
\qquad\text{and}\qquad
\overbrace{(\tau_s)_0 (\tau_t)_0 (\tau_s)_0 \cdots}^{n(s,t)}
=
\overbrace{(\tau_t)_0 (\tau_s)_0 (\tau_t)_0 \cdots}^{n(s,t)}
\end{equation*}
if $n(s,t) < \infinity$.  
Hence $a_s \mapsto (\tau_s)_0$ defines a representation
$\rho_0 : H_0 \rightarrow \gl(M_0)$, giving $M_0$ the
structure of an $H_0$-module.   
In particular, for $\alpha\in \vertices(\Gamma)$, 
\begin{equation}
\label{eq:h0action}
a_s \alpha
=
\begin{cases}
\beta & \text{if } \solidedge{$\alpha$}{$\beta$}{$s$} \in \edges(\Gammaarrow), \cr
- \alpha & \text{if $\alpha$ is a sink in $\Gamma_{\set{s}}$}. \cr
\end{cases} 
\end{equation}

%%%%%%%%%%%%%%%%%%%%%%%%%%%%%%%%%%%%%%%%%%%%

\begin{lemma}
\label{lemma:h0paths}
Assume $(W,S)$ is a  
Coxeter system, $\Gamma$ is a $W$-digraph, 
$X = \vertices(\Gamma)$, and
$H_0$ acts on $M_0$ as described above.
Then the following hold.
\begin{enumerate}[{\upshape(i)}]
\item
If $\alpha \in X$ and $w \in W$, then
$a_w \alpha \in X$ or $-a_w \alpha \in X$. 
\item
If $\beta \in X$, then there exists some $w \in W$
such that $\beta = \pm a_w \alpha$ if and only
if there is a
directed path from $\alpha$ to $\beta$ in $\Gamma$.
\end{enumerate}
\end{lemma}

\begin{proof}
Since $a_w = a_{s_1} a_{s_2} \cdots a_{s_\ell}$ if
$s_1 s_2 \dots s_\ell$ is a reduced expression for
$w \in W$ by \eqref{eq:eqnsw}, an easy
induction argument based on \eqref{eq:h0action} 
establishes (i).  

For (ii), we can argue with $\Gammaarrow$ in place of $\Gamma$.  
Suppose $\beta \in X$ and there is some directed
path 
\begin{equation}
\label{eq:path}
\begin{split}
%%%%%%%%%%%%%%%%%%%%%%%%%%%%%%%%%%%%%%%%%%%%
\begin{tikzpicture}
[node distance=0.8cm,%
pre/.style={<-,shorten <=1pt,>=angle 45},%
post/.style={->,shorten >=1pt,>=angle 45}];%
rectangle/.style={inner sep=0pt,minimum size=5mm}];%
\node[rectangle]	(node0)						{${\strut}\gamma_{0}$};
\node[rectangle]	(node1) [right=of node0]			{${\strut}\gamma_{1}$}
	edge [pre]	node[yshift=3mm]	{$s_1$}		(node0);
\node[rectangle]	(node2) [right=of node1]		{${\strut}\gamma_{2}$}
	edge [pre]		node[yshift=3mm]	{$s_2$}	(node1);
\node[rectangle]	(node3) [right=of node2]	{${\strut}\cdots$}
	edge [pre]		node[yshift=3mm]	{$s_3$}		(node2);
\node[rectangle]	(node4) [right=of node3]	{${\strut}\gamma_{k-1}$}
	edge [pre]		node[yshift=3mm]	{$s_{k-1}$}	(node3);
\node[rectangle]	(node5) [right=of node4]				{${\strut}\gamma_{k}$}
	edge [pre]			node[yshift=3mm]	{$s_{k}$}	(node4);
\end{tikzpicture}
%%%%%%%%%%%%%%%%%%%%%%%%%%%%%%%%%%%%%%%%%%%%
%%%%%%%%%%%%%%%%%%%%%%%%%%%%%%%%%%%%%%%%%%%%
\end{split}
\end{equation}
in $\Gammaarrow$ 
with $\gamma_0=\alpha$, $\gamma_k = \beta$.
Define $y \in W$ by
$\pm a_y = a_{s_k} a_{s_{k-1}} \cdots a_{s_2} a_{s_1}$.
Then  
\[
\pm a_y \alpha = 
a_{s_k} a_{s_{k-1}} \cdots a_{s_2} a_{s_1} \gamma_0
= \gamma_k = \beta.
\]
Conversely, assume 
$\beta = \pm a_w \alpha \in X$, where $w \in W$.  Let 
$w = t_{k} t_{k-1} \cdots t_2 t_1$
be a reduced expression for $w$ as a product of generators
$t_k, \dots, t_1 \in S$, so
$a_w = a_{t_k} \cdots a_{t_2} a_{t_1}$ by \eqref{eq:eqnsw}.
Put $\delta_0 = \alpha$ and
$\delta_{j} = a_{t_j}\delta_{j-1}$ for $1 \le j \le k$, so
$\beta = \pm \delta_k$.  
By (i), there are  $\varepsilon_j \in \set{-1,1}$ such that
$\alpha_j = \varepsilon_j \delta_j \in X$ for $0 \le j \le k$.
Then for $1 \le j \le k$, 
$\alpha_{j-1} \ne \alpha_j$ if and only if
$\solidedge{$\alpha_{j-1}$}{$\alpha_{j}$}{$t_j$} \in \edges(\Gammaarrow)$ .
If $0 < j_1 < j_2 < \cdots < j_\ell$ are the
values of $j$, $1 \le j \le k$, for which $\alpha_{j-1} \ne \alpha_j$, then
\[
%%%%%%%%%%%%%%%%%%%%%%%%%%%%%%%%%%%%%%%%%%%%
\begin{tikzpicture}
[node distance=0.8cm,%
pre/.style={<-,shorten <=1pt,>=angle 45},%
post/.style={->,shorten >=1pt,>=angle 45}];%
rectangle/.style={inner sep=0pt,minimum size=5mm}];%
\node[rectangle]	(node0)						{${\strut}\alpha_{0}$};
\node[rectangle]	(node1) [right=of node0]			{${\strut}\alpha_{j_1}$}
	edge [pre]	node[yshift=3.5mm]	{${\strut}t_{j_1}$}		(node0);
\node[rectangle]	(node2) [right=of node1]		{${\strut}\alpha_{j_2}$}
	edge [pre]		node[yshift=3.5mm]	{${\strut}t_{j_2}$}	(node1);
\node[rectangle]	(node3) [right=of node2]	{${\strut}\cdots$}
	edge [pre]		node[yshift=3.5mm]	{${\strut}t_{j_3}$}		(node2);
\node[rectangle]	(node4) [right=of node3]	{${\strut}\alpha_{j_{\ell-1}}$}
	edge [pre]		node[yshift=3.5mm]	{${\strut}t_{t_{\ell-1}}$}	(node3);
\node[rectangle]	(node5) [right=of node4]		{${\strut}\alpha_{j_{\ell}}$}
	edge [pre]			node[yshift=3.5mm]	{${\strut}t_{j_\ell}$}	(node4);
\end{tikzpicture}
%%%%%%%%%%%%%%%%%%%%%%%%%%%%%%%%%%%%%%%%%%%%
%%%%%%%%%%%%%%%%%%%%%%%%%%%%%%%%%%%%%%%%%%%%
\]
is a directed path in $\Gammaarrow$ from $\alpha$ to $\beta$.
Thus (ii) holds.
\end{proof}

%%%%%%%%%%%%%%%%%%%%%%%%%%%%%%%%%%%%%%%%%%%%

\begin{lemma}
\label{lemma:h0sinks}
Assume $(W,S)$ is a  
Coxeter system, $\Gamma$ is a $W$-digraph, 
$X = \vertices(\Gamma)$, and
$H_0$ acts on $M_0$ as in \eqref{eq:h0action}.
For $\omega \in X$, the following are equivalent. 
\begin{enumerate}[{\upshape(i)}]
\item
$\omega$ is a sink in $\Gamma$.
\item
$a_s \omega = - \omega$ for all $s \in S$.
\item
$a_w \omega = (-1)^{\ell(w)} \omega$ for all $w \in W$.
\end{enumerate}
Moreover, if $(W,S)$ is finite, then (i)--(iii) are equivalent to
\begin{enumerate}[{\upshape(i)}]
\setcounter{enumi}{3}
\item
$\omega = \pm a_{w_0} \alpha$ for some $\alpha \in X$.
\end{enumerate}
\end{lemma}

\begin{proof}
If $\omega$ is a sink in $\Gamma$, then $a_s \omega = - \omega$
for all $s\in S$ by \eqref{eq:h0action}.  Thus (i) implies (ii).

Assume $a_s \omega = - \omega$ for all $s \in S$ 
and $w \in W$ has reduced expression
$w = s_1 s_2 \cdots s_k$.  Then by \eqref{eq:eqnsw}, 
\begin{equation*}
a_w \omega = a_{s_1} a_{s_2} \cdots a_{s_k} \omega
= (-1)^{k} \omega = (-1)^{\ell(w)} \omega.
\end{equation*}
Hence (ii) imples (iii).

Suppose $a_w \omega = (-1)^{\ell(w)} \omega$ for 
all $w \in W$. Then 
$a_s \omega = - \omega$ for all $s\in S$, and thus
$\omega$ must be a sink in $\Gamma$ by \eqref{eq:h0action}.
Hence (iii) implies (i).

Suppose 
$(W,S)$ is finite and
$\omega = \varepsilon a_{w_0} \alpha$, where
$\alpha \in X$ and $\varepsilon \in \set{-1,1}$.  
Then  
\begin{equation*}
a_s \omega = a_s (\varepsilon a_{w_0} \alpha) =
\varepsilon (a_s a_{w_0}) \alpha = - \varepsilon a_{w_0} \alpha = - \omega
\end{equation*}
for any $s \in S$, 
and so $\omega$ is a sink in
$\Gamma$.  Conversely, if 
$\omega$ is a sink in $\Gamma$, then
$a_{w_0} \omega = (-1)^N \omega = \pm \omega$ by \eqref{eq:h0action},
where $N = \ell(w_0$), 
and thus $\omega = \pm a_{w_0} \omega$.
Hence (i) and (iv) are equivalent.  
\end{proof}

%%%%%%%%%%%%%%%%%%%%%%%%%%%%%%%%%%%%%%%%%%%%

Define relations $\equiv_{s,t}$ and $\equiv$ on the set of
directed paths in $\Gamma$ as follows.
If $\pi_1$ and $\pi_2$ are directed paths in $\Gamma$ 
and $s, t \in S$ satisfy $1 < n(s,t) < \infinity$, 
then $\pi_1 \equiv_{s,t} \pi_2$ if there is some 
connected component 
$C$ of $\Gamma_{\set{s,t}}$ such that $\pi_1$ and $\pi_2$
both pass through the source $\sigma$ and the sink
$\omega$ of $C$, and $\pi_2$ can be obtained from
$\pi_1$ by replacing one of the directed paths from
$\sigma$ to $\omega$ in $\Gamma_{\set{s,t}}$ by the other.
Let $\equiv$ be the equivalence relation on directed paths
in $\Gamma$ generated by the relations $\equiv_{s,t}$ 
for $s, t \in S$, $1 < n(s,t) < \infinity$.  
Similar  relations, also denoted
$\equiv_{s,t}$ and $\equiv$, can be defined for
directed paths in $\Gammaarrow$.  It is clear that two directed paths
in $\Gamma$ are in the same $\equiv$
equivalence class if and only if their images in
$\Gammaarrow$ are in the same 
$\equiv$ equivalence class.

%%%%%%%%%%%%%%%%%%%%%%%%%%%%%%%%%%%%%%%%%%%%

For $\alpha \in \vertices(\Gamma)$, denote by
$\closedray{\alpha}$ the set of all $\beta \in \vertices(\Gamma)$
such that there exists a directed path in $\Gamma$ from
$\alpha$ to $\beta$.  Clearly if $\beta \in \closedray{\alpha}$ and
$\gamma \in \closedray{\beta}$, then
$\gamma \in \closedray{\alpha}$.  
For $\beta \in \closedray{\alpha}$, let $\mu(\alpha,\beta)$ be 
the minimum number of edges in a directed path from $\alpha$ to 
$\beta$ (with $\mu(\alpha,\alpha) = 0$). 

%%%%%%%%%%%%%%%%%%%%%%%%%%%%%%%%%%%%%%%%%%%%

\begin{lemma}
\label{lemma:locallyfinite}
Suppose $(W,S)$ is a Coxeter system such that
$n(s,t) < \infinity$ for all $s,t \in S$, and
$\Gamma$ is a  $W$-digraph with source $\sigma$.
\begin{enumerate}[{\upshape(i)}]
\item
If $\alpha \in \closedray{\sigma}$, then
any two directed paths
from $\sigma$ to $\alpha$ are in the same 
$\equiv$-equivalence class.
\item
If $\alpha \in \closedray{\sigma}$, 
$\zeta \in \vertices(\Gamma)$, and 
$\alpha \in \closedray{\zeta}$, then
$\zeta \in \closedray{\sigma}$.
\item
If $\Gamma$ is connected, then 
$\vertices(\Gamma) = \closedray{\sigma}$.
\end{enumerate}
\end{lemma}

\begin{proof}
We can argue with $\Gammaarrow$ in place of $\Gamma$. 
We prove (i) and (ii) simultaneously by induction on $\mu(\sigma,\alpha)$.  
If $\mu(\sigma,\alpha) = 0$, then $\alpha = \sigma$, so (i) holds
because the only directed path from $\sigma$ to $\sigma$ is the
empty path because $\sigma$ is a source.  
Also, if $\alpha = \sigma \in \closedray{\gamma}$, 
then there is a directed path from $\gamma$ to
$\sigma$, so the path must be empty and 
$\gamma = \sigma \in \closedray{\sigma}$, and thus (ii) holds.

Suppose $\mu(\sigma,\alpha) = k > 0$ and (i) and (ii) hold with $\beta$
in place of $\alpha$ whenever $\beta \in \closedray{\sigma}$ and
$\mu(\sigma,\beta) < k$.  Let $\pi_1$ be some directed path from 
$\sigma$ to $\alpha$ with $k$ edges, and let $\pi_2$ be an arbitrary
directed path from $\sigma$ to $\alpha$.  
For $j = 1, 2$, let $\varepsilon_j \in \edges(\Gammaarrow)$
be the last edge of $\pi_j$ and let
$\rho_j$ be the remainder of the path $\pi_j$, so
$\pi_j = \rho_j \varepsilon_j$, where juxtaposition indicates 
concatination of paths. Thus $\varepsilon_1$ takes the
form \solidedge{$\beta$}{$\alpha$}{$s$} for some $s\in S$ and
$\beta \in \vertices(\Gamma)$ with 
$\beta \in \closedray{\sigma}$ and $\mu(\sigma,\beta) = k-1$.  
Also, $\varepsilon_2$ has the form 
\solidedge{$\gamma$}{$\alpha$}{$t$}
for some $t\in S$, $\gamma \in \vertices(\Gamma)$.  
If $t = s$, then $\gamma = \beta$, so $\rho_1\equiv \rho_2$ by
(i) applied to $\beta$, and thus $\pi_1 \equiv \pi_2$ as desired.
Suppose $t \ne s$.  Let $\tau$ be the source of the connected component
$C$ of $(\Gammaarrow)_{\set{s,t}}$ whose sink is $\alpha$. 
(Note that $C$ has a unique source by the classification of possible
connected
components of $\Gamma_{\set{s,t}}$ given in Theorem~\ref{theorem:main}.)
Let $\nu_1$ ($\nu_2$) be the directed path in $C$
from $\tau$ to $\beta$ ($\gamma$, respectively), so
$\nu_1 \varepsilon_1 \equiv_{s,t} \nu_2 \varepsilon_2$.  Since
$\beta \in \closedray{\tau}$, we have $\tau \in \closedray{\sigma}$ by
(ii) applied to $\beta$, and so there is some directed path $\rho$ from
$\sigma$ to $\tau$ in $\Gammaarrow$.  
(See Figure~\ref{fig:newclaimfig1}, in which edges represent 
directed paths in $\Gammaarrow$.)
\begin{figure}[ht]
\centering
%%%%%%%%%%%%%%%%%%%%%%%%%%%%%%%%%%%%%%%%%%%%
\begin{tikzpicture}
[node distance=0.6cm,%
pre/.style={<-,shorten <=1pt,>=angle 45},%
post/.style={->,shorten >=1pt,>=angle 45}];%
rectangle/.style={inner sep=0pt,minimum size=5mm}];%
\node[rectangle] (sigma)								{${\strut}\sigma$};
\node[rectangle] (tau) [right=of sigma,xshift=15mm]			{${\strut}\tau$}
	edge[pre]		node[yshift=2mm,xshift=4mm] {$\rho$}	(sigma);
\node[rectangle] (ghost) [right=of tau,xshift=3mm]			{$\strut$};
\node[rectangle] (beta) [above=of ghost]					{$\beta$}
	edge[pre]		node[xshift=2mm,yshift=-2mm] {$\nu_1$}		(tau)
	edge[pre]		node[xshift=-2mm,yshift=2mm] {$\rho_1$}	(sigma);
\node[rectangle] (gamma) [below=of ghost]				{${\strut}\gamma$}
	edge[pre]		node[xshift=2mm,yshift=2mm] {$\nu_2$}		(tau)
	edge[pre]		node[xshift=-2mm,yshift=-2mm] {$\rho_2$}	(sigma);
\node[rectangle] (alpha) [right=of ghost,xshift=3mm]		{${\strut}\alpha$}
	edge[pre]		node[xshift=2mm,yshift=2mm] {$\varepsilon_1$} (beta)
	edge[pre]		node[xshift=2mm,yshift=-2mm] {$\varepsilon_2$} (gamma);
\end{tikzpicture}
%%%%%%%%%%%%%%%%%%%%%%%%%%%%%%%%%%%%%%%%%%%%
%%%%%%%%%%%%%%%%%%%%%%%%%%%%%%%%%%%%%%%%%%%%
\caption{}
\label{fig:newclaimfig1}
\end{figure}
We have
$\rho_1 \equiv \rho \nu_1$ by (i) applied to $\beta$, and thus
$\rho \nu_1$ has $k-1$ edges.
Since $\nu_1$ and $\nu_2$ have the same number of edges, it
follows that $\rho \nu_2$ also has $k-1$ edges, and thus 
$\mu(\sigma,\gamma) \le k-1$.  Hence by (i) applied to 
$\gamma$, we also have 
$\rho \nu_2 \equiv \rho_2$.  
Thus
\begin{equation*}
\begin{split}
\pi_1 
& =  \rho_1 \varepsilon_1
\equiv (\rho \nu_1) \varepsilon_1
= \rho ( \nu_1 \varepsilon_1 ) \cr
 & \equiv \rho ( \nu_2 \varepsilon_2)
 = (\rho \nu_2 ) \varepsilon_2
\equiv \rho_2 \varepsilon_2 = \pi_2 . 
 \end{split}
 \end{equation*}
Therefore (i) holds for $\alpha$.

Now suppose
 $\alpha \in \closedray{\delta}$.  
Let $\psi$ be a directed path from $\delta$
to $\alpha$.  Write $\psi = \psi_0 \varepsilon_0$, where
$\varepsilon_0 \in \edges(\Gammaarrow)$ is the last edge of $\psi$, 
so $\varepsilon_0$ has the form
\solidedge{$\phi$}{$\alpha$}{$r$} for some $r \in S$,
$\phi \in \vertices(\Gamma)$.
If $r = s$, then $\beta = \phi \in \closedray{\delta}$, and hence
$\delta \in \closedray{\sigma}$ by (ii) applied to $\beta$.  
Assume $r \ne s$.  Let $\kappa$ be the source
of the connected component of $(\Gammaarrow)_{\set{r,s}}$ 
whose sink is $\alpha$.  
(See Figure~\ref{fig:newequivdiag2}, in which the edges 
represent directed paths in $\Gammaarrow$.) 
There exists some directed path from $\sigma$
to $\kappa$ by (ii) 
\begin{figure}[ht]
\centering
%%%%%%%%%%%%%%%%%%%%%%%%%%%%%%%%%%%%%%%%%%%%
\begin{tikzpicture}
[node distance=0.6cm,%
pre/.style={<-,shorten <=1pt,>=angle 45},%
post/.style={->,shorten >=1pt,>=angle 45}];%
rectangle/.style={inner sep=0pt,minimum size=5mm}];%
\node[rectangle] (sigma)								{${\strut}\sigma$};
\node[rectangle] (delta) [below=of sigma,xshift=12mm]		{${\strut}\delta$}
	edge[pre,dashed]	node[xshift=-2mm,yshift=-2mm] {?}	(sigma);
\node[rectangle] (kappa) [right=of sigma,xshift=15mm]		{${\strut}\kappa$}
	edge[pre]										(sigma);
\node[rectangle] (ghost) [right=of tau,xshift=3mm]			{$\strut$};
\node[rectangle] (beta) [above=of ghost]					{$\beta$}
	edge[pre]										(kappa)
	edge[pre]		node[xshift=-2mm,yshift=2mm] {$\rho_1$}	(sigma);
\node[rectangle] (phi) [below=of ghost]					{${\strut}\phi$}
	edge[pre]		node[yshift=-2mm] {$\psi_0$}			(delta)
	edge[pre]										(kappa);
\node[rectangle] (alpha) [right=of ghost,xshift=3mm]		{${\strut}\alpha$}
	edge[pre]		node[xshift=2mm,yshift=2mm] {$\varepsilon_1$} (beta)
	edge[pre]		node[xshift=2mm,yshift=-2mm] {$\varepsilon_0$} (phi);
\end{tikzpicture}
%%%%%%%%%%%%%%%%%%%%%%%%%%%%%%%%%%%%%%%%%%%%
%%%%%%%%%%%%%%%%%%%%%%%%%%%%%%%%%%%%%%%%%%%%
\caption{}
\label{fig:newequivdiag2}
\end{figure}
applied to $\beta$.
The argument given above for $\gamma$ applies to show that
$\mu(\sigma,\phi) \le k-1$.  Since 
$\phi \in \closedray{\delta}$, it follows that
$\delta \in \closedray{\sigma}$ by (ii) applied to $\phi$.
Hence (ii) holds for $\alpha$, so the proof of (i) and (ii) is complete. 

Finally, suppose $\Gamma$ is connected.
For $\alpha\in \vertices(\Gamma)$, let 
 $\delta(\alpha)$  be the minimal number of edges in
a path in $\Gammaundir$ from $\sigma$ to $\alpha$.  
We prove $\alpha \in \closedray{\sigma}$ for all
 $\alpha\in\vertices(\Gamma)$
by induction on $\delta(\alpha)$.  
If $\delta(\alpha) = 0$, then $\alpha = \sigma  \in \closedray{\sigma}$.
Suppose 
$\delta(\alpha) = \ell > 0$ and $\gamma  \in \closedray{\sigma}$ whenever
$\gamma \in \vertices(\Gamma)$ and $\delta(\gamma) < \ell$.  
Let $\beta$ ----- $\alpha$
 be the last edge of
a path in $\Gammaundir$ from $\sigma$ to $\alpha$ of length $\ell$, so
$\delta(\beta) = \ell-1$ and $\beta \in \closedray{\sigma}$.  
If 
$\reversedsolidedge{$\alpha$}{$\beta$}{$s$} \in \edges(\Gammaarrow)$ for some
$s \in S$, then 
$\alpha \in \closedray{\sigma}$ because
$\beta \in \closedray{\sigma}$ and $\alpha \in \closedray{\beta}$.  
On the other hand, if
$\solidedge{$\alpha$}{$\beta$}{$s$} \in \edges(\Gammaarrow)$, 
then $\beta \in \closedray{\alpha}$, and so $\alpha \in \closedray{\sigma}$
 by (ii) applied to $\beta$.
Hence  $\alpha \in \closedray{\sigma}$
for all $\alpha\in \vertices(\Gamma)$, and 
therefore $\vertices(\Gamma) = \closedray{\sigma}$.
This completes the proof.
\end{proof}

%%%%%%%%%%%%%%%%%%%%%%%%%%%%%%%%%%%%%%%%%%%%

\begin{example}
\label{example:inequiv}
Let $W = W(A_3) = \spanof{r,s,t} $, with
$n(r,s) = 3 = n(s,t)$, $n(r,t)=2$, and let $\Gamma$ be as
in Figure~\ref{fig:nonequiv}. 
The directed paths
\solidrightsolidright{$\alpha_2$}{$\alpha_3$}{$s$}{$\beta_3$}{$r$}
and
\solidrightsolidright{$\alpha_2$}{$\beta_2$}{$t$}{$\beta_3$}{$s$}
from $\alpha_2$ to $\beta_3$ are not in the same
$\equiv$-equivalence class (even though adjoining the edge
\solidedge{$\alpha_1$}{$\alpha_2$}{$r$} to both does produce two 
equivalent paths).
Therefore the conclusion of
Lemma~\ref{lemma:locallyfinite}(i) does not apply
to arbitrary directed paths in a $W$-digraph.
\begin{figure}[ht]
\centering
%%%%%%%%%%%%%%%%%%%%%%%%%%%%%%%%%%%%%%%%%%%%
\begin{tikzpicture}
[node distance=0.8cm,%
pre/.style={<-,shorten <=1pt,>=angle 45},%
post/.style={->,shorten >=1pt,>=angle 45}];%
rectangle/.style={inner sep=0pt,minimum size=5mm}];%
\tikzstyle{mydot}=[circle,on grid,draw=black,fill=black,inner sep=0pt,minimum size=1mm];%
\node[rectangle] (a1)									{${\strut}\alpha_1$};
\node[rectangle] (a2) [right=of a1]							{${\strut}\alpha_2$}
	edge [pre]				node[yshift=2.5mm]	{$r$}			(a1);
\node[rectangle] (a3) [right=of a2]							{${\strut}\alpha_3$}
	edge [pre]				node[yshift=2.5mm]	{$s$}			(a2);
\node[rectangle] (a4) [right=of a3,xshift=3mm]					{${\strut}\alpha_4$}
	edge [pre]				node[yshift=2.5mm]	{$t$}			(a3);
\node[rectangle] (b1) [below=of a1,yshift=-3mm]				{${\strut}\beta_1$}
	edge [pre,bend left=20]	node[xshift=-2.5mm]	{$s$}			(a1)
	edge [pre,bend right=20]	node[xshift=2.5mm]	{$t$}			(a1);
\node[rectangle] (b2) [right=of b1]							{${\strut}\beta_2$}
	edge [pre]				node[yshift=2.5mm]	{$r$}			(b1)
	edge [pre]				node[xshift=2.5mm]	{$t$}			(a2);
\node[rectangle] (b3) [right=of b2]							{${\strut}\beta_3$}
	edge [pre]				node[yshift=2.5mm]	{$s$}			(b2)
	edge [pre]				node[xshift=2.5mm]	{$r$}			(a3);
\node[rectangle] (b4) [below=of a4,yshift=-3mm]				{${\strut}\beta_4$}
	edge [pre]				node[yshift=2.5mm]	{$t$}			(b3)
	edge [pre,bend left=20]	node[xshift=-2.5mm]	{$r$}			(a4)
	edge [pre,bend right=20]	node[xshift=2.5mm]	{$s$}			(a4);
\end{tikzpicture}
%%%%%%%%%%%%%%%%%%%%%%%%%%%%%%%%%%%%%%%%%%%%
%%%%%%%%%%%%%%%%%%%%%%%%%%%%%%%%%%%%%%%%%%%%
\caption{Digraph for Example~\ref{example:inequiv}}
\label{fig:nonequiv}
\end{figure}
\end{example}
  
%%%%%%%%%%%%%%%%%%%%%%%%%%%%%%%%%%%%%%%%%%%%

We now prove Theorems~\ref{theorem:acyclic}, 
\ref{theorem:linearcharmults}, 
\ref{theorem:index}, 
and \ref{theorem:equallengths}.  

\begin{proof}[Proof of Theorem~\ref{theorem:acyclic}]
Assume $n(s,t) < \infinity$ for all $s, t \in S$ and
$\Gamma$ is a connected $W$-digraph.
Since $\Gammarev$ is also a connected $W$-digraph
by Corollary~\ref{corollary:reversed}, it is enough to
prove the assertions involving sources.
Suppose $\sigma$ is a source of $\Gamma$.
Then $\vertices(\Gamma) = \closedray{\sigma}$ by
Lemma~\ref{lemma:locallyfinite} (iii).  Hence if
$\gamma \ne \sigma$ is a vertex of $\Gamma$, there
must be some nonempty directed path in $\Gamma$
from $\sigma$ to $\gamma$, and so $\gamma$ 
cannot be a source.  Thus $\sigma$ is the unique source
of $\Gamma$, so part (i) of the theorem holds.

Suppose $\Gamma$ has source $\sigma$ but  is not
acyclic.  Let $\alpha \in \vertices(\Gamma)$ be
 contained in a nonempty directed circuit $\rho$ in $\Gamma$.
Since $\alpha \in \closedray{\sigma}$, there is some directed
path $\pi$ from $\sigma$ to $\alpha$.
Then the directed paths $\pi$ and $\pi \rho$ from $\sigma$ to
$\alpha$ are in different $\equiv$ equivalence classes 
because their lengths are different, contradicting 
Lemma~\ref{lemma:locallyfinite} (i).  Thus part (ii) of the theorem holds.

Finally, assume $(W,S)$ is finite.
By Lemma~\ref{lemma:h0sinks}, 
$\Gammarev$ 
has a sink.  Thus 
$\Gamma$ has a source, so part (iii) of the
theorem holds.
\end{proof}

%%%%%%%%%%%%%%%%%%%%%%%%%%%%%%%%%%%%%%%%%%%%

\begin{proof}[Proof of Theorem~\ref{theorem:linearcharmults}]
Assume $\vertices(\Gamma)$ is finite.
For a linear character $\lambda$ of $H$, it is easily seen that
$M(\Gamma)_\lambda$ is the direct sum of $M(C)_\lambda$
as $C$ ranges over the connected components of $\Gamma$.
By Lemma~\ref{lemma:eigenvalues}(i), 
if $C$ is a connected component of $\Gamma$, 
then $v \in M(C)_{\ind}$ if and only if 
$v$ is a scalar multiple of $\sum_{\alpha \in \vertices(C)} \alpha$.  
Thus (i) holds.

Suppose now that $n(s,t) < \infinity$ for $s, t \in S$.
Let $C$ be a connected component of $\Gamma$.  
Assume $C$ is acyclic.
Then since $\vertices(C)$ is finite, 
there must be a source $\sigma$ in $C$, and this source 
is unique by Theorem~\ref{theorem:acyclic}(i).  
Assign to each solid edge in $C$ the weight
$-1/u^2$, and to each dashed edge in $C$ assign
the weight $-(u+1)/(u^2-u)$. 
For $\alpha \in \vertices(C)$, 
let $\mu_\alpha$ be the product of the weights of the edges
of any directed path from $\sigma$ to $\alpha$ in $C$: 
$\mu_\alpha$ is well-defined by
Lemma~\ref{lemma:locallyfinite}(i) since such products are constant
on $\equiv$-equivalence classes. 
If \solidedge{$\alpha$}{$\beta$}{$s$} is an edge of $C$, then
$\mu_\beta = - \mu_\alpha / u^2$, while if
\dashededge{$\alpha$}{$\beta$}{$s$} is an edge of $C$, then
$\mu_\beta = -(u+1) \mu_\alpha / (u^2-u)$.
By Lemma~\ref{lemma:eigenvalues}(ii),
$v \in M(C)_{\sgn}$ 
if and only if $v$ is
a scalar multiple of $\sum_{\alpha \in \vertices(C)} \mu_\alpha \alpha$.
Therefore $\dim M(C)_{\sgn} = 1$.  

Conversely, suppose  
$v = \sum_{\alpha \in \vertices(C)} \nu_\alpha \alpha \in M(C)_{\sgn}$
is nonzero.
Since at least one of
the coefficients $\nu_\alpha$ is nonzero and $C$ is connected, 
all of the coefficients $\nu_\alpha$ 
are nonzero by Lemma~\ref{lemma:eigenvalues}(ii).  
Assume 
\[
%%%%%%%%%%%%%%%%%%%%%%%%%%%%%%%%%%%%%%%%%%%%
\begin{tikzpicture}
[node distance=0.8cm,%
pre/.style={<-,shorten <=1pt,>=angle 45},%
post/.style={->,shorten >=1pt,>=angle 45}];%
rectangle/.style={inner sep=0pt,minimum size=5mm}];%
\node[rectangle]	(node0)						{${\strut}\gamma_{0}$};
\node[rectangle]	(node1) [right=of node0]			{${\strut}\gamma_{1}$}
	edge [pre]	node[yshift=3mm]	{$s_1$}		(node0);
\node[rectangle]	(node2) [right=of node1]		{${\strut}\gamma_{2}$}
	edge [pre]		node[yshift=3mm]	{$s_2$}	(node1);
\node[rectangle]	(node3) [right=of node2]	{${\strut}\cdots$}
	edge [pre]		node[yshift=3mm]	{$s_3$}		(node2);
\node[rectangle]	(node4) [right=of node3]	{${\strut}\gamma_{k-1}$}
	edge [pre]		node[yshift=3mm]	{$s_{k-1}$}	(node3);
\node[rectangle]	(node5) [right=of node4]				{${\strut}\gamma_{k}$}
	edge [pre]			node[yshift=3mm]	{$s_{k}$}	(node4);
\end{tikzpicture}
%%%%%%%%%%%%%%%%%%%%%%%%%%%%%%%%%%%%%%%%%%%%
%%%%%%%%%%%%%%%%%%%%%%%%%%%%%%%%%%%%%%%%%%%%
\]
is a directed path in $C_{\to}$, where
$k > 0$, $s_1, s_2, \dots, s_k \in S$.
For $1 \le j \le k$ 
we have
\begin{equation*}
\nu_{\gamma_j}
=
\begin{cases}
- \displaystyle{\frac{1}{u^2}} \nu_{\gamma_{j-1}} 
& \text{if } \solidedge{$\gamma_{j-1}$}{$\gamma_j$}{$s_j$}
  \in \edges(C), \cr
- \displaystyle{\frac{u+1}{u^2-u}} \nu_{\gamma_{j-1}}
& \text{if } \dashededge{$\gamma_{j-1}$}{$\gamma_j$}{$s_j$}
  \in \edges(C), \cr
\end{cases}
\end{equation*}
If  $\gamma_0 = \gamma_k$, then
\[
\nu_{\gamma_0}
=
\nu_{\gamma_k}
=
\nu_{\gamma_0} 
  \prod_{j = 1}^{k} \frac{\nu_{\gamma_j}}{\nu_{\gamma_{j-1}}},
\]
and therefore
$\prod_{j = 1}^{k} (\nu_{\gamma_j} / \nu_{\gamma_{j-1}})= 1$, 
which is impossible since the product is equal to 
$\pm u^{-2(k-j)} (u+1)^{j} (u^2-u)^{-j}$ for
some $j$ with 
$0 \le j \le k$.  Therefore $C$ is acyclic, so the
proof of (ii) is complete.
\end{proof}

%%%%%%%%%%%%%%%%%%%%%%%%%%%%%%%%%%%%%%%%%%%%

\begin{proof}[Proof of Theorem~\ref{theorem:index}]
Assume  $(W,S)$ is a finite Coxeter system, $\Gamma$ is a
connected $W$-digraph,
and $J \subseteq S$. 
Let
\[
\Gamma_J = \bigcup_{i \in I} C_i
\]
be the decomposition of $\Gamma_J$ into its connected components, 
indexed by some set $I$.
Let $\sigma_i \in \vertices(\Gamma)$ be the source of $C_i$.  
If $\sigma$ is the source of $\Gamma$, then
in $M(\Gamma)_0$ we have 
$\sigma_i = \pm a_{x(i)} \sigma$ for some $x(i) \in W$ by
Lemma~\ref{lemma:locallyfinite}(iii) and 
Lemma~\ref{lemma:h0paths}(ii).  
Since $\sigma_i$ is the source of $C_i$, 
we have $a_s \sigma_i \ne - \sigma_i$ for $s \in J$, so 
$a_s a_{x(i)} \ne -a_{x_(i)}$, and thus  $sx(i) > x(i)$, for all $s \in J$.  Hence
$x(i)$ is in the set of distinguished right coset representatives 
$X_J = \setof{w \in W}{s w > w \text{ for }s \in J}$ of $W_J$ in $W$.  
Since $\sigma_i \ne \sigma_j$ when
$i \ne j$ are in $I$, 
$i \mapsto x(i)$ is an injection from $I$ into $X_J$.
\end{proof}

%%%%%%%%%%%%%%%%%%%%%%%%%%%%%%%%%%%%%%%%%%%%

\begin{example}
\label{example:regular}
Let $(W,S)$ be a Coxeter system.  
Let $\Gamma$ be the $W$-digraph defined by
$\vertices(\Gamma) = W$ and
$\solidedge{$x$}{$y$}{$s$}\in \edges(\Gamma)$ if and
only if $x < sx = y$ for $x, y \in W$, $s \in S$. 
Then the $H$-module $M(\Gamma)$ 
afforded by $\Gamma$ is isomorphic to 
 the left regular module $H$.  
Note that $\Gamma$ is connected since if 
$w = s_k s_{k-1} \cdots s_1$ is a reduced expression for
$w \in W$ and $x_j = s_j s_{j-1} \cdots s_1$ 
for $0 \le j \le k$
(with $x_0 = e$), then
\begin{equation*}
%%%%%%%%%%%%%%%%%%%%%%%%%%%%%%%%%%%%%%%%%%%%
\begin{tikzpicture}
[node distance=0.8cm,%
pre/.style={<-,shorten <=1pt,>=angle 45},%
post/.style={->,shorten >=1pt,>=angle 45}];%
rectangle/.style={inner sep=0pt,minimum size=5mm}];%
\node[rectangle]	(node0)						{${\strut}x_{0}$};
\node[rectangle]	(node1) [right=of node0]			{${\strut}x_{1}$}
	edge [pre]	node[yshift=3mm]	{$s_1$}		(node0);
\node[rectangle]	(node3) [right=of node1]	{${\strut}\cdots$}
	edge [pre]		node[yshift=3mm]	{$s_2$}	(node1);
\node[rectangle]	(node4) [right=of node3]	{${\strut}x_{k-1}$}
	edge [pre]		node[yshift=3mm]	{$s_{k-1}$}	(node3);
\node[rectangle]	(node5) [right=of node4]				{${\strut}x_{k}$}
	edge [pre]			node[yshift=3mm]	{$s_{k}$}	(node4);
\end{tikzpicture}
%%%%%%%%%%%%%%%%%%%%%%%%%%%%%%%%%%%%%%%%%%%%
%%%%%%%%%%%%%%%%%%%%%%%%%%%%%%%%%%%%%%%%%%%%
\end{equation*}
is a directed path from $e$ to $w$ in $\Gamma$.
When $(W,S)$ is finite, this
example shows that the bound in Corollary~\ref{corollary:bound} is
always attained. 
\end{example}

%%%%%%%%%%%%%%%%%%%%%%%%%%%%%%%%%%%%%%%%%%%%

\begin{proof}[Proof of Theorem~\ref{theorem:equallengths}]
Suppose $n(s,t) < \infinity$ for all $s, t \in S$ and $\Gamma$ is a
connected $W$-digraph with source $\sigma$.  
(The case in which $\Gamma$ has a sink
follows by applying the same reasoning to $\Gammarev$.)
Let $\pi_1$, $\pi_2$ be two directed paths in $\Gamma$ from
$\alpha$ to $\beta$.  Let $\rho$ be some directed path from 
$\sigma$ to $\alpha$: such a path exists
by Lemma~\ref{lemma:locallyfinite} (iii).  By 
Lemma~\ref{lemma:locallyfinite} (i), the directed paths $\rho \pi_1$ and
$\rho \pi_2$ from $\sigma$ to $\beta$ are in the same
$\equiv$ equivalence class, and thus have the same number of
edges.   Hence $\pi_1$ and $\pi_2$ have the same number of edges.
\end{proof}

%%%%%%%%%%%%%%%%%%%%%%%%%%%%%%%%%%%%%%%%%%%%
%%%%%%%%%%%%%%%%%%%%%%%%%%%%%%%%%%%%%%%%%%%%

\section{The proof of Theorem~\ref{theorem:charactervalues}}
\label{section:charactervalues}

Let $\sigma$ be an automorphism of $\rationals(u)$, and 
let $M$ be a vector space over $\rationals(u)$.  Let ${}^\sigma M$ be
the vector space over $\rationals(u)$ that has the same additive
group as $M$ and scalar multiplication 
$(\alpha,v) \mapsto \alpha *_\sigma v$ given by 
\[
\alpha *_\sigma v = ({}^\sigma \alpha)  v,
\]
where the scalar multiplication on the right hand side
is that of $M$.  
It is clear that if $Y \subseteq M$, then $Y$ is a 
basis (subspace) of $M$ if and only if $Y$ is
a basis (subspace, respectively) of ${}^\sigma M$.  
Moreover, 
$\gl(M) = \gl({}^\sigma M)$ since if
$\varphi : M \rightarrow M$ is an additive mapping, then   
\[
\varphi(\alpha *_\sigma v) = \alpha *_\sigma \varphi(v)
 \iff
\varphi ( ({}^\sigma \alpha) v ) = ({}^\sigma \alpha) \varphi(v)
\]
for 
$\alpha\in \rationals(u)$, $v\in M$.

%%%%%%%%%%%%%%%%%%%%%%%%%%%%%%%%%%%%%%%%%%%%

\begin{lemma}
\label{lemma:revequiv1}
Let $(W,S)$ be a Coxeter system, and let
$M=M(\Gamma)$ be the $H$-module afforded by the $W$-digraph $\Gamma$.
Let $\sigma$ be the automorphism of $\rationals(u)$ determined
by ${}^\sigma u = -1/u$.  
For $s\in S$, let $\tau_s \in \gl(M)$ be the operator 
$v \mapsto T_s v$.  
Then $T_s \mapsto \tau_s^{-1}\in \gl(^\sigma M)$ extends to a 
representation $H \rightarrow \gl({}^\sigma M)$.  Moreover,
as a basis for the $H$-module ${}^\sigma M$, 
$X = \vertices(\Gamma)$ supports the 
$W$-digraph $\Gammarev$.
\end{lemma}

\begin{proof}
Let $s \in S$.  Since $(\tau_s - u^2)(\tau_s + 1) = 0$ in $\gl(M)$,
we have $(\tau_s - u^{-2})(\tau_s +1) = 0$ in $\gl({}^\sigma M)$, and
thus
$(\tau_s^{-1} - u^2)(\tau_s^{-1} + 1) = 0$ in $\gl({}^\sigma M)$.
Also, 
if $s,t\in S$ and $1< n(s,t) < \infinity$,
then
\begin{equation*}
\overbrace{\tau_s^{-1} \tau_t^{-1} \cdots }^{n(s,t)}
=
(\overbrace{\cdots \tau_t \tau_s}^{n(s,t)})^{-1}
=
(\overbrace{\cdots \tau_s \tau_t}^{n(s,t)})^{-1}
=\overbrace{\tau_t^{-1} \tau_s^{-1} \cdots }^{n(s,t)}.
\end{equation*}
Therefore $T_s \mapsto \tau_s^{-1}$ extends to a 
representation $H \rightarrow \gl({}^\sigma M)$.  

Now suppose 
$\alpha, \beta\in X$, $s\in S$.
If \solidedge{$\alpha$}{$\beta$}{$s$} is an edge of $\Gamma$, then
one checks that 
\begin{equation*}
\tau_s^{-1} (\alpha)
=
(u^2-1) *_\sigma \alpha + u^2 *_\sigma \beta
\quad\text{and}\quad
\tau_s^{-1} (\beta)
=  \alpha
\end{equation*}
in ${}^\sigma M$. On the other hand,
if \dashededge{$\alpha$}{$\beta$}{$s$} is an edge of $\Gamma$, 
then
\begin{equation*}
\tau_s^{-1} (\alpha) 
=
(u^2-u-1) *_\sigma \alpha + (u^2-u) *_\sigma \beta
\quad\text{and}\quad
\tau_s^{-1} (\beta) 
=
(u + 1 ) *_\sigma \alpha + u *_\sigma \beta
\end{equation*}
in ${}^\sigma M$.  
These relations 
show that the basis $X$ for 
${}^\sigma M$ supports the $W$-digraph
$\Gammarev$, so the proof is complete.
\end{proof}

%%%%%%%%%%%%%%%%%%%%%%%%%%%%%%%%%%%%%%%%%%%%

%%%%%%%%%%%%%%%%%%%%%%%%%%%%%%%%%%%%%%%%%%%%

For a matrix $A$ over $\rationals(u)$, denote
by ${}^\sigma \!A$ the matrix obtained by applying
the automorphism $\sigma$ of $\rationals(u)$ to
each entry of $A$. 

\begin{corollary}
\label{corollary:revequiv1}
Suppose $\Gamma$ is a $W$-digraph, 
$X = \vertices(\Gamma)$ is finite, and $\sigma$ is the
automorphism of $\rationals(u)$ determined by
${}^\sigma u = -1/u$.  Let 
$\rho$ and
 $\rho_{\text{rev}}$ 
 be the matrix representations relative to the
 basis $X$ for
 the actions of $H$ on $M=M(\Gamma)$ and ${}^\sigma\! M$
 according to the $W$-digraphs $\Gamma$ and
 $\Gammarev$, respectively.  
Then
\begin{equation}
\label{eq:reveqn1}
 \rho_{\text{rev}}(T_w)
= {}^\sigma \rho(T_{w^{-1}}^{-1})
\end{equation}
for $w \in W$.
\end{corollary}

\begin{proof}
From the proof of Lemma~\ref{lemma:revequiv1}, we
have
$ \rho_{\text{rev}}(T_s) = {}^\sigma\rho(T_{s}^{-1})$
for $s \in S$.  The assertion follows since
if $w \in W$ has reduced expression $w = s_1 \cdots s_k$,
then $T_w = T_{s_1} T_{s_2} \cdots T_{s_k}$ and 
$T_{w^{-1}}^{-1} =
T_{s_1}^{-1} T_{s_2}^{-1} \cdots T_{s_k}^{-1}$. 
\end{proof}

%%%%%%%%%%%%%%%%%%%%%%%%%%%%%%%%%%%%%%%%%%%%

Next assume  
$n(s,t) < \infinity$ for $s, t \in S$, $\Gamma$ is an 
 acyclic $W$-digraph, and $\vertices(\Gamma)$ is finite.  
For $\alpha \in X = \vertices(\Gamma)$, let
$\sigma_\alpha$ be the source in the connected component
of $\Gamma$ containing $\alpha$, and let $\mu(\alpha)$
be the number of edges in a directed path from
$\sigma_\alpha$ to $\alpha$.
(Thus $\mu(\alpha)$ is well-defined
by Lemma~\ref{lemma:locallyfinite}(i).)
Put $\varepsilon_\alpha = (-1)^{\mu(\alpha)}$ for $\alpha\in X$,
and define 
$X^\prime = \setof{\varepsilon_\alpha \alpha}{\alpha \in X}$.
Let $\rho^\prime$ be the matrix representation 
afforded by $M(\Gamma)$ with basis $X^\prime$, and
let $\rho_{\text{rev}}$ be the matrix representation
corresponding to $M(\Gammarev)$ with basis $X$.

%%%%%%%%%%%%%%%%%%%%%%%%%%%%%%%%%%%%%%%%%%%%

\begin{lemma}
\label{lemma:revequiv2}
If $n(s,t) < \infinity$ for $s, t \in S$, $\Gamma$ is an 
acyclic $W$-digraph, $\vertices(\Gamma)$ is finite, 
and $\rho_{\text{rev}}$ and 
$\rho^\prime$ are defined as above, then
\begin{equation}
\label{eq:reveqn2}
\rho_{\text{rev}}(T_w)
=
\varepsilon_w u_w \rho^\prime( T_w^{-1})^{T}
\qquad
\text{for $w \in W$.}
\end{equation}
\end{lemma}

\begin{proof}
Let $s \in S$.  
Suppose 
$\solidedge{$\alpha$}{$\beta$}{$s$} \in \edges(\Gamma)$, 
so also
$\solidedge{$\beta$}{$\alpha$}{$s$} \in \edges(\Gammarev)$. 
Thus in $M(\Gammarev)$ we have
\begin{equation*}
T_s \alpha
=  (u^2-1) \alpha + u^2 \beta 
\qquad\text{and}\qquad
T_s \beta
= \alpha ,
\end{equation*}
so the matrix of $T_s$ acting on the subspace
with  basis $\set{\alpha,\beta}$ is
\[
\begin{pmatrix}  
u^2-1 & 1 \\
u^2 & 0
\end{pmatrix} .
\]
On the other hand,  $\varepsilon_\beta = - \varepsilon_\alpha$ 
and $u_s T_s^{-1} = T_s - (u^2-1)$, so 
in $M(\Gamma)$ we have
\[
\varepsilon_s u_s T_s^{-1} \varepsilon_\alpha \alpha
=
- \varepsilon_{\alpha} \left(T_s - (u^2-1)\right) \alpha
=
- \varepsilon_{\alpha} \left(\beta - (u^2-1) \alpha \right)
=
(u^2-1)\varepsilon_\alpha \alpha + \varepsilon_\beta \beta
\]
and
\begin{equation*}
\begin{split}
\varepsilon_s u_s T_s^{-1}  \varepsilon_\beta \beta
& =
- \varepsilon_{\beta} \left(T_s - (u^2-1)\right) \beta \\
& = 
- \varepsilon_{\beta} \left( (u^2-1)\beta + u^2\alpha - (u^2-1)\beta\right) \\
& =
u^2 \varepsilon_\alpha \alpha, 
\end{split}
\end{equation*}
so the matrix of $\varepsilon_s u_s T_s^{-1}$ acting on the
subspace with  
basis $\set{\varepsilon_\alpha \alpha, \varepsilon_\beta \beta}$ is
\[
\begin{pmatrix}
u^2-1 & u^2 \\
1 & 0 
\end{pmatrix} .
\]

Now suppose that 
$\dashededge{$\alpha$}{$\beta$}{$s$} \in \edges(\Gamma)$.
Then in $M(\Gammarev)$ we have
\[
T_s \alpha
=  (u^2-u-1) \alpha + (u^2-u) \beta 
\]
and
\[
T_s  \beta
= (u+1) \alpha + u \beta,
\]
so the matrix of $T_s$ acting on the subspace
with  basis $\set{\alpha,\beta}$ is
\[
\begin{pmatrix}  
u^2-u-1 & u+1 \\
u^2 -u & u
\end{pmatrix} .
\]
In $M(\Gamma)$ we have
\begin{equation*}
\begin{split}
\varepsilon_s u_s T_s^{-1} \varepsilon_\alpha \alpha
& =
- \varepsilon_{\alpha} \left(T_s - (u^2-1)\right) \alpha \\
& =
- \varepsilon_{\alpha} \left(u \alpha + (u+1) \beta - (u^2-1) \alpha \right) \\
& =
(u^2-u-1)\varepsilon_\alpha \alpha + (u+1)\varepsilon_\beta \beta
\end{split}
\end{equation*}
and
\begin{equation*}
\begin{split}
\varepsilon_s u_s T_s^{-1}  \varepsilon_\beta \beta
& =
- \varepsilon_{\beta} \left(T_s - (u^2-1)\right) \beta \\
& = 
- \varepsilon_{\beta} \left( (u^2-u-1)\beta + (u^2-u)\alpha - (u^2-1)\beta\right) \\
& =
(u^2-u) \varepsilon_\alpha \alpha + \varepsilon_\beta \beta, 
\end{split}
\end{equation*}
so the matrix of $\varepsilon_s u_s T_s^{-1}$ acting on the
subspace with  
basis $\set{\varepsilon_\alpha \alpha, \varepsilon_\beta \beta}$ is
\[
\begin{pmatrix}
u^2-u-1 & u^2-u \\
u+1 & u 
\end{pmatrix} .
\]

To summarize,  
\eqref{eq:reveqn2} holds when $w = s \in S$.
The general case follows since if 
$w \in W$ has reduced expression
$w = s_1 s_2 \cdots s_k$, then
$T_w = T_{s_1}T_{s_2} \cdots T_{s_k}$.
\end{proof}

%%%%%%%%%%%%%%%%%%%%%%%%%%%%%%%%%%%%%%%%%%%%

\begin{proof}[Proof of Theorem~\ref{theorem:charactervalues}]
Part (i) of the theorem follows by taking 
traces in \eqref{eq:reveqn1}.
Since  
$\chi_{\Gamma}$ coincides with the character afforded by the matrix
representation $\rho^\prime$,
part (ii) of the theorem follows by taking traces
in \eqref{eq:reveqn2}.
\end{proof}

%%%%%%%%%%%%%%%%%%%%%%%%%%%%%%%%%%%%%%%%%%%%
%%%%%%%%%%%%%%%%%%%%%%%%%%%%%%%%%%%%%%%%%%%%

\section{The proof of Theorem~\ref{theorem:wgraphacyclic}}
\label{section:wgraphacyclic}

\begin{proof}
Suppose $(W_J,J)$ is finite for proper subsets $J$ of $S$ and
$\Gamma$ is a finite, connected $W$-digraph.  
Suppose further that $M(\Gamma)$ is isomorphic to the
module $M(\Psi)$ afforded by a $W$-graph
 $\Psi$ for $(W,S)$ (in the sense of \cite{kazhdanlusztig}),
and that $\Gamma$ is not acyclic.  
For $x$ in the set of vertices $\vertices(\Psi)$ of $\Psi$, let
$I_x \subseteq S$ be the associated set of generators.
For $\beta \in \vertices(\Gamma)$, 
let $\In(\beta)$ be the set of $s\in S$ such that $\Gamma$ 
as an edge of the form 
\solidedge{$\alpha$}{$\beta$}{$s$} or
\dashededge{$\alpha$}{$\beta$}{$s$}
for some $\alpha \in \vertices(\Gamma)$.
Let $\chi_{\Gamma} = \chi_{\Psi}$ 
be the character of $H$ afforded by
$M(\Gamma)$ or $M(\Psi)$.  
Put
\[
N_\Gamma(J) = \abs{ \setof{\beta \in \vertices(\Gamma)} {\In(\beta) = J} },
\qquad
N_\Psi(J) = \abs{ \setof{x \in \vertices(\Psi) } {I_x = J}}
\]
for $J \subseteq S$.
Since
$M(\Gamma)_{\ind} \cong M(\Psi)_{\ind}$ is one-dimensional
by Theorem~\ref{theorem:linearcharmults}(i), 
we must have $N_\Psi(\emptyset) > 0$. 
Also, since $\Gamma$ is not acyclic,
$M(\Gamma)_{\sgn} = \set{0} = M(\Psi)_{\sgn}$ by
Theorem~\ref{theorem:linearcharmults}(ii), and therefore
$N_{\Psi}(S) = 0$.  Further, $\Gamma$ has no sink by
Theorem~\ref{theorem:acyclic}, and 
so also $N_{\Gamma}(S) = 0$.  
For $w \in W$, let $J(w)$ be the minimal $J \subseteq S$
such that $w \in W_J$.  Then
\[
\eval{\chi_{\Psi}(T_w)}{u}{0}
=
\varepsilon_w \abs{ \setof{x \in \vertices(\Psi) }{ J(w) \subseteq I_x } }.
\]
Since $I_x \ne S$ for $x \in \vertices(\Psi)$, we have
\begin{equation*}
\begin{split}
0 
< N_{\Psi} (\emptyset)
& =
\sum_{x \in \vertices(\Psi)}
    \sum_{w \in W_{I_x}} \varepsilon_w 
 =
\sum_{w \in W, J(w) \ne S}
  \varepsilon_w \abs{ \setof{x \in \vertices(\Psi)} {J(w) \subseteq I_x} } \cr
& =  
\sum_{w \in W, J(w) \ne S} 
    \eval{\chi_\Psi(T_w)}{u}{0}, \cr
\end{split}
\end{equation*}
with the sums finite by assumption. 
On the other hand, if $J(w) \ne S$, then 
$\Gamma_{J(w)}$ is acyclic by Theorem~\ref{theorem:acyclic}(iii),
so if
$\vertices(\Gamma)$ is ordered in a way consistent with
directed paths in $\Gamma_{J(w)}$, then the matrix representing
$T_w$ acting on $M(\Gamma)$, when evaluated at
$u=0$, is triangular.  Moreover, the nonzero diagonal entries 
of this matrix are all equal to
$\varepsilon_w$, occurring in positions
corresponding to those $\beta \in \vertices(\Gamma)$
such that $J(w) \subseteq \In(\beta)$.
Since $\In(\beta) \ne S$ for $\beta \in \vertices(\Gamma)$
and $\chi_\Gamma = \chi_\Psi$, it follows that
\begin{equation*}
\begin{split}
0 
<
\sum_{w \in W, J(w) \ne S}
   \eval{\chi_{\Gamma}(T_w)}{u}{0} 
& =
\sum_{w \in W, J(w) \ne S}
 \varepsilon_w \abs{ \setof{\beta \in \vertices(\Gamma)} {J(w) \subseteq \In(\beta)}} \cr
& =
\sum_{\beta \in \vertices(\Gamma)}
   \sum_{w \in W_{\In(\beta)}} \varepsilon_w  
 =
N_{\Gamma}(\emptyset), \cr
\end{split}
\end{equation*}
and so $\Gamma$ has a source.
Therefore 
$\Gamma$ is acyclic by Theorem~\ref{theorem:acyclic}.
\end{proof}

%%%%%%%%%%%%%%%%%%%%%%%%%%%%%%%%%%%%%%%%%%%%
%%%%%%%%%%%%%%%%%%%%%%%%%%%%%%%%%%%%%%%%%%%%

\section{Additional Examples}
\label{section:additionalexamples}

Let $(W,S)$ be a Coxeter system,  let
$\gamma\mapsto \overline{\gamma}$ be the
automorphism
of $\rationals(u)$ determined by $\overline{u} = u^{-1}$,
and let $h \mapsto \overline{h}$ be the
ring automorphism 
$
\sum_{w \in W} \gamma_w T_w 
\mapsto
\sum_{w\in W} \overline{\gamma_w} \, T_{w^{-1}}^{-1}
$
of $H$.
Following Lusztig \cite{lusztigbarop}, 
define a {\it bar operator} on an $H$-module $M$
to be an additive bijection 
$\varphi : M \rightarrow M$ such that
\begin{equation}
\label{eq:bar}
\varphi(h v) = \overline{h} \, \varphi(v)
\qquad
\text{for $h \in H$, $v \in M$.}
\end{equation}
Let $\Gamma_{*}$ be the $W$-digraph associated
with an involutory automorphism 
$w \mapsto w^*$ of $(W,S)$, as described before
Theorem~\ref{theorem:involutions}.
Lusztig has shown that
$M(\Gamma_{*})$ admits a unique bar operator 
that fixes the source of $\Gamma_{*}$
 (\cite{lusztigbarop}, Theorem~0.2).
It can be shown that if
$(W,S)$ is finite, then any $H$-module admits
a bar operator.  However, 
there need not be a bar operator if
$(W,S)$ is infinite, as the next example shows.

%%%%%%%%%%%%%%%%%%%%%%%%%%%%%%%%%%%%%%%%%%%%

\begin{example}
\label{example:affinea2}
With 
$W = W(\widetilde{A}_2) = \spanof{r,s,t}$,  let
$\Gamma$ be as in Figure~\ref{fig:cyclic3}, 
and  put $w = t s r$.
Suppose a bar operator 
$\varphi$ exists on $M(\Gamma)$.  Let $\alpha$ be the vertex
in the lower left corner of Figure~\ref{fig:cyclic3}, so
$T_{tsr} \alpha = T_t T_s T_r \alpha = \alpha$.  Then 
$\overline{T_{tsr}}\varphi(\alpha) = \varphi(\alpha)$, so
$\varphi(\alpha) = T_{rst} \varphi(\alpha)$ is a fixed
point of $T_{rst}$.  
However, one checks that
the characteristic polynomial of $T_{rst}$ acting on $M(\Gamma)$
is 
$(\lambda^2+1)(\lambda^2-u^6)(\lambda-u^6)^2$, 
so a contradiction is obtained.
\begin{figure}[ht]
\centering
%%%%%%%%%%%%%%%%%%%%%%%%%%%%%%%%%%%%%%%%%%%%
\begin{tikzpicture}
[node distance=0.8cm,%
pre/.style={<-,shorten <=2pt,>=angle 45},%
post/.style={->,shorten >=2pt,>=angle 45}];%
rectangle/.style={inner sep=0pt,minimum size=5mm}];%
\tikzstyle{mydot}=[circle,on grid,draw=black,fill=black,inner sep=1pt,minimum size=1mm];
\newcommand{\rone}{0.8}
\newcommand{\rtwo}{2.2}
\node[mydot]	(a1) at (\rone*cos{90},\rone*sin{90})					{};
\node[mydot]	(a2) at (\rone*cos{210},\rone*sin{210})				{};
\node[mydot]	(a3) at (\rone*cos{330},\rone*sin{330})				{};
\node[mydot]	(b1) at (\rtwo*cos{90},\rtwo*sin{90})					{};
\node[mydot]	(b2) at (\rtwo*cos{210},\rtwo*sin{210})				{};
\node[mydot]	(b3) at (\rtwo*cos{330},\rtwo*sin{330})				{};
\path (a1) edge[post]		node[xshift=2mm,yshift=1mm] 		{$s$}		(a3);
\path (a3) edge[post]		node[yshift=-2mm]				{$t$}		(a2);
\path (a2) edge[post]		node[xshift=-2mm,yshift=1mm]		{$r$}		(a1);
\path (a1) edge[post]		node[xshift=1.5mm,yshift=-0.5mm]	{$t$}		(b1);
\path (a2) edge[post]		node[xshift=-0.3mm,yshift=1.8mm]	{$s$}		(b2);
\path (a3) edge[post]		node[xshift=0.3mm,yshift=1.7mm]	{$r$}		(b3);
\path (b1) edge[post]		node[xshift=2mm,yshift=1mm] 		{$s$}		(b3);
\path (b3) edge[post]		node[yshift=-2mm]				{$t$}		(b2);
\path (b2) edge[post]		node[xshift=-2mm,yshift=1mm]		{$r$}		(b1);
\end{tikzpicture}
%%%%%%%%%%%%%%%%%%%%%%%%%%%%%%%%%%%%%%%%%%%%
%%%%%%%%%%%%%%%%%%%%%%%%%%%%%%%%%%%%%%%%%%%%
\caption{$W$-digraph for Example~\ref{example:affinea2}}
\label{fig:cyclic3}
\end{figure}
Also, $\Gamma$  
provides an example in which 
the equation of Theorem~\ref{theorem:charactervalues}(ii) fails: 
with 
  $y = w^{-1} = r s t$,  one checks that
$\chi_{\Gammarev}(T_y) = 
{}^{\sigma} \chi_{\Gamma} (T_{y^{-1}}^{-1} ) = 2$
and
$\varepsilon_{y} u_{y} \chi_{\Gamma} \left(  T_{y}^{-1} \right) = - 2$.  
Moreover, $M(\Gamma)$ does not afford
a  $W$-graph by Theorem~\ref{theorem:wgraphacyclic}.
\end{example}

%%%%%%%%%%%%%%%%%%%%%%%%%%%%%%%%%%%%%%%%%%%%

Even
if $(W,S)$ is finite and $\Gamma$ is connected, there may not exist
a bar operator on $M(\Gamma)$ that fixes the source of 
$\Gamma$, as the next example shows.

\begin{example}
\label{example:nosourcefixingbarop}
Let $W = W(B_3) = \spanof{r,s,t}$, with
$n(r,s) = 3$, $n(r,t) = 2$, $n(s,t) = 4$.  Let $\Gamma$ be
the $W$-digraph of Figure~\ref{fig:b3dgr21}.
\begin{figure}[ht]
\centering
%%%%%%%%%%%%%%%%%%%%%%%%%%%%%%%%%%%%%%%%%%%%
\begin{tikzpicture}
[node distance=0.6cm,%
pre/.style={<-,shorten <=1pt,>=angle 45},%
post/.style={->,shorten >=1pt,>=angle 45}];%
rectangle/.style={inner sep=0pt,minimum size=5mm}];%
\newcommand{\bigvshift}{-25mm}
\newcommand{\smallshift}{2mm}
\tikzstyle{mydot}=[circle,on grid,draw=black,fill=black,inner sep=0pt,minimum size=1mm];
\node[rectangle] (v0)									{${\strut}v_0$};
\node[rectangle] (v2) [right=of v0,xshift=\smallshift]				{${\strut}v_2$}
	edge[pre]				node[yshift=2.5mm]		{$s$}		(v0);
\node[rectangle] (v4) [right=of v2,xshift=\smallshift]				{${\strut}v_4$}
	edge[pre,bend right=20]	node[yshift=2.5mm]		{$r$}			(v2)
	edge[pre,bend left=20]	node[yshift=-2.5mm]		{$t$}			(v2);
\node[rectangle] (v6) [right=of v4,xshift=\smallshift]				{${\strut}v_6$}
	edge[pre]				node[yshift=2.5mm]		{$s$}		(v4);	
\node[rectangle] (v1) [below=of v0,yshift=\bigvshift]				{${\strut}v_1$}
	edge[pre]				node[xshift=-2.5mm]		{$t$}			(v0);	
\node[rectangle] (v3) [below=of v2,yshift=\bigvshift]				{${\strut}v_3$}
	edge[pre]				node[yshift=2.5mm]		{$s$}		(v1);
\node[rectangle] (v5) [below=of v4,yshift=\bigvshift]				{${\strut}v_5$}
	edge[pre,bend right=20]	node[yshift=2.5mm]		{$r$}			(v3)
	edge[pre,bend left=20]	node[yshift=-2.5mm]		{$t$}			(v3);
\node[rectangle] (v7) [below=of v6,yshift=\bigvshift]				{${\strut}v_7$}
	edge[pre]				node[yshift=2.5mm]		{$s$}		(v5)
	edge[pre]				node[xshift=2.5mm]		{$t$}			(v6);
\node[rectangle] (v8) [below=of v2,yshift=\smallshift]			{${\strut}v_8$}
	edge[pre]			node[xshift=-2mm,yshift=-2mm] {$r$}		(v0);
\node[rectangle] (v9) [above=of v3,,yshift=-\smallshift]			{${\strut}v_9$}
	edge[pre]				node[xshift=-2.5mm] 	{$t$}			(v8)
	edge[pre]			node[xshift=-2mm,yshift=2mm] {$r$}		(v1);
\node[rectangle] (v10) [below=of v4,,yshift=\smallshift]			{${\strut}v_{10}$}
	edge[pre]				node[yshift=2.5mm]		{$s$}		(v8)
	edge[post]		node[xshift=2mm,yshift=-2mm] {$r$}		(v6);
\node[rectangle] (v11) [above=of v5,yshift=-\smallshift]			{${\strut}v_{11}$}
	edge	[pre]				node[yshift=-2.5mm]		{$s$}		(v9)
	edge[pre]				node[xshift=2.5mm]		{$t$}			(v10)
	edge[post]		node[xshift=2mm,yshift=2mm] {$r$}			(v7);
\end{tikzpicture}
%%%%%%%%%%%%%%%%%%%%%%%%%%%%%%%%%%%%%%%%%%%%
%%%%%%%%%%%%%%%%%%%%%%%%%%%%%%%%%%%%%%%%%%%%
\caption{Digraph for Example~\ref{example:nosourcefixingbarop}}
\label{fig:b3dgr21}
\end{figure}
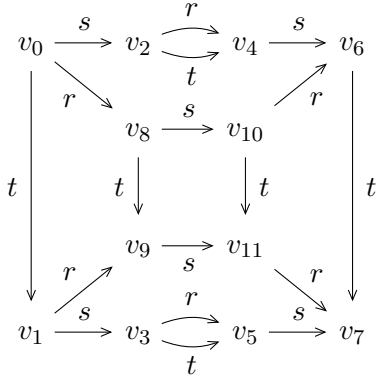
Suppose $M=M(\Gamma)$ admits a 
source-fixing bar operator $\varphi : M \rightarrow M$.
Since $v_0$ is the source
of $\Gamma$ and $v_4 = T_r T_s v_0$, we have 
\begin{equation}
\begin{split}
\varphi(v_4)
& =
\overline{T_r} \, \overline{T_s} \, v_0 
=
u^{-4} (T_r - (u^2-1)) (T_s - (u^2-1)) v_0 \cr
& = 
u^{-4}
\left( v_4 - (u^2-1) v_2 - (u^2-1)v_8 + (u^2-1)^2 v_0 \right). \cr
\end{split}
\label{eq:b3dgr21a}
\end{equation}
On the other hand, $v_4 = T_t T_s v_0$, so we also have
\begin{equation}
\begin{split}
\varphi(v_4)
& =
\overline{T_t} \, \overline{T_s} \, v_0 
=
u^{-4} (T_t - (u^2-1)) (T_s - (u^2-1)) v_0 \cr
& =
u^{-4}
\left( v_4 - (u^2-1) v_2 - (u^2-1)v_1 + (u^2-1)^2 v_0 \right). \cr
\end{split}
\label{eq:b3dgr21b}
\end{equation}
Since \eqref{eq:b3dgr21a} and
\eqref{eq:b3dgr21b} cannot simultaneously hold,
a contradiction is reached. Thus $M$ does not admit
a source-fixing bar operator.
\end{example}

%%%%%%%%%%%%%%%%%%%%%%%%%%%%%%%%%%%%%%%%%%%%

Let $(W,S)$ be finite, and let $\Gamma$ be a finite $W$-digraph.  
By Theorem~\ref{theorem:charactervalues}  we have
$
\eval{\chi_{\Gamma}}{u}{-1} 
=
\sgn_W \cdot \eval{\chi_{\Gamma}}{ u}{1}.
$
Thus if $(W,S)$ has no connected components with
exceptional characters in the sense of 
Gyoja \cite{gyojawgraph}, then 
$\eval{\chi_{\Gamma}}{u}{1} =
\eval{\chi_{\Gamma}}{u}{-1}$ is self-associated, that is,
$\eval{\chi_{\Gamma}}{u}{1} = \sgn_W \cdot \eval{\chi_{\Gamma}}{u}{1}$.
In particular, if $(W,S)$ has no component of
type $H_3$, $H_4$, $E_7$, or $E_8$, 
then $\eval{\chi_{\Gamma}}{u}{1}$ is self-associated.
Our final example shows that if 
$(W,S)$ has an exceptional character, then $\eval{\chi_{\Gamma}}{u}{1}$
need not be self-associated.

\begin{example}
Let $W = W(H_3) = \spanof{r,s,t}$ with
$n(r,s) = 3$, $n(s,t) = 5$, $n(r,t)=2$.  The $W$-digraph
$\Gamma$ of Figure~\ref{fig:h3dgr8}
affords the non-self-associated character
$\eval{\chi_{\Gamma}}{u}{1} = 1_W + \sgn_W + \chi_{4^\prime}$, 
where $\chi_{4^\prime}$ is the irreducible character
of degree 4 with value $-4$ at the longest element of $W$.  
Then 
$\eval{\chi_{\Gammarev}}{u}{1}
= \sgn_W \cdot \eval{\chi_{\Gamma}}{u}{1} 
\ne \eval{\chi_{\Gamma}}{u}{1}$.  
\begin{figure}[ht]
\[
%%%%%%%%%%%%%%%%%%%%%%%%%%%%%%%%%%%%%%%%%%%%
\begin{tikzpicture}
[node distance=1.5cm,%
pre/.style={<-,shorten <=1pt,>=angle 45},%
post/.style={->,shorten >=1pt,>=angle 45}];%
rectangle/.style={inner sep=0pt,minimum size=5mm}];%
\tikzstyle{mydot}=[circle,on grid,draw=black,fill=black,inner sep=0pt,minimum size=1mm];
\node[mydot] (alpha1)									{ };
\node[mydot] (alpha2) [right=of alpha1,xshift=2mm]				{ }
	edge[pre,dashed]		node[yshift=2.5mm]	{$s$}		(alpha1);
\node[mydot] (beta1) [below=of alpha1]						{ }
	edge[pre,dashed,bend left=20]	    node[xshift=-2.5mm] {$r$}	(alpha1)
	edge[pre,dashed,bend right=20]  node[xshift=2.5mm] {$t$} 	(alpha1);
\node[mydot] (beta2)	[below=of alpha2]					{ }
	edge[pre]			node[xshift=2.5mm]			{$r$}		(alpha2)
	edge[pre]			node[yshift=2.5mm]			{$s$}	(beta1);
\node[mydot] (alpha3) [right=of alpha2,xshift=2mm]				{ }
	edge[pre]			node[yshift=2.5mm]			{$t$}		(alpha2);
\node[mydot] (beta3)	[below=of alpha3]					{ }
	edge[pre]				node[yshift=2.5mm]		{$t$}		(beta2)
	edge[pre,bend left=20]	node[xshift=-2.5mm]		{$r$}		(alpha3)
	edge[pre,bend right=20] node[xshift=2.5mm]		{$s$}	(alpha3);
\end{tikzpicture}
%%%%%%%%%%%%%%%%%%%%%%%%%%%%%%%%%%%%%%%%%%%%
%%%%%%%%%%%%%%%%%%%%%%%%%%%%%%%%%%%%%%%%%%%%
\]
\caption{Digraph for Example~\ref{example:h3dgr8}}
\label{fig:h3dgr8}
\end{figure}
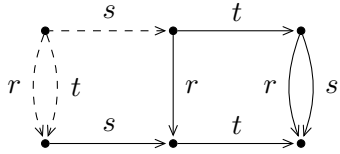
\label{example:h3dgr8}
\end{example}

%%%%%%%%%%%%%%%%%%%%%%%%%%%%%%%%%%%%%%%%%%%%
%%%%%%%%%%%%%%%%%%%%%%%%%%%%%%%%%%%%%%%%%%%%

\bibliographystyle{plain}
\bibliography{combined}

%%%%%%%%%%%%%%%%%%%%%%%%%%%%%%%%%%%%%%%%%%%%
%%%%%%%%%%%%%%%%%%%%%%%%%%%%%%%%%%%%%%%%%%%%

\vfil\eject
\enddocument
\bye
\bye